\documentclass[11pt, reqno]{amsart}

\usepackage[margin=1in]{geometry}
\usepackage{mathpazo}
\usepackage{amsmath, amssymb, amsthm}
\usepackage[foot]{amsaddr}
\usepackage[shortlabels]{enumitem}
\usepackage{thmtools}
\usepackage{tikz, tikz-cd}
\usepackage{mathtools}
\usepackage[labelfont=rm, subrefformat=parens]{subcaption}
\usepackage{fbox}

\usepackage{xcolor}
\definecolor{blue}{HTML}{5E81AC} 
\definecolor{slateblue}{HTML}{5E7A8B} 
\definecolor{darkbluegrey}{HTML}{3A4957}
\definecolor{orange}{HTML}{DE8638} 
\definecolor{red}{HTML}{A4243B} 
\definecolor{green}{HTML}{5D8B50} 

\usepackage{hyperref}
\hypersetup{
    colorlinks,
    linkcolor={darkbluegrey},
    citecolor={red},
    urlcolor={red}}

\usepackage[capitalize, noabbrev, nameinlink]{cleveref} 

\title{Interpreting the Ooguri-Vafa symplectic form à la Atiyah-Bott}
\author{Danny Nackan}
\address{Department of Mathematics, Yale University, New Haven, CT 06511, USA}
\email{danny.nackan@yale.edu}

\numberwithin{equation}{section}

\newtheorem{theorem}{Theorem}[section]
\newtheorem{lemma}[theorem]{Lemma}
\newtheorem{cor}[theorem]{Corollary}
\newtheorem{prop}[theorem]{Proposition}

\newtheorem{theoremalph}{Theorem}

\theoremstyle{definition}
\newtheorem{definition}[theorem]{Definition}
\newtheorem{notation}[theorem]{Notation}
\newtheorem{example}[theorem]{Example}
\newtheorem{remark}[theorem]{Remark}
\newtheorem{construction}[theorem]{Construction}
\newtheorem{definitionalph}{Definition}

\crefname{construction}{Construction}{Constructions}
\crefname{prop}{Proposition}{Propositions}

\newenvironment{subproof}
{\begin{proof}}{\end{proof}}

\newcommand{\R}{\mathbb{R}}
\newcommand{\Z}{\mathbb{Z}}

\newcommand{\C}{\mathbb{C}}

\newcommand{\CP}{\mathbb{CP}}
\renewcommand{\O}{\mathcal{O}}

\DeclareMathOperator{\GL}{GL}
\DeclareMathOperator{\SL}{SL}
\DeclareMathOperator{\SU}{SU}

\DeclareMathOperator{\id}{id}
\newcommand{\bbid}{\mathbb{1}}

\DeclareMathOperator{\End}{End}

\renewcommand{\Re}{\operatorname{Re}}
\renewcommand{\Im}{\operatorname{Im}}

\DeclareMathOperator{\tr}{tr}
\DeclareMathOperator{\rank}{rank}
\DeclareMathOperator{\diag}{diag}

\renewcommand{\epsilon}{\varepsilon}

\newcommand{\defeq}{\stackrel{\mathrm{def}}{=}}

\newcommand{\ds}{\displaystyle}

\newcommand{\ab}{\mathrm{ab}}
\renewcommand{\L}{\mathcal{L}}
\newcommand{\W}{\mathcal{W}} 
\DeclareMathOperator{\CG}{CG} 

\newcommand{\fr}{\mathrm{fr}}
\renewcommand{\H}{\mathcal{H}} 
\newcommand{\Hfr}{\H^\fr} 
\newcommand{\sfHfr}{\H^{\fr, \sf}} 
\newcommand{\uHfr}{\widehat{\Hfr}} 
\newcommand{\Xfr}{\mathfrak{X}^\fr} 
\newcommand{\sfXfr}{\mathfrak{X}^{\fr, \sf}} 
\newcommand{\A}{\mathcal{A}}
\newcommand{\Afr}{\A^\fr} 
\newcommand{\abAfr}{\A^{\fr, \ab}} 
\newcommand{\sfAfr}{\A^{\fr, \sf}} 
\newcommand{\sfabAfr}{\A^{\fr, \sf, \ab}} 
\newcommand{\M}{\mathcal{M}} 
\newcommand{\Mfr}{\M^\fr} 
\newcommand{\abMfr}{\M^{\fr, \ab}} 

\newcommand{\B}{\mathcal{B}}
\newcommand{\Bfr}{\B^\fr} 

\newcommand{\ov}{\mathrm{ov}}
\newcommand{\Mov}{\M^\ov}
\renewcommand{\sf}{\mathrm{sf}}
\newcommand{\inst}{\mathrm{inst}}
\newcommand{\Xe}{\mathcal{X}_e}
\newcommand{\Xm}{\mathcal{X}_m}

\newcommand{\shift}{\mathrm{shift}}
\newcommand{\Xms}{\mathcal{X}_m^\shift}

\newcommand{\reg}{\mathrm{reg}}
\newcommand{\glue}{\mathrm{glue}}

\DeclareMathOperator{\NAH}{NAH}
\newcommand{\dR}{\mathrm{dR}}
\newcommand{\Stokes}{\mathrm{Stokes}}
\newcommand{\Higgs}{\mathrm{Higgs}}
\newcommand{\AB}{\mathrm{AB}} 
\newcommand{\WKB}{\mathrm{WKB}}
\newcommand{\Hit}{\mathrm{Hit}}

\DeclareMathOperator{\Sect}{Sect}
\newcommand{\eSect}{\widehat{\Sect}}

\renewcommand{\top}{\mathrm{top}}
\renewcommand{\bot}{\mathrm{bot}}
\newcommand{\ins}{\mathrm{in}}
\newcommand{\out}{\mathrm{out}}

\DeclareMathOperator{\Hol}{Hol}
\DeclareMathOperator{\wind}{wind}

\DeclareMathOperator{\pdeg}{pdeg}

\newcommand{\D}{\mathbb{D}} 

\DeclareMathOperator{\shom}{\mathcal{H}\kern -.5pt \mathit{om}} 

\makeatletter
\renewcommand{\tocsection}[3]{%
  \indentlabel{\@ifnotempty{#2}{\bfseries\ignorespaces#1 #2\quad}}\bfseries#3}
\renewcommand{\tocsubsection}[3]{%
  \indentlabel{\@ifnotempty{#2}{\ignorespaces#1 #2\quad}}#3}

\newcommand\@dotsep{4.5}
\def\@tocline#1#2#3#4#5#6#7{\relax
  \ifnum #1>\c@tocdepth 
  \else
    \par \addpenalty\@secpenalty\addvspace{#2}%
    \begingroup \hyphenpenalty\@M
    \@ifempty{#4}{%
      \@tempdima\csname r@tocindent\number#1\endcsname\relax
    }{%
      \@tempdima#4\relax
    }%
    \parindent\z@ \leftskip#3\relax \advance\leftskip\@tempdima\relax
    \rightskip\@pnumwidth plus1em \parfillskip-\@pnumwidth
    #5\leavevmode\hskip-\@tempdima{#6}\nobreak
    \leaders\hbox{$\m@th\mkern \@dotsep mu\hbox{.}\mkern \@dotsep mu$}\hfill
    \nobreak
    \hbox to\@pnumwidth{\@tocpagenum{\ifnum#1=1\bfseries\fi#7}}\par 
    \nobreak
    \endgroup
  \fi}
\def\l@subsection{\@tocline{2}{0pt}{2.5pc}{5pc}{}}
\def\l@subsubsection{\@tocline{3}{0pt}{5pc}{5pc}{}}
\makeatother

\begin{document}

\begin{abstract}
Gaiotto, Moore, and Neitzke predicted that the hyperkähler Ooguri-Vafa space $\Mov$ should provide a local model for Hitchin moduli spaces near the discriminant locus.
To this end, Tulli identified $\Mov$ with a certain space of framed Higgs bundles with an irregular singularity.
We extend this result by identifying the Ooguri-Vafa holomorphic symplectic form with a regularized version of the Atiyah-Bott form on the associated space of framed connections.
We also prove the analogous statement for the corresponding semiflat forms. 
Finally, restricting to the Hitchin section, we identify a regularized version of Hitchin's $L^2$-metric with the Ooguri-Vafa metric.
\end{abstract}

\maketitle
\thispagestyle{empty}

\tableofcontents

\section{Introduction} 
\label{sec:introduction}

Moduli spaces of Higgs bundles carry an incredibly rich structure, due in large part to the presence of Hitchin's hyperkähler metric $g_{L^2}$ \cite{Hitchin:1987}.
Although $g_{L^2}$ is naturally defined, it is highly transcendental, involving solutions to Hitchin's equation (a nonlinear PDE). 

A precise conjectural picture of $g_{L^2}$ was described in the work of Gaiotto, Moore, and Neitzke \cite{Gaiotto:2010,Gaiotto:2013a}.
In particular---and of main interest to us in this paper---the local picture near the discriminant locus of the Hitchin base was conjectured to be described by the hyperkähler \emph{Ooguri-Vafa metric} (originally defined in \cite{Ooguri:1996}).

The data of a hyperkähler metric can equivalently be formulated in terms of a \emph{twistor family} of holomorphic symplectic forms $\Omega_\zeta$ (see \cite{Hitchin:1992b,Hitchin:1987b}, or \cite[Section 3]{Gaiotto:2010} for a summary).

\begin{itemize}

    \item In the case of a moduli space $\M^\Hit$ of Higgs bundles, say for now on a compact Riemann surface $C$ (see \cref{sec:hyperkahler-metric} for more background), the form $\Omega_\zeta^\Hit$ corresponding to $g_{L^2}$ can be studied in terms of flat connections on $C$.\footnote{We omit discussion of the relevant stability conditions here; all of the Higgs bundles we consider later will be stable.}

    Given a Higgs bundle $(E, \theta)$ and a harmonic metric $h$ solving Hitchin's equation, there is an associated family of flat connections
    \begin{equation} \label{eq:INTRO-NAH}
        \nabla_\zeta  =  \zeta^{-1} \theta + D_h + \zeta \theta^{\dagger_h}, \quad \zeta \in \C^*
    \end{equation}
    (where $D_h$ denotes the Chern connection).
    For each $\zeta \in \C^*$, the \emph{nonabelian Hodge correspondence} $\NAH_\zeta: (E, \theta) \mapsto (E, \nabla_\zeta)$ identifies $\M^\Hit$ (in complex structure $I_\zeta$) with the de~Rham moduli space $\M^\dR$ of flat $\SL(2, \C)$-connections on $C$ (in its natural complex structure).
    Furthermore, the form $\Omega_\zeta^\Hit$ on $\M^\Hit$ is identified with the holomorphic symplectic \emph{Atiyah-Bott form} \cite{Atiyah:1983}
    \begin{equation} \label{eq:atiyah-bott}
        \Omega^\AB (\dot{\nabla}_1, \dot{\nabla}_2) = \int_C \tr (\dot{\nabla}_1 \wedge \dot{\nabla}_2)
    \end{equation}
    on $\M^\dR$.

\item In the case of the Ooguri-Vafa space $\Mov$, the holomorphic symplectic form can be written
    \begin{equation} \label{eq:INTRO-ov-twistor-coords}
        \Omega^\ov_\zeta = -\frac{1}{4 \pi^2} d \log \Xe(\zeta) \wedge d \log \Xm(\zeta), \quad \zeta \in \C^*,
    \end{equation}
    in terms of certain ``electric and magnetic twistor coordinates'' $\Xe$ and $\Xm$ \cite{Gaiotto:2010}.
\end{itemize}

Following the predictions of Gaiotto-Moore-Neitzke, Tulli \cite{Tulli:2019} identified $\Mov$ with a moduli space $\Xfr$ of (framed) rank $2$ harmonic bundles over $C = \CP^1$ with an irregular singularity at $\infty$.
Under this correspondence, the twistor coordinates $\Xe$ and $\Xm$ of the holomorphic symplectic form $\Omega^\ov_\zeta$ are described in terms of the Stokes data of the irregular connections $\nabla_\zeta$.

The question of an $L^2$-interpretation of the Ooguri-Vafa metric (or form) was left open in \cite{Tulli:2019}.
Unlike the usual moduli spaces of wild Higgs bundles (see e.g.\ \cite{Biquard:2004}), the parabolic weights and residues of the bundles in $\Xfr$ are allowed to vary, and as a result the naive formulas for $g_{L^2}$ and the corresponding Atiyah-Bott form are divergent.

In this paper, we introduce a regularization and gluing procedure to study the hyperkähler structure of $\Xfr$.
We use this procedure to identify the Ooguri-Vafa form $\Omega_\zeta^\ov$ with a regularized version of the Atiyah-Bott form on $\Xfr$.
By further restricting to the Hitchin section we deduce the corresponding statement about metrics, i.e.\ we identify a regularized version of Hitchin's metric with the Ooguri-Vafa metric.
Additionally, we prove the analogous results about the ``semiflat'' metrics and forms, which also play a role in the conjectural picture of Gaiotto-Moore-Neitzke.

Our main technique is to study the corresponding framed \emph{abelianized} connections on the spectral cover $\Sigma \to C$.\footnote{More precisely, the spectral cover $\Sigma$ itself varies along with the nonabelian connections on $C$. However, all of our calculations in the moduli space are local, and nearby $\Sigma$ are diffeomorphic, so we can identify them with a fixed surface when computing variations.}
Below we will give a high-level overview of the argument, followed by a more detailed summary of our results and strategy.

\subsection{A schematic guide} 
\label{sub:schematic-guide}

We will be interested in three related sets of framed objects (see \cref{fig:summary-of-spaces}): 

\begin{enumerate}
    \item {\color{green} $\Hfr$} -- the set of {\color{green} framed harmonic bundles} considered in \cite{Tulli:2019}, whose moduli space of isomorphism classes is $\Xfr \cong \Mov$.

    \item {\color{orange} $\Afr_\zeta$} -- a set of {\color{orange} framed flat $\SL(2)$-connections} on $C = \CP^1$, which can be obtained from harmonic bundles in $\Hfr$ by the \emph{nonabelian Hodge correspondence}.

    \item {\color{red} $\abAfr_\zeta$} --  a set of {\color{red} framed ``almost-flat'' $\GL(1)$-connections} on the spectral cover $\Sigma$ of $C$, which can be obtained by \emph{abelianizing} connections in $\Afr_\zeta$.
\end{enumerate}

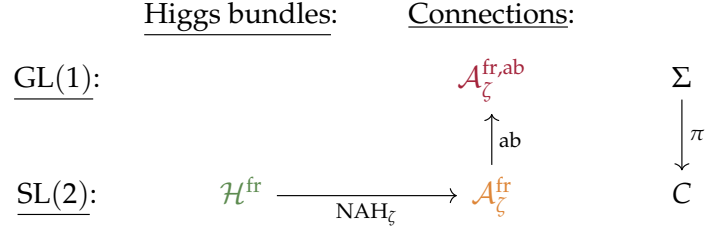
\begin{figure}[ht]
    \begin{tikzcd}
        &[-20pt] \text{\underline{Higgs bundles}:} &[-10pt] \text{\underline{Connections}:} \\[-20pt]
         \text{\underline{$\GL(1)$}:}  & & \color{red} \abAfr_\zeta & \Sigma \arrow[d, "\pi"] \\
        \text{\underline{$\SL(2)$}:} & \color{green} \Hfr \arrow[r, "\NAH_\zeta"'] & \color{orange} \Afr_\zeta \arrow[u, "\ab"'] & C
    \end{tikzcd}
    \caption{The main sets of framed objects and the maps between them.}
    \label{fig:summary-of-spaces}
\end{figure}

Our goal is to compare the following closed $2$-forms on {\color{green} $\Hfr$}, which descend to holomorphic symplectic forms on its space of isomorphism classes $\Xfr \cong \Mov$:
\begin{enumerate}
    \item {\color{green} $\Omega^\ov_\zeta$} -- the Ooguri-Vafa form, interpreted Stokes-theoretically via a form {\color{orange} $\Omega^\ov_\Stokes$ } on {\color{orange} $\Afr_\zeta$}.
    \item {\color{green} $\Omega^\reg_\zeta$} -- the pullback of a regularized version {\color{orange} $\Omega^\reg$} of the Atiyah-Bott form on {\color{orange} $\Afr_\zeta$}.
\end{enumerate}

We will argue that the above forms are in fact equal\footnote{up to a factor of $-4\pi^2$, which we suppress here} via the following commutative diagram, involving two intermediary abelian forms on {\color{red} $\abAfr_\zeta$}:
\begin{enumerate}
    \item[(3)] {\color{red} $\Omega^\glue$} -- a ``glued symplectic form'', pulled back from the Atiyah-Bott form on the torus.
    \item[(4)] {\color{red} $\Omega^{\reg, \ab}$} -- a regularized abelian Atiyah-Bott form on the spectral cover.
\end{enumerate}

\begin{equation}\label{eq:forms-cd}
\begin{tikzcd} 
 & &  & \color{red} \Omega^\glue \arrow[dd, maps to, "\eqref{eq:INTRO-glue-equals-ov}"', "\ab^*"]  \arrow[r, equals, "\eqref{eq:INTRO-glue-equals-reg}"] &\color{red} \Omega^{\reg, \ab} \arrow[dd, maps to, "\eqref{eq:INTRO-ab-symplectomorphism}"', "\ab^*"] \\
 & &  & & \\
\color{green} \Omega^\ov_\zeta \arrow[r, equals, dashed ] & \color{green} \Omega^\reg_\zeta &  & \color{orange} \Omega^\ov_\Stokes \arrow[r, equals, dashed ] \arrow[lll, "\NAH_\zeta^*", "\eqref{eq:INTRO-ov-stokes}"', maps to, bend left, dotted] & \color{orange} \Omega^{\reg} \arrow[lll, "\NAH_\zeta^*", "\coloneqq"', maps to, bend left, dotted]
\end{tikzcd}
\end{equation}

There are also \emph{semiflat} versions of the above spaces and forms, which are simpler and more explicit. 
We will introduce them in \cref{sec:sf-analysis} and carry out an analogous argument.

\pagebreak

\subsection{Detailed summary and strategy} 
\label{sub:summary}

\subsubsection{Framed bundles} 
\label{ssub:summary-framed-bundles}

The set $\Hfr$ of \emph{compatibly framed wild harmonic bundles} introduced in \cite{Tulli:2019} consists of tuples $(E, \theta, h, g)$ where, roughly: 
\begin{itemize}
    \item $E|_{\CP^1 \setminus \{\infty\}}$ is a holomorphic rank $2$ vector bundle equipped with
    \begin{itemize}
        \item a traceless Higgs field $\theta$ such that $\det \theta = -(z^2 + 2m) dz^2$ for some $m \in \C$,\footnote{The underlying unframed bundles $(E, \theta, h)$ thereby provide a local model of the Higgs moduli space near the generic part of the discriminant locus, which for $\SL(2)$-Higgs bundles consists of quadratic differentials with one double zero.} and
        \item a harmonic metric $h$.
    \end{itemize}

    \item $g$ is a frame of $E$ near $z=\infty$ with respect to which the Higgs field $\theta$ and holomorphic structure $\bar{\partial}_E$ are of a certain singular form.
\end{itemize}
(See \cref{sub:framed-harmonic-bundles} for the full definitions.)

Let $\Xfr$ denote the set of isomorphism classes of $\Hfr$.
The elements of $\Xfr$ are parametrized by the value $m \in \C$ describing the simple pole term of the singularity, $m^{(3)} \in (-\frac{1}{2}, \frac{1}{2}]$ describing the parabolic weights, and another $U(1)$-valued parameter describing the framing $g$.
(cf.\ the picture of the Ooguri-Vafa space $\Mov$ in \cref{fig:ov-space}  below.)

Given $(E, \theta, h, g) \in \Hfr$, the corresponding flat connection $\nabla_\zeta$ defined by \eqref{eq:INTRO-NAH} also has an irregular singularity at $z = \infty$, and consequently it undergoes Stokes phenomena.
With respect to the frame $g$, it is of the form
\begin{equation} \label{eq:INTRO-framed-connection-form}
    \nabla_\zeta = d + \left[- \zeta^{-1} \frac{dw}{w^3}  - \zeta \frac{d \overline{w}}{\overline{w}^3} - (\zeta^{-1} m - \frac{1}{2} m^{(3)}) \frac{dw}{w} -  (\zeta \overline{m} + \frac{1}{2} m^{(3)})\frac{d \overline{w}}{\overline{w}}  \right]H + \text{regular terms}
\end{equation}
near $w = 1/z = 0$, where $H = \diag(1, -1)$.

Let $\Afr_\zeta$ denote the space of framed connections $(E, \nabla, g)$ which are of the above form.

\begin{definitionalph}[=\,\cref{def:reg-form}] \label{def:INTRO-reg-form}
    Define a \emph{regularized Atiyah-Bott form} $\Omega^\reg$ on $\Afr_\zeta$ by
    \begin{equation}
        \Omega^\reg(\dot{\nabla}_1, \dot{\nabla}_2) = \lim_{R \to 0} \left[\int_{C_R} \tr (\dot{\nabla}_1 \wedge \dot{\nabla}_2) - 2 \pi \log R \cdot \tr \left( \mu_1 \lambda_2 - \mu_2 \lambda_1 \right) \right],
    \end{equation}
    where
    $C_R \coloneqq \CP^1 \setminus \{|w| < R \}$
    and, in polar coordinates $w = re^{i \theta}$ near $w=0$,
    \begin{equation} 
    \dot{\nabla}_i = (\mu_i + \O(r)) d \theta + (\lambda_i + \O(r)) \frac{dr}{r} \quad \text{for some diagonal matrices } \mu_i, \lambda_i.
    \end{equation}
    (The ``regularization term'' $- 2 \pi \log R \cdot  \tr \left( \mu_1 \lambda_2 - \mu_2 \lambda_1 \right)$ can be explicitly calculated using \eqref{eq:INTRO-framed-connection-form}.)
\end{definitionalph}

We can pull back $\Omega^\reg$ to a form on $\Hfr$ via $\NAH_\zeta: (E, \theta) \mapsto (E, \nabla_\zeta)$, and it furthermore descends to the moduli space $\Xfr$.
Slightly abusing notation, we will denote both of these pulled-back forms by $\Omega^\reg_\zeta$.

\subsubsection{The Ooguri-Vafa space} 
\label{ssub:summary-ov-space}

The Ooguri-Vafa space $\Mov$ is a hyperkähler space of complex dimension $2$.
Technical details on the construction of $\Mov$ using the Gibbons-Hawking ansatz can be found in \cite{Gross:2000}; see also the summaries in \cite{Gaiotto:2010,Tulli:2019}. We include a brief recap in \cref{sec:ooguri-vafa-construction}, but the most important attributes for our purposes are as follows.

The space $\Mov$ is  a singular torus fibration over a disc $\B$ in $\C$, whose central fibre over $0$ is a torus with a node. 
The other fibres are nonsingular tori parametrized by an electric angle $\theta_e$ and magnetic angle $\theta_m$ (see \cref{fig:ov-space}).

\begin{figure}[ht]
    \centering
    \includegraphics[scale=1]{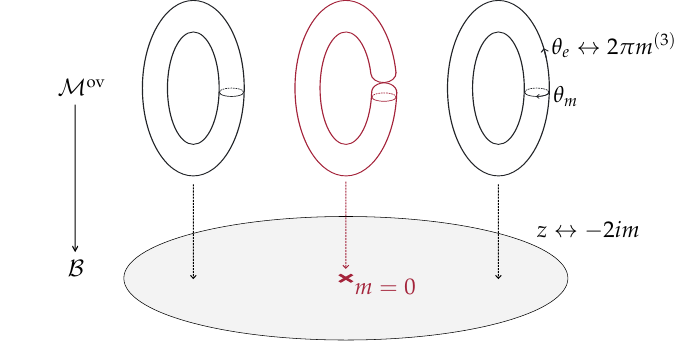}
    \caption{The Ooguri-Vafa space $\Mov$, regarded as a singular torus fibration over a disc $\B \subset \C$.}
    \label{fig:ov-space}
\end{figure}

\begin{remark}[$\theta_m$ monodromy]
    The coordinate $\theta_m$ is not globally defined; it has monodromy
    \begin{equation}
        \theta_m \to \theta_m + \theta_e - \pi
    \end{equation}
     as we go counterclockwise around $z=0$.
\end{remark}

Under the correspondence $\Mov \cong \Xfr$ in \cite{Tulli:2019}:\footnote{More precisely, \cite{Tulli:2019} identifies a \emph{subset} of $\Xfr$ (consisting of isomorphism classes whose parameter $m$ lies in a sufficiently small disc) with the Ooguri-Vafa space $\Mov(\Lambda)$, for a certain cutoff $\Lambda \in \C$. See \cref{sub:ov-correspondence} for more details.}
\begin{itemize}
    \item The angles $\theta_e$ and $\theta_m$ are expressed in terms of certain explicit quantities involving $m, m^{(3)}$, and the framing $g$.
    
    \item The twistor coordinates $\Xe(\zeta)$ and $\Xm(\zeta)$ of the holomorphic symplectic form $\Omega^\ov_\zeta$ are interpreted in terms of the Stokes data of the connections $\nabla_\zeta$ (namely the formal monodromy of $\nabla_\zeta$ and the non-trivial element of a certain Stokes matrix, respectively).
\end{itemize}

We will say more about this correspondence below, but at this point we can formulate our main result.
\begin{theoremalph}[=\,\cref{thm:reg-equals-ov}] \label{thm:INTRO-a}
    Under the identification of spaces $\Mov \cong \Xfr$,  
    \begin{equation} \label{eq:INTRO-reg-equals-ov}
         \Omega^\ov_\zeta = - \frac{1}{4\pi^2} \Omega^\reg_\zeta,
    \end{equation}
    i.e.\ the Ooguri-Vafa symplectic form coincides with (a multiple of) the regularized Atiyah-Bott form, pulled back to $\Xfr$.
\end{theoremalph}

\subsubsection{Stokes data and abelianization} 
\label{ssub:summary-stokes-data}

The interpretation of $\Xe$ and $\Xm$ from \cite{Tulli:2019} can be fully stated in terms of the Stokes data of the connections $\nabla_\zeta$ (see \cref{ssub:stokes-conventions,sub:ov-correspondence}), but for our purposes it will be more useful to describe them---especially $\Xm$---in terms of the corresponding \emph{abelianized} connections $\nabla^\ab_\zeta$.

An $\SL(2)$-Higgs bundle $(E, \theta)$ on a surface $C$ has an associated double cover, its \emph{spectral curve} 
\begin{equation*}
    \Sigma = \{ \lambda \in T^*C : \det (\theta - \lambda I) = 0 \} \subseteq T^*C.
\end{equation*} 
Many of the related geometric objects have simpler $\GL(1)$-versions on $\Sigma$.
In particular, flat connections on $C$ can be lifted to abelian connections on $\Sigma$ using a \emph{spectral network} \cite{Gaiotto:2013}.
(See \cref{sub:spectral-networks-review} for a detailed review of the relevant material.)

Each connection $\nabla_\zeta$ coming from $(E, \theta)$ has an associated spectral network $\W_\zeta$, which is a collection of walls\footnote{with certain labels, as described in \cref{def:spectral-network}} on $C$.
In our case, for a connection coming from $(E, \theta, g) \in \Hfr$, the topology of the spectral network $\W_\zeta$ depends on the values of $m$ and $\zeta$ (see \cref{fig:intro-sn} for an example).

\begin{figure}[ht]
    \centering
    \includegraphics[scale=1]{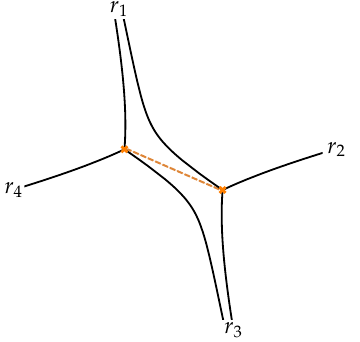}
    \caption{One of the two generic topologies for a spectral network $\W_\zeta$ coming from $(E, \theta, g) \in \Hfr$, shown here for $\Re(\zeta^{-1} m) > 0$.}
    \label{fig:intro-sn}
\end{figure}

The abelianization procedure uses $\W_\zeta$ to lift $\nabla_\zeta$ to an ``almost-flat'' connection $\nabla_\zeta^\ab$ over $\Sigma$.
Concretely, $\nabla^\ab_\zeta$ can be constructed by choosing a basis of flat sections $(s_i, s_j)$ in each cell of the network, as shown in \cref{fig:intro-sn-labelled}.

\begin{figure}[ht]
    \centering
    \includegraphics[scale=1]{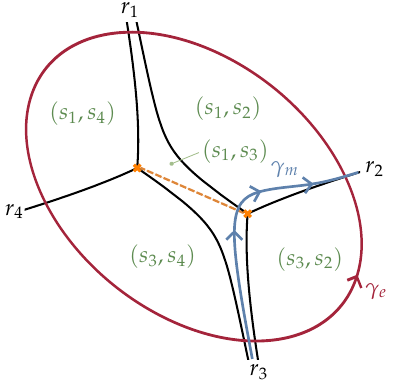}
    \caption{Additional decorations for a spectral network $\W_\zeta$ coming from $(E, \theta, g) \in \Hfr$, shown for $\Re(\zeta^{-1} m) > 0$.
     The flat sections $(s_i, s_j)$ used for abelianization in each cell are labelled in green.
    The paths $\gamma_e$ and $\gamma_m$ on $\Sigma$ are used to interpret $\Xe$ and $\Xm$ as parallel transports of $\nabla^\ab_\zeta$.}
    \label{fig:intro-sn-labelled}
\end{figure}

The Ooguri-Vafa twistor coordinates can be interpreted Stokes-theoretically in terms of the flat sections $s_i$, or equivalently in terms of parallel transports of the abelianized connection $\nabla^\ab_\zeta$.

\begin{itemize}
    \item $\Xe$ is one of the diagonal entries of the formal monodromy of $\nabla_\zeta$. 
    It can be calculated by the cross-ratio
    \begin{equation} \label{eq:Xe-cross-ratio}
        \Xe(\zeta) =  \frac{s_1 \wedge s_4}{s_3 \wedge s_4} \, \frac{s_3 \wedge s_2}{s_1 \wedge s_2}
    \end{equation}
    which corresponds to the parallel transport of $\nabla^\ab_\zeta$ around the path $\gamma_e$ shown in \cref{fig:intro-sn-labelled}.
    (Note that \eqref{eq:Xe-cross-ratio} is the usual formula for a \emph{spectral coordinate} \cite{Gaiotto:2014a}/Fock-Goncharov coordinate \cite{Fock:2006}, and is invariant under rescaling $s_i \to c_i s_i$.)

    \item $\Xm$ is the off-diagonal entry of one of the Stokes matrices.
    In order to make gauge-invariant sense of this, the framing $g$ of the bundles $(E, \nabla_\zeta)$ is crucial: it allows us to single out a normalization for each section $s_i$ by prescribing the asymptotics near the singularity \cite{Gaiotto:2013a}.
    Choosing suitably normalized sections (as described in \cref{ssub:stokes-conventions}), we can calculate $\Xm$ by the ratio
    \begin{equation} \label{eq:Xm-ratio}
        \Xm(\zeta) = \begin{cases}
            \dfrac{s_3 \wedge s_1}{s_2 \wedge s_1} & \text{if }  \Re(\zeta^{-1} m) > 0, \\\\
            -\dfrac{s_3\wedge s_2}{s_4 \wedge s_2} & \text{if } \Re(\zeta^{-1} m) < 0,
        \end{cases} 
    \end{equation}
    which corresponds to a \emph{regularized parallel transport} of $\nabla^\ab_\zeta$ along the open path $\gamma_m$ shown in \cref{fig:intro-sn-labelled}.
    (See the recent paper \cite{Alekseev:2024} for discussion of spectral coordinates for open paths.)
\end{itemize}

We can think of these Stokes-theoretic interpretations of $\Xe$ and $\Xm$ as defining a form
\begin{equation}
    \Omega^\ov_\Stokes \coloneqq -\frac{1}{4 \pi^2} d \log \Xe \wedge d \log \Xm
\end{equation}
on $\Afr_\zeta$, which pulls back to 
\begin{equation} \label{eq:INTRO-ov-stokes}
    (\NAH_\zeta)^*\Omega^\ov_\Stokes = \Omega^\ov_\zeta
\end{equation}
on $\Xfr \cong \Mov$ under the identification in \cite{Tulli:2019}.
(We discuss this further in \cref{sub:ov-correspondence}.)

In \cref{sub:framing-near-the-punctures} we explain how the given frame $g$ for $\nabla_\zeta$ naturally induces a frame $g^\ab$ for $\nabla_\zeta^\ab$.
This allows us to define a space $\abAfr_\zeta$ of framed abelian connections and a corresponding regularized abelian Atiyah-Bott form $\Omega^{\reg, \ab}$, analogously to \cref{def:INTRO-reg-form} above.
We prove in \cref{sub:reg-and-abelianization} that abelianization preserves these forms; that is,
\begin{equation} \label{eq:INTRO-ab-symplectomorphism}
    \ab^* \Omega^{\reg, \ab} = \Omega^\reg.
\end{equation}

\subsubsection{Gluing and regularization} 
\label{ssub:summary-gluing}

The next question is how to relate the Ooguri-Vafa form \eqref{eq:INTRO-ov-twistor-coords} to the regularized Atiyah-Bott form.

As motivation, note that if $S$ is a compact surface of genus $g$ with standard homology basis $a_1, \dots, a_g, b_1, \dots, b_g$, then the logs of the holonomies of flat $\C^*$-connections along $a_i$ and $b_i$ are Darboux coordinates for the abelian Atiyah-Bott form:
\begin{equation}
    \int_S \delta \nabla \wedge \delta \nabla = \sum_{i=1}^{g} d \log \Hol_{a_i} \! \nabla \wedge \, d \log \Hol_{b_i} \! \nabla.
\end{equation}
(This is essentially just a restatement of the Riemann bilinear identity.)

In our case the spectral cover $\Sigma \xrightarrow{\pi} C$ is not compact, and the (non-regularized) Atiyah-Bott form on $\abAfr_\zeta$ is divergent.
To remedy this we instead consider the cut off surface $\Sigma_R \coloneqq \pi^{-1} (C_R)$, which is topologically a cylinder, and glue the ends to form a torus $T$.

The framed $\C^*$-connections in $\abAfr_\zeta$ have a prescribed form near the boundary $\partial \Sigma_R$, which does \emph{not} automatically glue to define a connection on $T$, but we can glue them by making an appropriate gauge transformation $\chi=\chi(\alpha)$ for each connection $\nabla = d + \alpha$.
This allows us to pull back the abelian Atiyah-Bott form from $T$ to obtain a ``glued symplectic form'' 
\begin{equation}
    \Omega^\glue(\dot{\alpha}_1, \dot{\alpha}_2) =  \int_{\Sigma_R} (\dot{\alpha}_1 - d \dot{\chi}_1) \wedge (\dot{\alpha}_2 - d \dot{\chi}_2) 
\end{equation}
on $\abAfr_\zeta$ (which is in fact independent of the cutoff $R \ll 1$).
We describe this gluing construction in more detail in \cref{sub:general-gluing-construction}.

Now we reach the key point---once suitably chosen, \emph{the gluing map $\chi$ simultaneously provides the regularization} for the other constructions:
\begin{itemize}
    \item On the one hand, we can rewrite 
\begin{equation}
    \Omega^\glue (\dot{\alpha}_1, \dot{\alpha}_2) =  \int_{\Sigma_R} \dot{\alpha}_1 \wedge \dot{\alpha}_2 + \int_{\partial  \Sigma_R} (\dot{\chi}_2 \dot{\alpha}_1 -  \dot{\chi}_1 \dot{\alpha}_2 + \dot{\chi}_1 d \dot{\chi}_2).
\end{equation}
We show in \cref{ssub:regularization-atiyah-bott} that the boundary integral coincides with the regularization term of the regularized abelian Atiyah-Bott form $\Omega^{\reg, \ab}$ as $R \to 0$, and consequently
\begin{equation} \label{eq:INTRO-glue-equals-reg}
    \Omega^\glue = \Omega^{\reg, \ab}.
\end{equation}

\item On the other hand, it follows from the Riemann bilinear identity on the torus that
\begin{equation}
     \Omega^\glue(\dot{\alpha}_1, \dot{\alpha}_2) = \int_{\gamma_e} \dot{\alpha}_1 \int_{\gamma_{m, R}} (\dot{\alpha}_2 - d \dot{\chi}_2) - \int_{\gamma_e} \dot{\alpha}_2 \int_{\gamma_{m, R}} (\dot{\alpha}_1 - d \dot{\chi}_1)
\end{equation}
(where $\gamma_{m, R}$ denotes the restriction of the open path from \cref{fig:intro-sn-labelled} to $\Sigma_R$).
When $\Omega^\glue$ is pulled back to $\Afr_\zeta$, the integrals $\int_{\gamma_e} \dot{\alpha}$ correspond to the parallel transport for $\Xe$, and we show in \cref{sub:interpreting-magnetic-coordinate} that the integrals $\int_{\gamma_{m, R}} (\dot{\alpha} - d \dot{\chi})$ calculate the appropriate regularized parallel transports for $\Xm$.
Consequently
 \begin{equation} \label{eq:INTRO-glue-equals-ov}
     \ab^* \Omega^\glue = \Omega^\ov_\Stokes.
 \end{equation}
\end{itemize}

Combining all of these identifications via the commutative diagram \eqref{eq:forms-cd} gives Theorem A.

\subsubsection{The semiflat story} 
\label{ssub:summary-semiflat}

Solutions to the abelian Hitchin equation on the spectral cover $\Sigma$ lead to a corresponding \emph{semiflat} hyperkähler metric $g_{L^2}^\sf$ on the Higgs moduli space, defined away from the discriminant locus (at which $\Sigma$ fails to be smooth).

This simpler metric is also part of the picture of Hitchin's metric described by Gaiotto-Moore-Neitzke: they predicted that $g_{L^2}$ exponentially approaches $g_{L^2}^\sf$ along a ray $(E, t \theta)$ as $t \to \infty$.
Many versions of this statement have now been proved, such as in \cite{Dumas:2019} for $\SL(2)$-Higgs bundles on the Hitchin section, \cite{Fredrickson:2022} for the parabolic case, and \cite{Fredrickson:2020} for higher rank.
The asymptotic behaviour of $g_{L^2}$ continues to be a subject of active research. 
The recent work \cite{He:2025} has extended these exponential convergence statements to classes of Higgs bundles in the discriminant locus of the (ordinary) $\SL(2)$-Hitchin moduli space.

Returning to our setting, one could ask how the semiflat metric $g_{L^2}^\sf$ (or its corresponding holomorphic symplectic form) behaves near the discriminant locus.
There are natural semiflat versions of all of the constructions described above, such as a semiflat Ooguri-Vafa form $\Omega^{\ov, \sf}_\zeta$, and a regularized Atiyah-Bott form $\Omega_\zeta^{\reg, \sf}$ for the ``semiflat connections''
\begin{equation} \label{eq:INTRO-sf-connection}
    \nabla^\sf_\zeta = \zeta^{-1} \theta + D_{h_\sf} + \zeta \theta^{\dagger_{h_\sf}}.
\end{equation}

However, the argument in \cite{Tulli:2019} does \emph{not} prove that the semiflat Ooguri-Vafa magnetic coordinate is given by the Stokes data of the  semiflat connections. 
Our task is now reversed:
\begin{itemize}
    \item Before, we started with a Stokes-theoretic interpretation of the (non-explicit) magnetic coordinate $\Xm$, but had to develop the gluing procedure to study the corresponding integral.

    \item Now, we can follow essentially the same gluing procedure, but still need to match up the Stokes-theoretic integral with the explicit formula for $\Xm^\sf$ .
\end{itemize}

The formula for the magnetic angle $\theta_m$ (and hence $\Xm^{\sf}$) under the correspondence $\Mov \cong \Xfr$ involves an integral of the Chern connection $D_h$.
In the semiflat setting it is more natural to consider a ``shifted angle'' $\theta_m^\shift$ defined in terms of the \emph{semiflat} Chern connection $D_{h_\sf}$, leading to a corresponding shifted form $\Omega^{\ov, \shift}_\zeta$ (see \cref{def:shifted-magnetic}).
Adapting our previous argument to this setting, we obtain the following analogue of \cref{thm:INTRO-a}.

\begin{theoremalph}[=\,\cref{thm:sf-reg-equals-shifted-ov}] 
    Under the identification of spaces $\Mov \cong \Xfr$,  
        \begin{equation} \label{eq:INTRO-sf-reg-equals-shifted-ov}
             \Omega^{\ov, \shift}_\zeta = - \frac{1}{4\pi^2} \Omega^{\reg, \sf}_\zeta,
        \end{equation}
    i.e.\ the shifted semiflat Ooguri-Vafa form coincides with the regularized semiflat Atiyah-Bott form, pulled back to $\Xfr$.
\end{theoremalph}

The angles $\theta_m$ and $\theta_m^\shift$ are not obviously the same in general, but we prove in \cref{sec:duality-on-hitchin-section} that they both vanish on a suitable framed version of the Hitchin section $\Bfr \subset \Xfr$.
Here $\Bfr$ consists of bundles $E = K_C^{-1/2} \oplus K_C^{1/2}$ with Higgs field
\begin{equation}
    \theta =\begin{pmatrix}
            0 & 1 \\ 
            (z^2 + 2m)  dz^2 & 0
        \end{pmatrix}, \quad m \in \C^*,
\end{equation}
and a specific choice of framing and parabolic weights (see \cref{def:framed-hitchin-section} for the details).
This Hitchin section $\Bfr$ exhibits a natural notion of self-duality, which we use to calculate $\theta_m$ and $\theta_m^\shift$.

\begin{theoremalph}[=\,\cref{thm:hitchin-section-shift}] 
    Restricted to the Hitchin section $\Bfr \subset \Xfr$,
    \begin{equation}
         \theta_m|_{\Bfr} \equiv 0 \equiv  \theta_m^\shift|_{\Bfr}.
     \end{equation}
    Consequently
    \begin{equation}
        \Omega^{\ov, \sf}_\zeta|_{\Bfr}  = -\frac{1}{4 \pi^2} \Omega^{\reg, \sf}_\zeta|_{\Bfr},
    \end{equation}
    i.e.\ the (usual) semiflat Ooguri-Vafa form coincides with the regularized semiflat Atiyah-Bott form.
\end{theoremalph}

We leave analysis of the difference $\theta_m - \theta_m^\shift$ away from the Hitchin section as an interesting question for future work.

The other natural next step is to translate our results from the symplectic forms to their corresponding metrics.
On the Hitchin section, the $L^2$-metric can be written down more explicitly and the complex structure $I$ is easy to understand.
In \cref{sec:metric-on-the-hitchin-section} we use these explicit descriptions to define a regularized metric $g^{\reg}_{L^2}$ and deduce the analogue of \cref{thm:INTRO-a}.

\begin{theoremalph}[=\,\cref{thm:reg-equals-ov-metric}] 
    On the Hitchin section parametrized by quadratic differentials $\phi(z) = (z^2 + 2m) dz^2$ for $m \in \C$,
    \begin{equation}
        g^\reg_{L^2} = 4 \pi^2 \cdot g^\ov,
    \end{equation}
    i.e.\ the regularized version of Hitchin's metric coincides with the Ooguri-Vafa metric.
\end{theoremalph}
The same is also true of the corresponding semiflat metrics.

Off the Hitchin section (i.e.\ on the rest of the space $\Xfr$), it is less clear how the complex structure should act on variations of the parabolic weights or framing.
We leave these considerations for future work.

More broadly, it would be interesting to extend our techniques to study larger classes of hyperkähler structures, such as for other wild Higgs moduli spaces (e.g.\ with multiple poles or in higher rank), or for Poisson-Lie groups (cf.\ \cite{Alekseev:2024}).

\subsection{Organization of the paper} 
\label{sub:organization}

In \cref{sec:setup} we introduce the main spaces under consideration and recall the relevant background.
In \cref{sec:abelianization-and-framing} we review the abelianization procedure and describe how it extends to framed bundles. 
In \cref{sec:regularized-forms} we introduce a regularized version of the Atiyah-Bott form and show that it is preserved by abelianization.
In \cref{sec:glued-symplectic-form} we describe a construction for a glued symplectic form on the space of abelianized connections, and use it to show that the Ooguri-Vafa form coincides with the regularized Atiyah-Bott form.

The next sections are focused on the analogous semiflat story.
In \cref{sec:sf-analysis} we apply our earlier arguments from \crefrange{sec:setup}{sec:glued-symplectic-form} to this modified setting.
In \cref{sec:duality-on-hitchin-section} we introduce a framed Hitchin section on which we further study the Ooguri-Vafa magnetic angle.

In \cref{sec:metric-on-the-hitchin-section} we use our main results about the Ooguri-Vafa and Atiyah-Bott symplectic forms to describe the metric on the Hitchin section.

The appendices contain some more technical background and calculations.
The reader may find it useful to look through these sections first.

\subsection*{Acknowledgements} 

I would like to sincerely thank my advisor, Andy Neitzke, for introducing this problem to me and for the extremely helpful discussions, suggestions, and encouragement over the course of preparing this paper.
Many of the spectral network figures below were produced using his Mathematica notebook \texttt{swn-plotter.nb} \cite{Neitzke:swn}.

I also thank the Simons Center for Geometry and Physics for hospitality during the \emph{Geometric, Algebraic, and Physical Structures around the moduli of Meromorphic Quadratic Differentials} program in Spring 2024, during which part of this work was completed.

\section{General background and setup} 
\label{sec:setup}

In this section we give an overview of the main objects and spaces that will appear throughout the text.
We will relegate some of the other technical background to the later sections and introduce additional tools as they become relevant.

\subsection{Framed harmonic bundles} 
\label{sub:framed-harmonic-bundles}

First, we recall the spaces of Higgs bundles studied in \cite{Tulli:2019}.
For consistency we will use the same notation and conventions.

\begin{definition}[Harmonic bundles in $\H$] \label{def:harmonic-bundles}
    Let $\H$ denote the set of rank $2$ wild\footnote{We will only work with unramified wild objects.} harmonic bundles $(E, \bar{\partial}_E, \theta, h)$ over $\CP^1 \setminus \{\infty\}$ such that $\tr \theta =0 $ and $\det \theta = -(z^2 + 2m) dz^2$ for some $m \in \C$.

    Recall that ``wild harmonic'' means:
    \begin{itemize}
        \item $(E, \bar{\partial}_E)$ is a holomorphic vector bundle over $\CP^1 \setminus \{ \infty \}$.
        \item the Higgs field $\theta$ is an $\End(E)$-valued $1$-form with $\bar{\partial}_E \theta = 0$.
        \item the harmonic metric $h$ is a hermitian metric satisfying Hitchin's equation
      \begin{equation} \label{eq:hitchin}
          F_{D_h} + [ \theta, \theta^{\dagger_h}] = 0,
      \end{equation}
    where $D_h$ is the Chern connection for $(\bar{\partial}_E, h)$ and $F_{D_h}$ is its curvature.

    \item (wildness): there is a holomorphic coordinate $w$ in a neighbourhood $U$ of $\infty$ and a decomposition
    \begin{equation*}
        (E, \bar{\partial}_E, \theta)|_U = \bigoplus_{a \in \mathcal{I}} (E_a, \bar{\partial}_{E_a}, \theta_a),
    \end{equation*}
    where $\mathcal{I} \subset w^{-1} \C[w^{-1}]$ is the set of \emph{irregular types} and each $\theta_a - da \cdot \id_{E_a}$ has at worst a simple pole.
    \end{itemize}
\end{definition}

\begin{definition}[Framed harmonic bundles in $\Hfr$] \label{def:framed-harmonic}
    Let $\Hfr$ denote the set of \emph{compatibly framed harmonic bundles} $(E, \bar{\partial}_E, \theta, h, g)$ where:

    \begin{itemize}
    \item $(E, \bar{\partial}_E, \theta, h)|_{\CP^1 \setminus \{\infty\}} \in \H$.
    \item $(E, h)$ is an $\SU(2)$-bundle over $\CP^1$ (i.e.\ a unitary extension of the above bundle over $\infty$).\footnote{This also means that $(E, h)$ comes with a volume form trivializing $\det E$, but this won't play much of a role for us.}

    \item $g$ is an $\SU(2)$-frame of $E_\infty$ that extends to an $\SU(2)$-frame in a neighbourhood of $z = \infty$ with respect to which
        \begin{gather}
            \theta = -H \frac{dw}{w^3} - m H \frac{dw}{w} + \text{regular terms}, \label{eq:higgs-framed-form} \\
            \bar{\partial}_E = \bar{\partial} - \frac{m^{(3)}}{2} H \frac{d \overline{w}}{\overline{w}} + \text{regular terms} \quad \text{for some $ m^{(3)} \in (-\tfrac{1}{2}, \tfrac{1}{2}]$.} \label{eq:holo-str-framed-form}
        \end{gather}
        Here and throughout, $w = 1/z$ and $H = \begin{pmatrix}
            1 & 0 \\ 0 & -1
        \end{pmatrix}$.
    \end{itemize}
\end{definition}

For brevity we will usually omit $\bar{\partial}_E$ from the notation, and write $(E, \theta, h)\in \H$ for the underlying wild harmonic bundle and $(E, \theta, h, g) \in \Hfr$ for the framed bundle.
(Note that the former is a bundle over $\CP^1 \setminus \{\infty\}$ while the latter is a bundle over $\CP^1$.)
We will sometimes also omit $h$ when it is not relevant.

\begin{remark}[Parabolic interpretation of $m^{(3)}$] \label{rem:parabolic-weight-cases}
    The bundles $(E, \theta, h) \in \H$ naturally carry a \emph{filtered structure}, defined in terms of growth rates with respect to the harmonic metric $h$.
    This in turn induces a \emph{parabolic structure} with weights in $(-\frac{1}{2}, \frac{1}{2}]$.
    (We review these definitions in \cref{sub:parabolic-and-filtered-bundles}.)
    The parameter $m^{(3)}$ describes the parabolic weights: 
    \begin{itemize}
        \item If $m^{(3)} \in (-\frac{1}{2}, \frac{1}{2})$, then the parabolic weights are $\pm m^{(3)}$, associated to the $\theta$-eigenlines near $z=\infty$ with respective eigenvalues $\pm (z + m/z +\dots) dz = \pm (-1/w^3 -m/w + \dots) dw$.

        \item If $m^{(3)} = \frac{1}{2}$, then the parabolic weight is $\frac{1}{2}$ with multiplicity $2$, associated to the trivial filtration near $\infty$.
    \end{itemize}
\end{remark}

\begin{definition}[Sets of isomorphism classes] 
    \begin{enumerate}[(i)]
        \item[]
        \item Let $\Xfr$ denote the set of isomorphism classes of $\Hfr$.
    
        \item For fixed $m \in \C$ and $m^{(3)} \in (-\frac{1}{2}, \frac{1}{2}]$, let $\Xfr(m, m^{(3)}) \subseteq \Xfr$ consist of the classes of framed Higgs bundles whose singularity is described by the parameters $m$ and $m^{(3)}$ in \eqref{eq:higgs-framed-form} and \eqref{eq:holo-str-framed-form} respectively.
    \end{enumerate}
\end{definition}

\begin{prop}[$U(1)$-action, {\cite[Proposition 4.1 \& Lemmas 4.3 and 4.4]{Tulli:2019}}]
    For $g = (e_1, e_2)$, let $e^{i \vartheta} \cdot g \coloneqq (e^{i \vartheta} e_1, e^{-i \vartheta} e_2)$.
    Then
    \begin{equation}
        e^{i \vartheta} \cdot [(E, \theta, g)] \coloneqq [(E, \theta, e^{i \frac{\vartheta}{2}} \cdot g)]
    \end{equation}
    defines a $U(1)$-action on $\Xfr$. 
    \begin{enumerate}
    \item For $m \neq 0$ and any $m^{(3)} \in (-\frac{1}{2}, \frac{1}{2}]$, the set $\Xfr(m, m^{(3)})$ is a $U(1)$-torsor under this action. 

    \item For $m=0$, the set $\Xfr(0, m^{(3)})$ is a $U(1)$-torsor if $m^{(3)} \neq 0$ and a single point if $m^{(3)} = 0$. 
    \end{enumerate}
    (cf.\ the picture of the Ooguri-Vafa space in \cref{fig:ov-space}.)
\end{prop}

In \cref{sub:constructing-and-extending-compatible-frames} we discuss explicit constructions of compatibly framed bundles, but this will not come into play until \cref{sec:sf-analysis}.

\subsection{Framed connections} 
\label{sub:framed-connections}

For each $(E, \theta, h) \in \H$ and $\zeta \in \C^*$, there is a corresponding flat connection 
\begin{equation} \label{eq:NAH}
    \nabla_\zeta \coloneqq \zeta^{-1} \theta + D_h + \zeta \theta^{\dagger_h}.
\end{equation}
(Flatness of $\nabla_\zeta$ is equivalent to $h$ satisfying Hitchin's equation \eqref{eq:hitchin}.)
This defines the \emph{nonabelian Hodge map} $\NAH_\zeta: (E, \theta) \mapsto (E, \nabla_\zeta)$.

If we start with a framed harmonic bundle $(E, \theta, h, g) \in \Hfr$, then with respect to the (extension of the) frame $g$ in a neighbourhood of $w=0$,
the Chern connection is of the form 
\begin{equation} \label{eq:chern-framed-form}
    D_h = d  + \frac{m^{(3)}}{2} H \left( \frac{dw}{w} - \frac{d \overline{w}}{\overline{w}} \right) + \text{regular terms,}
\end{equation}
and hence
\begin{align}
    \begin{split}
        \nabla_\zeta &= d + \left[\zeta^{-1} \left(- \frac{dw}{w^3} - m \frac{dw}{w}  \right) + \frac{m^{(3)}}{2}  \left( \frac{dw}{w} - \frac{d \overline{w}}{\overline{w}} \right) + \zeta \left(- \frac{d \overline{w}}{\overline{w}^3} - \overline{m} \frac{d \overline{w}}{\overline{w}}\right)\right]H  \\
    &\qquad + \text{regular terms}
    \end{split} \\
    \begin{split} \label{eq:framed-connection-form}
        &= d + \left[- \zeta^{-1} \frac{dw}{w^3}  - \zeta \frac{d \overline{w}}{\overline{w}^3} - (\zeta^{-1} m - \frac{1}{2} m^{(3)}) \frac{dw}{w} -  (\zeta \overline{m} + \frac{1}{2} m^{(3)})\frac{d \overline{w}}{\overline{w}}  \right]H \\
        &\qquad+ \text{regular terms}.
    \end{split}
\end{align}

These framed connections will be our primary objects of interest, so we will give a name to the corresponding space.

\begin{definition}[Framed flat bundles in $\Afr_\zeta$] \label{def:framed-connections}
    For fixed $\zeta \in \C^*$, let $\Afr_\zeta$ denote the set of \emph{$\zeta$-compatibly framed flat bundles} $(E, \nabla, g)$ where:
        \begin{itemize}
        \item $E$ is a rank $2$ holomorphic vector bundle over $\CP^1$, with parabolic weights at $z=\infty$ described by the parameter $m^{(3)} \in (-\frac{1}{2}, \frac{1}{2}]$ as in \cref{rem:parabolic-weight-cases}.

        \item $\nabla$ is a flat (complex) connection on $E$ with irregular singularity at $\infty$, of the form \eqref{eq:framed-connection-form} with respect to the framing $g$ near $\infty$.
    \end{itemize}
    Say that $(E, \nabla, g) \cong (E', \nabla', g')$ if there is a bundle isomorphism $E \xrightarrow{\sim} E'$ preserving the additional structure, and let $\Mfr_\zeta$ denote the set of isomorphism classes of $\Afr_\zeta$.
\end{definition}

To summarize, \cref{tab:notation-summary} lists the main spaces of objects and their isomorphism classes introduced so far (cf.\ \cref{fig:summary-of-spaces}).

\begin{table}[ht] 
    \renewcommand{\arraystretch}{1.3}
    \begin{tabular}{c|c|c}
     & framed harmonic bundles & framed connections \\ \hline
    sets of objects & \color{green} $\Hfr$ & \color{orange} $\Afr_\zeta$ \\
    moduli spaces & \color{green}$\Xfr$ & \color{orange} $\Mfr_\zeta$
    \end{tabular}
    \caption{Spaces of framed bundles over $C = \CP^1$.}
    \label{tab:notation-summary}
\end{table}

\begin{remark}[$\zeta$-compatible framings]
    The terminology above is nonstandard; our notion of a $\zeta$-compatible framing is just a translation of the definition of a compatible framing for a harmonic bundle under the nonabelian Hodge map $\NAH_\zeta$.

    There is already a more common notion of compatible framing in the Stokes theory of \emph{meromorphic} connections (see e.g.\ \cite{Boalch:2001}), where it means that the leading coefficient of the singular part of the connection is diagonal.\footnote{This is the definition used in \cite{Tulli:2019}, in which it is shown how to obtain such a compatible frame $\tau$ from $g$. We summarize the relevant details in \cref{sub:holomorphic-frames}.}
    In our case it will be more convenient to work in the $C^\infty$ setting, but the needed results from the classical Stokes theory carry over, as we discuss in \cref{sec:classical-stokes-theory}.
\end{remark}

By construction, the nonabelian Hodge formula \eqref{eq:NAH} for $\nabla_\zeta$ defines a map
\begin{align} \label{eq:NAH-map-set}
\begin{split}
    \NAH_\zeta: \Hfr &\to \Afr_\zeta \\ 
    (E, \theta, h, g) &\mapsto (E, \nabla_\zeta, g),
\end{split}
\end{align}
which descends to a map of moduli spaces
\begin{align} \label{eq:NAH-map-moduli}
\begin{split}
    \NAH_\zeta: \Xfr &\to \Mfr_\zeta \\
    [(E, \theta, h, g)] &\mapsto [(E, \nabla_\zeta, g)].
\end{split}
\end{align}
Throughout the paper we will use $\NAH_\zeta$ to pull back various symplectic forms from the space $\Afr_\zeta$ (resp.\ $\Mfr_\zeta$) to $\Hfr$ (resp.\ $\Xfr$).

\begin{remark}[Moduli space expectations] \label{rem:nonabelian-moduli-expectations}
    In \cite{Tulli:2019}, the moduli space $\Xfr$ of isomorphism classes in $\Hfr$ is really just defined as a \emph{set}; it later obtains an induced hyperkähler structure from the identification with the Ooguri-Vafa space $\Mov$, but we do not have a direct gauge-theoretic construction of the moduli space.
    (The parabolic weights and residues of the bundles in $\Xfr$ are allowed to vary, unlike the usual spaces of wild Higgs bundles in \cite{Biquard:2004}.
    However, the associated moduli space $\Mfr_\zeta$ of connections is similar to the ``extended moduli spaces'' in \cite{Boalch:2001}, where the formal type of the connections is not fixed.\footnote{These extended moduli spaces were used to construct numerous new complete hyperkähler manifolds; see the recent paper \cite{Boalch:2024} for discussion of their deformation classes.}
    It would be interesting to compare these constructions and their respective symplectic forms.)

    We will assume in the following calculations that our moduli spaces have the ``obvious'' tangent spaces, i.e.\ that variations of elements in $\Xfr$ or $\Mfr_\zeta$ can be represented by endomorphism-valued $1$-forms which are of the appropriate framed form near $z=\infty$, modulo the action of gauge transformations which approach the identity near $\infty$ and preserve the framed form.

    We also expect (but will not prove or need) that other features of the usual wild nonabelian Hodge correspondence \cite{Biquard:2004} hold, e.g.\ that $\NAH_\zeta^\fr$ gives a homeomorphism of moduli spaces $\Xfr \xrightarrow{\sim} \Mfr_\zeta$ which is a diffeomorphism away from the discriminant locus $m = 0$.
\end{remark}

At this point we can define the regularized Atiyah-Bott form $\Omega^\reg$ on $\Afr_\zeta$, as in \cref{def:INTRO-reg-form}.
We will return to this and study its properties in \cref{sec:regularized-forms}.

\subsubsection{Conventions for Stokes data} 
\label{ssub:stokes-conventions}

Given $(E, \theta, h, g) \in \Hfr$, we can consider the classical Stokes data of the corresponding irregular connection $\nabla_\zeta$.\footnote{More specifically, $(E, \nabla_\zeta, g)$ determines a family of compatibly framed \emph{meromorphic} connections $(E_a, \nabla_\zeta, \tau_{a, \zeta})$ for $a \in \R$, where $\tau_{a, \zeta}$ is obtained by an explicit modification of $g$ (again see \cref{sec:classical-stokes-theory}).
All of the relevant Stokes data can be defined in terms of $(E_a, \nabla_\zeta, \tau_{a, \zeta})$, and is independent of the choice of $a$.}
We briefly summarize some of the key points below.

The connection $\nabla_\zeta$ has four anti-Stokes rays $r_1, \dots, r_4$ and four Stokes rays, corresponding in this case to directions in the $w$-plane where $\zeta^{-1} w^{-2}$ is real resp.\ imaginary.
With these naming conventions, the \emph{anti-Stokes} rays are the asymptotic directions of the relevant spectral network (described in \cref{sec:abelianization-and-framing}), and flat sections exchange dominance when crossing a \emph{Stokes} ray.

Let $\Sect_i$ denote the sector bounded by the anti-Stokes rays $r_i$ and $r_{i+1}$, and let $\eSect_i$ denote the extended sector bounded by the adjacent Stokes rays (see \cref{fig:stokes-sectors}). 
There is a canonical way of diagonalizing $\nabla_\zeta$ in each extended sector near $w=0$, which allows us (after choosing a branch of the logarithm) to define a corresponding sectorial frame of flat sections.

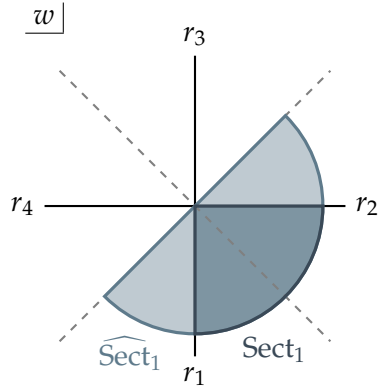
\begin{figure}[ht]
    \centering
    \begin{tikzpicture}[scale=0.9]
    \node at (-2,2.5) {\fbox[rb]{$w$}};

    \draw[fill=slateblue!40] (0,0) -- (45:-1.7) arc(45:225:-1.7) -- cycle;
    \draw[fill=slateblue!80] (0,0) -- (90:-1.7) arc(90:180:-1.7) -- cycle;

    \node[below left, slateblue] at (-0.3, -1.6) {$\eSect_1$};
    \node[below right, darkbluegrey] at (0.5, -1.6) {$\Sect_1$};

    \draw[thick] (-2, 0) -- (2, 0);
    \draw[thick] (0, -2) -- (0, 2);
    \draw[thick, dashed, gray] (-1.8, -1.8) -- (1.8, 1.8);
    \draw[thick, dashed, gray] (-1.8, 1.8) -- (1.8, -1.8);

    \draw[very thick, slateblue] (0,0) -- (45:-1.7) arc(45:225:-1.7) -- cycle;
    \draw[very thick, darkbluegrey] (0,0) -- (90:-1.7) arc(90:180:-1.7) -- cycle;

    \node[black, above] at (0, 2) {$r_3$};
    \node[black, left] at (-2, 0) {$r_4$};
    \node[black, below] at (0, -2) {$r_1$};
    \node[black, right] at (2, 0) {$r_2$};
    \end{tikzpicture}

    \caption{A sector $\Sect_1$ and extended sector $\eSect_1$ in a neighbourhood of $w=0$ (shown here for $\zeta = m < 0$). 
    The solid labelled rays are anti-Stokes rays, and the dotted rays are Stokes rays.}
    \label{fig:stokes-sectors}
\end{figure}

\begin{remark}[Labelling conventions]
    The formulas for the Stokes data depend on the labelling of the rays and the choice of logarithm branch.
    We will follow the same conventions as in \cite[Section 3.4.4]{Tulli:2019}, which depend on the parameters $\zeta, m \in \C^*$.
    In that paper, these choices were essential for determining the jumps of the Stokes data as $\zeta$ varies.
    For us, the exact details will not be as important, except to make sure our constructions match up with \cite{Tulli:2019}.

    Briefly: for $\zeta = m$, we denote the anti-Stokes ray in the direction $w=e^{-\frac{1}{2} i \arg (m)}$ by $r_1$ and number the others counterclockwise.
    As $\zeta \in \C^* \setminus \{ \zeta^{-1} m i < 0\}$ varies from $\zeta=m$, the rays (and choice of logarithm branch) also vary continuously; see \cref{fig:zeta-plane}.
\end{remark}

\begin{figure}[ht]
    \centering
    \includegraphics[scale=1]{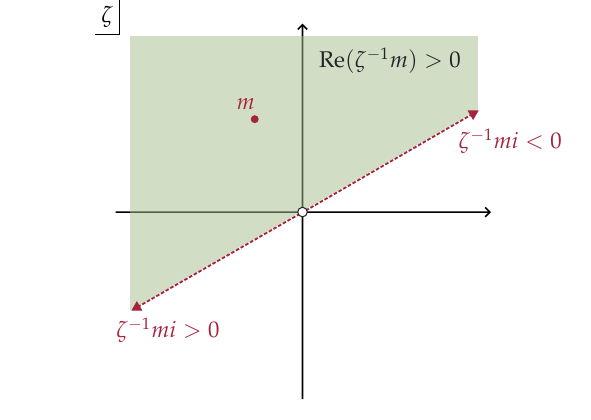}
    \caption{The half-plane $\{\zeta \in \C^*: \Re(\zeta^{-1} m) > 0 \}$. The Ooguri-Vafa magnetic coordinate $\Xm^\ov(\zeta)$, interpreted Stokes-theoretically in \eqref{eq:magnetic-coord} below, has jumps at its boundary rays $\{ \pm \zeta^{-1} m i< 0 \}$.}
    \label{fig:zeta-plane}
\end{figure}

More importantly for our application, the sectorial frames can be described by their asymptotics.

\begin{prop}[Sectorial asymptotics, cf.\ {\cite[Lemmas 3.5 and 3.6]{Tulli:2019}}] \label{prop:sectorial-asymptotics}
    In each extended sector $\eSect_i$ there exists a frame of flat sections $\Phi_i$ uniquely characterized by the following asymptotic condition: 
    
    If we write $\Phi_i = g \cdot M_i$ with respect to the original compatible frame $g$, then the matrix $M_i$ satisfies
    \begin{equation} \label{eq:sectorial-normalization}
        M_i \cdot e^{A_0(w) H} \to 1 \quad \text{as $w \to 0$ in $\eSect_i$},
    \end{equation}
    where 
    \begin{equation} \label{eq:normalization-antiderivative}
        A_0(w) \coloneqq  \frac{1}{2} \zeta^{-1} w^{-2}  +\frac{1}{2} \zeta \overline{w}^{-2}  - (\zeta^{-1} m - \frac{1}{2} m^{(3)}) \log w - (\zeta \overline{m} + \frac{1}{2} m^{(3)}) \log \overline{w}.
    \end{equation}
\end{prop}

\begin{remark} \label{rem:asymptotics-uniqueness}
    In fact, the uniqueness part only uses that the asymptotics \eqref{eq:sectorial-normalization} hold as $w \to 0$ in a neighbourhood of the Stokes ray inside $\Sect_i$ (cf.\ the proof of \cite[Lemma 3.6]{Tulli:2019}).
\end{remark}

The Stokes matrix $S_i$ is defined as the transition matrix from $\Phi_i$ to $\Phi_{i+1}$ on $\eSect_i \cap \eSect_{i+1}$.
For $i=4$, we interpret $\Phi_5 \coloneqq \Phi_1 \cdot M_0$, where
\begin{equation} \label{eq:formal-monodromy}
    M_0 = \exp\left (-2 \pi i (- \zeta^{-1} m + m^{(3)} + \zeta \overline{m} ) H \right)
\end{equation}
is the formal monodromy of $\nabla_\zeta$.

We will write the frames of flat sections in each sector as
\begin{equation} \label{eq:flat-section-frames}
    \Phi_1 = (s_1, s_2), \qquad
    \Phi_2 = (s_3, s_2), \qquad
    \Phi_3 = (s_3, s_4), \qquad
    \Phi_4 = (s_1, s_4),
\end{equation}
so that $s_i$ is exponentially decreasing along the $i$th anti-Stokes ray $r_i$ (see \cref{fig:stokes-sectors-sections} and cf.\ the spectral networks in \cref{fig:sn-labelled-sections}).
Then the Stokes matrices $S_1, S_3$ are lower-triangular and $S_2, S_4$ are upper-triangular.

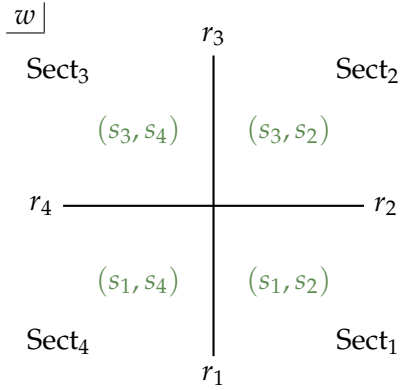
\begin{figure}[ht]
    \centering
    \begin{tikzpicture}[scale=0.9]
    \node at (-2.5,2.5) {\fbox[rb]{$w$}};

    \node[above left] at (-1.5, 1.5) {$\Sect_3$};
    \node[below left] at (-1.5, -1.5) {$\Sect_4$};
    \node[below right] at (1.5, -1.5) {$\Sect_1$};
    \node[above right] at (1.5, 1.5) {$\Sect_2$};

    \draw[thick] (-2, 0) -- (2, 0);
    \draw[thick] (0, -2) -- (0, 2);

    \node[black, above] at (0, 2) {$r_3$};
    \node[black, left] at (-2, 0) {$r_4$};
    \node[black, below] at (0, -2) {$r_1$};
    \node[black, right] at (2, 0) {$r_2$};

    \node[green] at (-1,1) {$(s_3, s_4)$};
    \node[green] at (-1,-1) {$(s_1, s_4)$};
    \node[green] at (1,-1) {$(s_1, s_2)$};
    \node[green] at (1,1) {$(s_3, s_2)$};
    \end{tikzpicture}

    \caption{Labelling for the frames of flat sections $\Phi_i$ in each sector $\Sect_i$.}
    \label{fig:stokes-sectors-sections}
\end{figure}
In particular, we will denote
\begin{equation} \label{eq:stokes-matrices}
    S_1 = \begin{pmatrix}
        1 & 0 \\
        a & 1
    \end{pmatrix} \quad \text{and} \quad 
    S_2 = \begin{pmatrix}
        1 & b \\
        0 & 1
    \end{pmatrix}
\end{equation}
so that we have the relations
\begin{equation}
    s_3 = s_1 + a s_2 \quad \text{and} \quad 
    s_4 = s_2 + b s_3,
\end{equation}
which we can rewrite as
\begin{equation} \label{eq:a-b-ratios}
     a = \frac{s_3 \wedge s_1}{s_2 \wedge s_1} \quad \text{and} \quad 
    b = \frac{s_4 \wedge s_2}{s_3 \wedge s_2}.
\end{equation}

We emphasize that everything described above depends on $\zeta \in \C^*$, including the sectors, frames of flat sections, and Stokes matrix elements $a = a(\zeta)$ and $b = b(\zeta)$.

\subsection{Correspondence with the Ooguri-Vafa space} 
\label{sub:ov-correspondence}

Here we will summarize the main aspects of the correspondence between $\Xfr$ and the Ooguri-Vafa space from \cite{Tulli:2019}.

Let $\Mov(\Lambda)$ denote the Ooguri-Vafa space with cutoff $\Lambda \in \C^*$. 
As described in \cref{ssub:summary-ov-space}, $\Mov(\Lambda)$ is a singular torus fibration over a neighbourhood $\B_\Lambda$ of the origin in $\C$, the size of which depends on $|\Lambda|$.
The fibres over nonzero $z \in \B_\Lambda$ are tori parametrized by an electric angle $\theta_e$ and magnetic angle $\theta_m$.
The central fibre over $z=0$ is a torus with a node, where the circle corresponding to $\theta_m$ degenerates to a point (again see \cref{fig:ov-space}).

We will describe the hyperkähler structure of $\Mov(\Lambda)$ via the electric and magnetic twistor coordinates $\Xe^\ov, \Xm^\ov$ \cite{Gaiotto:2010} for its holomorphic symplectic form 
\begin{equation} \label{eq:ov-twistor-coords}
    \Omega^\ov_\zeta = -\frac{1}{4 \pi^2} d \log \Xe^\ov(\zeta) \wedge d \log \Xm^\ov(\zeta), \quad \zeta \in \C^*.
\end{equation}
Explicitly, $\Xe^\ov$ is given by
\begin{equation} \label{eq:ov-electric-coord}
    \Xe^\ov(\zeta) =  \exp \left(\zeta^{-1} \pi z + i \theta_e + \pi \zeta \overline{z}\right),
\end{equation}
and $\Xm^\ov$ has the semiflat approximation
\begin{equation} \label{eq:sf-ov-magnetic-coord}
    \Xm^{\ov, \sf}(\zeta) = \exp \left( \zeta^{-1} \frac{z \log (z/\Lambda) - z}{2i} + i \theta_m - \zeta \frac{\overline{z} \log (\overline{z} / \overline{\Lambda}) - \overline{z}}{2i} \right).
\end{equation}
(These coordinates define the \emph{semiflat Ooguri-Vafa form}
\begin{equation}
    \Omega^{\ov, \sf}_\zeta = -\frac{1}{4 \pi^2} d \log \Xe^\ov(\zeta) \wedge d \log \Xm^{\ov, \sf}(\zeta),
\end{equation}
which we will study further starting in \cref{ssub:sf-ov-form}.)
The full magnetic coordinate $\Xm^\ov$ is given by an ``instanton correction'' $\Xm^\ov = \Xm^{\ov, \sf} \Xm^{\ov, \inst}$, involving an integral formula which we will not need to explicitly describe here.

Let $\Xfr(\Lambda) \subseteq \Xfr$ consist of the isomorphism classes of framed Higgs bundles in $\Hfr$ for which $-2 i m \in \B_\Lambda$.
\cite{Tulli:2019} identifies the Ooguri-Vafa space $\Mov(\Lambda = 4i)$ with $\Xfr(\Lambda=4i)$ via the correspondence in \cref{tab:dictionary}.

\begin{table}[ht] 
    \renewcommand{\arraystretch}{1.3}
    \begin{tabular}{c|c} 
        $\Mov(4i)$ & $\Xfr(4i)$ \\
        \hline 
        $z$ & $-2im$ \\ 
        $\theta_e$ & $2 \pi m^{(3)}$\\
        $\theta_m$ & [see \eqref{eq:magnetic-angle}]
    \end{tabular}
    \caption{Partial dictionary between $\Mov$ and $\Xfr$ parameters from \cite{Tulli:2019}. 
    The  formula for $\theta_m$ involves slightly more technical setup, so we postpone it to \cref{ssub:sf-ov-form}. 
    Until then, we will work directly with the twistor coordinate $\Xm$ instead.}
    \label{tab:dictionary}
\end{table}

This correspondence makes $\Xfr(4i)$ into a hyperkähler space, with twistor coordinates given by
\begin{equation} \label{eq:electric-coord}
    \Xe^\ov(\zeta) \leftrightarrow \Xe(\zeta) = \exp \left( -2 \pi i (\zeta^{-1} m - m^{(3)} - \overline{m} \zeta) \right)
\end{equation}
and
\begin{equation} \label{eq:magnetic-coord}
    \Xm^\ov(\zeta) \leftrightarrow \Xm(\zeta) = \begin{cases}
        a(\zeta) & \text{if }   \Re(\zeta^{-1} m) > 0, \\
        -1/b(\zeta) & \text{if } \Re(\zeta^{-1} m) < 0, 
    \end{cases}
\end{equation}
where $a(\zeta)$ and $b(\zeta)$ are the Stokes matrices elements associated to $\nabla_\zeta$ as in \cref{ssub:stokes-conventions}.
Note that $\Xe$ also has a Stokes-theoretic interpretation, as one of the diagonal entries of the formal monodromy \eqref{eq:formal-monodromy} of $\nabla_\zeta$.

\begin{remark}
    Although the above identification for $\Xe$ can be seen directly from \eqref{eq:ov-electric-coord} using \cref{tab:dictionary}, the identification for $\Xm$ is very nontrivial: it was proved in \cite{Tulli:2019} by matching up the asymptotics and jumps of the Stokes data (as $\zeta \to 0, \infty$ and at $\Re(\zeta^{-1} m) = 0$, respectively)  with those of the Ooguri-Vafa magnetic coordinate.
\end{remark}

Let us introduce some notation to spell out the identifications \eqref{eq:electric-coord} and \eqref{eq:magnetic-coord} a bit more carefully.
We will think of $\Xe$ and $\Xm$ (without superscripts) as functions on the space $\Afr_\zeta$, where they assign to each connection its Stokes data as described above.
These functions define a corresponding form
\begin{equation}
    \Omega^\ov_\Stokes \coloneqq -\frac{1}{4 \pi^2} d \log \Xe \wedge d \log \Xm
\end{equation}
on $\Afr_\zeta$.
The pulled-back functions $(\NAH_\zeta)^* \Xe$ and $(\NAH_\zeta)^* \Xm$ on $\Hfr$ are isomorphism invariant and hence descend to the moduli space $\Xfr$, where they coincide with the corresponding Ooguri-Vafa coordinates $\Xe^\ov$ and $\Xm^\ov$ under the identification $\Xfr \cong \Mfr$.
Consequently we can write
\begin{equation} \label{eq:ov-stokes}
    \Omega^\ov_\zeta = (\NAH_\zeta)^* \Omega^\ov_\Stokes
\end{equation}
for the induced form on $\Xfr \cong \Mov$, as in \eqref{eq:forms-cd}.

\section{Abelianization and framing} 
\label{sec:abelianization-and-framing}

In the last section we established our setup in terms of framed flat $\SL(2)$-connections $(E, \nabla_\zeta, g)$ on the base curve $C = \CP^1$.
In this section we explain how to \emph{abelianize} this data; that is, how to lift the relevant structures (particularly the framing) to corresponding rank $1$ objects on the spectral cover $\Sigma$ of $C$. 

\subsection{Review of spectral networks and abelianization} 
\label{sub:spectral-networks-review}

We start by briefly recalling some of the main definitions and constructions involving spectral networks \cite{Gaiotto:2013}, especially pertaining to the abelianization of flat connections.
We mostly follow the approach and exposition of \cite{Hollands:2016}; see also \cite{Gaiotto:2013a, Hollands:2021} for additional discussion of the relevant irregular singularity case.
Many of the definitions below admit generalizations (e.g.\ to higher rank), but we will just describe what is needed for our application.

\subsubsection{WKB spectral networks (in rank 2)} 
\label{ssub:wkb-spectral-networks}

Fix a compact Riemann surface $C$ and a meromorphic quadratic differential $\phi_2$ on $C$ (i.e.\ a meromorphic section of $K_C^{\otimes 2}$, the square of the canonical bundle).

Assume that all of the zeros of $\phi_2$ are simple, and that $\phi_2$ has at least one pole.
We will refer to the poles $p_i$ of $\phi_2$ as \emph{punctures} of $C$.

\begin{definition}[Spectral curve]
    The \emph{spectral curve} defined by $\phi_2$ is
    \begin{equation}
        \Sigma_{\phi_2} \coloneqq \{ \lambda \in T^* C : \lambda^2 - \phi_2 = 0  \} \subseteq T^*C.
    \end{equation}
\end{definition}
Note that $\Sigma = \Sigma_{\phi_2}$ is smooth since $\phi_2$ has only simple zeros. 
The projection $\pi: \Sigma \to C$ is a double cover branched at the zeros of $\phi_2$, and $\Sigma$ has punctures lying over the $p_i$.

Fix a phase $\vartheta \in \R/2\pi\Z$.

\begin{definition}[$\vartheta$-trajectories] \label{def:theta-traj}
    A \emph{$\vartheta$-trajectory} of $\phi_2$ is a curve $\gamma$ on $C$ such that 
    \begin{equation} \label{eq:theta-traj-condition}
        e^{-2 i \vartheta} \phi_2(v^2) \in \R_{> 0}
    \end{equation}
    for all nonzero tangent vectors $v$ along $\gamma$.
\end{definition}

The $\vartheta$-trajectories constitute the leaves of a singular foliation $\mathcal{F}_{\phi_2}^\vartheta$ on $C$.
(For $\vartheta = 0$ this is the classical ``horizontal foliation'' defined by $\phi_2$.)
Both endpoints of a generic trajectory are at punctures of $C$.
In particular, around each pole of order $k > 2$, each trajectory is asymptotic to one of $k-2$ distinguished tangent directions.
We will call these \emph{anti-Stokes directions}, for reasons to be justified below.

A trajectory is called \emph{critical} if (at least) one of its endpoints is at a branch point.
There are three critical trajectories emanating from each branch point (see \cref{fig:branch-point}).
The union of the critical leaves is the \emph{critical graph} $\CG(\phi_2, \vartheta)$.
In the current setting, the relevant spectral network $\W(\phi_2, \vartheta)$ is given by $\CG(\phi_2, \vartheta)$ along with some additional labels.

\begin{figure}[ht]
    \centering
    \includegraphics[scale=1]{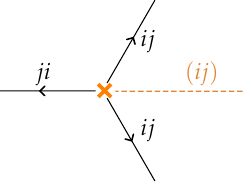}
    \caption{The three critical leaves emanating from a branch point, labelled as in \cref{def:spectral-network}. 
    The dashed orange line denotes a branch cut.}
    \label{fig:branch-point}
\end{figure}

\begin{definition}[Spectral network] \label{def:spectral-network}
    The \emph{(WKB) spectral network} $\W(\phi_2, \vartheta)$ is the following collection of oriented labelled paths on $C$, called \emph{walls}:
    \begin{itemize}
        \item Each leaf of $\CG(\phi_2, \vartheta)$ is a wall, oriented away from the branch point.
        (In particular, the network can contain ``double walls'' arising from a critical trajectory which has both of its endpoints on branch points, i.e.\ a saddle connection.)

        \item Each wall $w$ is labelled with an ordering of the sheets of $\Sigma$ over $w$.
        After choosing branch cuts on $C$ and labelling the sheets of $\Sigma$ by $k=1,2$, we write the label as either ``$12$'' or ``$21$'', determined as follows:
        \begin{itemize}
            \item The two sheets of $\Sigma$ correspond to the two square roots $\lambda_k$ of $\phi_2$.
        If $v$ is a positively oriented tangent vector along $w$, then by \eqref{eq:theta-traj-condition},
        \begin{equation}
            e^{-i \vartheta} \lambda_k (v) \in \R_{\neq 0}.
        \end{equation}
        \item Let $k_\pm$ denote the appropriate index so that $e^{-i \vartheta} \lambda_{k_\pm}(v) \in \R_\pm$.
        Then label $w$ by ``$k_-k_+$''.
        \end{itemize}
    \end{itemize}
\end{definition}

\begin{example} \label{ex:sn-quadratic-differential}
    We will be interested in the quadratic differential $\phi_2 = (z^2+2m) dz^2$ on $C = \CP^1$, with $m \in \C^*$.
    Three examples of the corresponding network $\W(\phi_2, \vartheta)$ (when $m < 0$) are shown in \cref{fig:sn-saddle-examples}, plotted using the Mathematica notebook \texttt{swn-plotter.nb} \cite{Neitzke:swn}.\footnote{Due to differing sign conventions, these are obtained by using the quadratic differential  $-\phi_2$ in the code.}

    Note that $\phi_2$ has a pole of order $k=6$ at $z=\infty$, where there are $4$ anti-Stokes directions to which trajectories are asymptotic.
    The network has a double wall at the phases $\vartheta = \arg(m) + \pi/2$ and $\vartheta = \arg(m) + 3\pi/2$.

    \begin{figure}[ht]
        \subcaptionbox{$\vartheta = (\pi/2)^-$}
            {\includegraphics[scale=1]{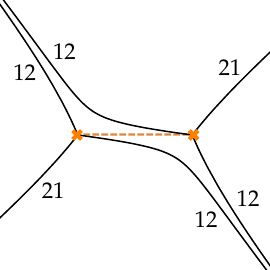}}
        \hspace{20pt}
        \subcaptionbox{$\vartheta = \pi/2$}[0.3\textwidth]
            {\includegraphics[scale=1]{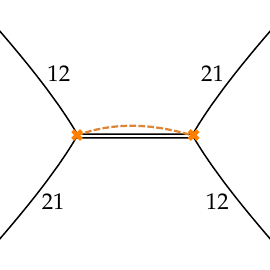}}
        \hspace{20pt}
        \subcaptionbox{$\vartheta = (\pi/2)^+$}
            {\includegraphics[scale=1]{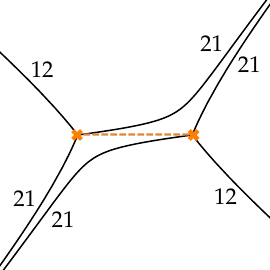}}
        \caption{The spectral network $\W(\phi_2, \vartheta)$ with $\phi_2 = (z^2 + 2m) dz^2$, shown here for $m < 0$ and for phases $\vartheta$  before, at, and after the ``critical phase'' $\vartheta_c = \arg(m) + \pi/2$.
        (The orientations of the walls are omitted.)}
        \label{fig:sn-saddle-examples}
    \end{figure}
\end{example}

\subsubsection{Spectral networks and nonabelian Hodge} 
\label{ssub:networks-and-NAH}

Given a (possibly singular) $\SL(2)$-Higgs bundle $(E, \theta)$ with harmonic metric $h$, we can consider the corresponding one-parameter family of flat connections
\begin{equation*}
    \nabla_\zeta = \zeta^{-1} \theta + D_h + \zeta \theta^{\dagger_h}, \quad \zeta \in \C^*.
\end{equation*}

We will associate to $(E, \theta, h)$ the spectral network 
\begin{equation}
    \W_\zeta(\theta) \coloneqq \W(\phi_2, \vartheta)
\end{equation}
where
\begin{equation}
    \begin{cases}
        \phi_2 = - \det \theta,   \\
        \vartheta = \arg(\zeta).
    \end{cases}
\end{equation}
With these choices:
\begin{itemize}
    \item The spectral curve $\Sigma_{\phi_2}$ coincides with the Higgs spectral curve
    \begin{equation}
        \Sigma_\theta \coloneqq \{ \lambda \in T^*C : \det (\theta - \lambda I) = 0 \} \subseteq T^*C.
    \end{equation}

    \item The anti-Stokes rays of the network at each puncture coincide with the anti-Stokes rays of the connection $\nabla_\zeta$ \cite{Gaiotto:2013a}.
\end{itemize}
(This choice of network is also relevant for studying the WKB asymptotics of $\nabla_\zeta$ as $\zeta \to 0$, as discussed in \cite{Gaiotto:2013, Gaiotto:2013a}.)

\begin{example}
    For $(E, \theta, g) \in \Hfr$, the associated quadratic differential is $\phi_2 = (z^2 + 2m)  dz^2$, so the corresponding spectral network $\W = \W_\zeta(\theta)$ is one of those described in \cref{ex:sn-quadratic-differential}, i.e.\ it topologically looks like one of the networks in \cref{fig:sn-saddle-examples}.
\end{example}

\subsubsection{Abelianization} 
\label{ssub:abelianization}

Spectral networks can be used to lift nonabelian connections on the base surface to abelian connections on its cover.
We continue to follow the presentation of \cite{Hollands:2016}.
As above, let $C$ be a compact Riemann surface with a spectral cover $\pi: \Sigma \to C$ and a spectral network $\W$ (both arising from the same quadratic differential $\phi_2$).
For simplicity we assume that $\W$ has no double walls.
Let $\Sigma'$ denote $\Sigma$ with the branch points removed.

\begin{definition}[Abelianization via $\W$-pairs]\label{def:w-pair}
    A \emph{$\W$-pair} is a tuple $(E, \nabla; \L, \nabla^\ab;  \iota)$ where:
    \begin{itemize}
        \item $(E, \nabla)$ is a flat $\SL(2)$-bundle over $C$,

        \item $(\L, \nabla^\ab)$ is a flat $\C^*$-bundle over $\Sigma'$, and

        \item $\iota: E|_{C \setminus \W} \xrightarrow{\sim} \pi_* \L|_{C \setminus \W}$ is an isomorphism,
    \end{itemize}
    such that
    \begin{enumerate}[(i)]
        \item $\iota$ takes $\nabla$ to $\pi_* \nabla^\ab$,
        \item \label{item:unipotent-jumps} at a wall $w$ of type $ij$, $\iota$ jumps by an automorphism
        \begin{equation}
             S_w = \bbid + e_w
         \end{equation}
        of $\pi_* \L \cong \L_1 \oplus \L_2$, where $e_w: \L_i \to \L_j$ (and the subscripts indicate the corresponding sheets of $\Sigma$).
    \end{enumerate}
    In this case we call $(\L, \nabla^\ab)$ an \emph{abelianization} of $(E, \nabla)$, and $(E, \nabla)$ a \emph{nonabelianization} of $(\L, \nabla^\ab)$.
\end{definition}

More concretely, finding an abelianization of $(E, \nabla)$ amounts to specifying a basis of sections $(s_1, s_2)$ in each cell of $C \setminus \W$ which ``diagonalize $\nabla$'' (that is, have $\nabla s_i = d_i s_i$ for some closed $1$-forms $d_i$) and have appropriate jumps at the walls of $\W$.
For instance, condition \ref{item:unipotent-jumps} says that on the two sides of a wall $w$ of type $21$, there are bases $(s_1, s_2)$ and $(s_1', s_2')$ related by
\begin{equation}
    s_1' = s_1 \quad \text{and} \quad s_2' = s_2 + \alpha s_1
\end{equation}
for some function $\alpha$, i.e.\ the jump $S_w$ is unipotent and upper-triangular with respect to the trivialization $(s_1, s_2)$.

One consequence of the form of these jumps is that any abelianized connection $\nabla^\ab$ has holonomy $-1$ around the branch points of $\Sigma$; such a connection is called \emph{almost-flat}.
Another consequence is that any connection $\nabla^\ab$ satisfying \cref{def:w-pair} but only defined over $\Sigma' \setminus \pi^{-1}(\W)$ automatically extends over the walls $\pi^{-1}(\W)$.
Therefore to construct an abelianization of $(E, \nabla)$ it suffices to provide an appropriate basis of $E$ in each cell of $C \setminus \W$.

\subsubsection{Constructions and computations} 
\label{ssub:constructions-and-computations}

Given $(E, \nabla)$ and $\W$, the choice of an abelianization $(\L, \nabla^\ab)$ is generally not unique.
This can be rectified by specifying an additional decoration called a ``$\W$-framing''. 
Roughly this is a choice of $\nabla$-invariant line subbundle at each puncture.
We will explain how this works for irregular singularities, which is the only case we will encounter.

\begin{construction}[$\W$-framing with irregular punctures, cf.\ {\cite[Section 8]{Gaiotto:2013a}}]
    \begin{enumerate}
    \item[]
    \item Consider an infinitesimal circle around each singularity $p$ of order $k > 2$, with $k-2$ marked points $q_i$ corresponding to the anti-Stokes directions of the foliation at $p$.\footnote{More precisely we could define the marked points on the real blow-up of $C$ at $p$, but we will be informal---this is ultimately just a labelling procedure.}
    Each wall ending at $p$ is asymptotic to one of the directions; we think of it as ending at the corresponding marked point. (See \cref{fig:w-framing-circle}.)

    \begin{figure}[ht]
        \centering
        \includegraphics[scale=1]{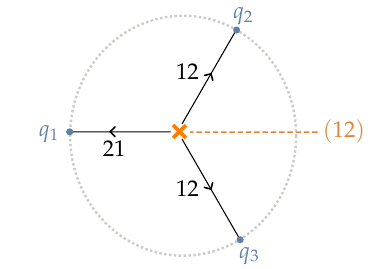}
        \caption{Marked points for a pole of order $k=5$ at $p = \infty$. The dotted circle should be thought of as an infinitesimal circle around $\infty$.}
        \label{fig:w-framing-circle}
    \end{figure}

    \item Choose a $\nabla$-invariant line subbundle $\ell_i$ of $E$ near each marked point\footnote{or regular puncture, if there are any} $q_i$.
    At any point $z \in C$ that can be joined to $q_i$ without crossing any walls, let $\ell_i(z)$ denote the parallel transport of $\ell_i$ to $E_z$.
    Assume that if two points $q_i$ and $q_j$ can be connected to a common point $z$ without crossing any walls, then $\ell_i(z) \neq \ell_j(z)$ (this holds generically).

    \item Choose sections $s_i$ such that $s_i(z) \in \ell_i(z)$.

    \item Each cell (without branch cuts) has trajectories of type $21$ going into one point $q_i$, and trajectories of type $12$ going into some $q_j$.
    Assign the basis $(s_i, s_j)$ to this cell, as shown in \cref{fig:w-framing-cell}.

    \begin{figure}[ht]
    \centering
    \includegraphics[scale=1]{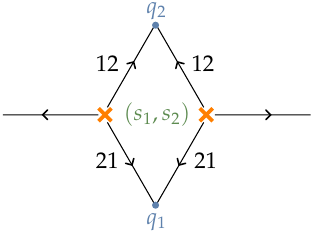}
    \caption{$\W$-framing in a single cell.
    We make a choice of line subbundles $\ell_i$ near each marked point $q_i$, and the basis $(s_1, s_2)$ is chosen so that $s_i(z) \in \ell_i(z)$.}
    \label{fig:w-framing-cell}
    \end{figure}

    \item[($4'$)] If there are branch cuts, the notation is adjusted in the natural way (see \cref{fig:w-framing-one-branch-point}).

    \begin{figure}[ht]
        \centering
        \includegraphics[scale=1]{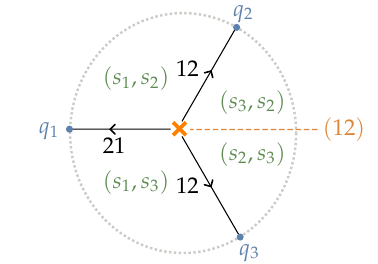}
        \caption{$\W$-framing for the network from \cref{fig:w-framing-circle}. 
        The order of the labelled sections is swapped upon crossing a branch cut.}
        \label{fig:w-framing-one-branch-point}
    \end{figure}
\end{enumerate}
\end{construction}

We can omit the labels $12$ and $21$ once we have specified a $\W$-framing, since they are determined by the labelling conventions.

The connection $\nabla$ is diagonal with respect to the basis $(s_i, s_j)$ in each cell.
By construction, it can be viewed as the pushforward of a connection $\nabla^\ab$ on the line bundle $\L$ defined by $\ell_i$ and $\ell_j$ on the two sheets of $\Sigma$.
By appropriately rescaling the sections
\begin{equation*}
    s_i(z) \to  \tilde{s}_i(z) \coloneqq c_i(z) s_i(z)
\end{equation*}
we can ensure they have the required unipotent jumps across walls to define $\iota$, but we can compute parallel transport even without rescaling.

\begin{lemma}[Parallel transport formulas] \label{lem:ab-parallel-transport}
    Suppose $(s_i, s_j)$ and $(s_k, s_j)$ are sections diagonalizing $\nabla$ on either side of a wall $w$ of type $12$, and view them as sections of the abelianized line bundle $\L$.

    Then the $\nabla^\ab$-parallel transports of $s_i$ and $s_j$ across $w$ are given by
    \begin{equation}
        s_i \mapsto \frac{s_i \wedge s_j}{s_k \wedge s_j} s_k
    \end{equation}
    and
    \begin{equation}
        s_j \mapsto s_j.
    \end{equation}
\end{lemma}
\begin{proof}
    From the triangular form of the jumps; see \cite[Section 8.3]{Hollands:2016}.
\end{proof}

As a useful mnemonic, note that the section $s_i$ associated to a puncture $q_i$ does not jump when crossing a wall going into $q_i$.

\begin{remark}[Canonical decoration]
    When $\W = \W_\zeta(\theta)$ corresponds to an irregular connection $\nabla_\zeta$, there is a canonical ``small flat section'' $s_i$ near each $q_i$ which exponentially decays along the corresponding anti-Stokes ray \cite{Gaiotto:2013a}.
    We will use these sections as our decorations.
\end{remark}

\begin{example}
    The canonical $\W$-framing for a spectral network coming from $(E, \theta, g) \in \Hfr$ is shown in \cref{fig:sn-labelled-sections}, depending on the values of $m$ and $\zeta$. 
    (The network has a double wall if $\Re(\zeta^{-1} m) = 0$, in which case it can be resolved into either of the ones shown.)

    \begin{figure}[ht]
        \subcaptionbox{$\Re(\zeta^{-1} m) > 0$}
            {\includegraphics[scale=1]{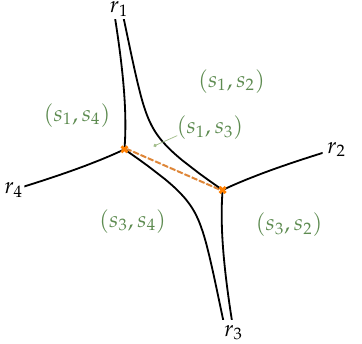}}
        \hspace{1cm}
        \subcaptionbox{$\Re(\zeta^{-1} m) < 0$}
            {\includegraphics[scale=1]{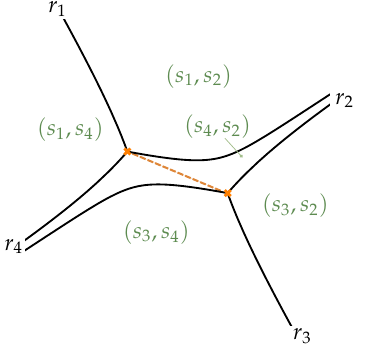}}
        \caption{$\W$-framings for the two generic topologies of a spectral network $\W_\zeta$ coming from $(E, \theta, g) \in \Hfr$.
        The canonical flat sections used for abelianization in each cell are labelled in green.}
        \label{fig:sn-labelled-sections}
    \end{figure}
\end{example}

\subsection{Framing near the punctures} 
\label{sub:framing-near-the-punctures}

Let $C = \CP^1$. 
Given a Higgs bundle $(E, \theta) \in \H$ with parameter $m \neq 0$, let 
\begin{equation}
    \Sigma \coloneqq \{ \lambda \in T^*C : \det (\theta - \lambda I) = 0 \} \subseteq T^*C
\end{equation}
be its spectral curve, with projection $\pi: \Sigma \to C$.
More explicitly,
\begin{align}
    \Sigma &= \{ \lambda \in T^*C : \lambda^2 + \det \theta  = 0 \} \\
        &= \{ (z \in C,\, s \in T_z^*C) : s^2 - (z^2 + 2m) dz^2 = 0 \}. \label{eq:spectral-curve-explicit}
\end{align}
In the notation of \cref{ssub:wkb-spectral-networks}, $\Sigma = \Sigma_{\phi_2}$ for the quadratic differential $\phi_2 = - \det \theta = (z^2 + 2m) dz^2$.

Note that $\Sigma$ is a branched double cover of $\CP^1$ with two branch points at $z = \pm \sqrt{-2m}$ and no ramification at $z =\infty$; consequently it has genus zero with two punctures lying over $\infty$.
Let $\Sigma'$ denote $\Sigma$ with the two branch points removed.

\begin{remark}[Punctures and compactification]
    For notational purposes, we will refer to the two punctures of $\Sigma$ as $\infty_-$ and $\infty_+$.
    We will not work directly with a compactification $\overline{\Sigma} = \Sigma \cup \{\infty_\pm\}$, however; when we say ``near $\infty_\pm$'', we just mean in the corresponding punctured neighbourhood on $\Sigma$.
\end{remark}

Fix $\zeta \in \C^*$ and consider the flat connection $(E, \nabla_\zeta)$ with associated spectral network $\W = \W_\zeta(\theta)$, as shown in \cref{fig:sn-labelled-sections}. 
The abelianization procedure uses $\W$ to define a rank 1 connection $(\L^\ab_\zeta, \nabla_\zeta^\ab) = \ab (E, \nabla_\zeta)$ over the  spectral cover $\Sigma$.
The resulting connection $\nabla^\ab_\zeta$ is almost-flat, i.e.\ it is flat on $\Sigma'$ but has holonomy $-1$ around the branch points $b_\pm = (\pm \sqrt{-2m}, 0) \in \Sigma$.

Our first goal is to extend the abelianization correspondence to framed bundles, that is, to explain how abelianization of a framed connection $(E, \nabla, g)$ induces a framing upstairs for $(\L^\ab_\zeta, \nabla_\zeta^\ab)$ near the punctures $\infty_\pm$.
We will show that there exists a natural frame $g^\ab_{\zeta, \pm}$ with respect to which $\nabla^\ab_\zeta$ is given by the diagonal entries of the singular part of $\nabla_\zeta$, so that $(\L_\zeta, \nabla^\ab_\zeta, g^\ab_{\zeta, \pm})$ belongs to the following space of connections. 

\begin{definition}[Framed abelian bundles in $\abAfr_\zeta$]\label{def:ab-framed-connections}
    For fixed $\zeta \in \C^*$, let $\abAfr_\zeta$ denote the set of framed almost-flat bundles $(\L, \nabla^\ab, g^\ab_\pm)$ where:
    \begin{itemize}    
        \item $\L$ is a holomorphic line bundle over $\Sigma$.
    
        \item $\nabla^\ab$ is an almost-flat (complex) connection on $\L$ with irregular singularities at $\infty_\pm$, of the form 
        \begin{equation} \label{eq:ab-framed-form}
            \nabla^\ab = d \pm \pi^* \left[ \zeta^{-1} \frac{dw}{w^3} + \zeta \frac{d \overline{w}}{\overline{w}^3} + (\zeta^{-1} m - \frac{1}{2} m^{(3)}) \frac{dw}{w} + (\zeta \overline{m} + \frac{1}{2} m^{(3)})\frac{d \overline{w}}{\overline{w}}   \right]
        \end{equation}
         with respect to the framing $g^\ab_\pm$ near $\infty_\pm$.
    \end{itemize}
    Say that $(\L, \nabla^\ab, g^\ab_\pm) \cong (\L', \nabla^{\ab,\prime}, g^{\ab, \prime}_\pm)$ if there is a bundle isomorphism $\L \xrightarrow{\sim} \L'$ preserving the additional structure, and let $\abMfr_\zeta$ denote the set of isomorphism classes of $\abAfr_\zeta$.
\end{definition}

\begin{remark}[Rigidity of framed form] \label{rem:abelian-form-rigidity}
    This definition is more rigid than \cref{def:framed-connections} for the space $\Afr_\zeta$ of nonabelian framed connections, in the following sense:
    \begin{itemize}
        \item For $\Afr_\zeta$, the framed form \eqref{eq:framed-connection-form} of the connections $\nabla$ specified the singular terms, but allowed for unspecified regular terms.
        A gauge transformation approaching the identity near the puncture could still modify these regular terms.

        \item For $\abAfr_\zeta$, we are \emph{fully specifying} the form of $\nabla^\ab$, i.e.\ specifying the singular terms and requiring that there are \emph{no additional regular terms}. 
        A gauge transformation preserving the framed form must therefore be identically equal to $1$ near the punctures.
    \end{itemize}
\end{remark}
The abelianization construction requires a specification of flat sections $(s_i, s_j)$ in each cell of $\W$ (again see \cref{fig:sn-labelled-sections}).
More generally we can look for a frame $(f_i, f_j)$ in each cell with respect to which $\nabla$ is diagonal.
We will use these formulations interchangeably: given a frame $(f_i, f_j)$ with respect to which $\nabla$ is of the form
\begin{equation}
    \nabla = d + dQ + \Lambda \frac{dw}{w} + \Lambda' \frac{d \overline{w}}{\overline{w}}
\end{equation}
where $Q, \Lambda, \Lambda'$ are diagonal,
a corresponding frame of flat sections is given by
\begin{equation}\label{eq:sections-vs-frames}
    (s_i, s_j) = (f_i, f_j) \cdot w^{-\Lambda} \overline{w}^{- \Lambda'} e^{-Q},
\end{equation}
and vice versa. 

For now we will restrict our focus to a neighbourhood of the puncture.
We assume for notational concreteness that $\Re(\zeta^{-1} m) > 0$, but the argument for the other case is the same.
Label the ``big cells'' near $w=0$ by $\mathcal{C}_1, \dots, \mathcal{C}_4$, and label the two components of the ``small cell'' by $\mathcal{C}_{s,1}$ and $\mathcal{C}_{s,3}$, as shown in \cref{subfig:sn-cell-components}.
We can choose the frames in each cell using a slight modification of the classical Stokes theory.

\begin{figure}[ht]
    \centering
    \subcaptionbox{Cells near $z=\infty$\label{subfig:sn-cell-components}}
        {\includegraphics[scale=1]{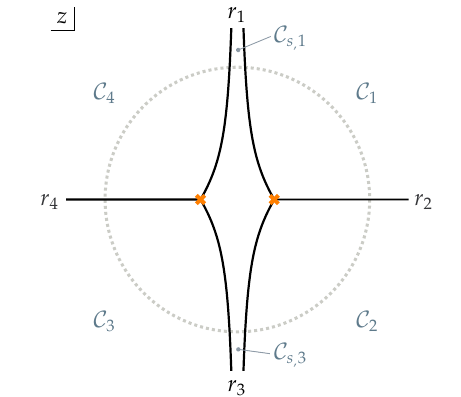}}
    \hspace{1cm}
    \subcaptionbox{Choice of frames near $w=0$\label{subfig:sn-w-frames}}
    {\begin{tikzpicture}[scale=1]
      \draw[thick] (-2, 0) -- (2, 0);
      \draw[dashed, gray!50] (0, -3) -- (0, 3);

      \draw[thick] (1,2.5) arc(120:240:2 and 3);
      \draw[thick] (-1, 2.5) arc(60:-60:2 and 3);

      \node[black] at (0, 3.1) {$r_3$};
      \node[black, left] at (-2, 0) {$r_4$};
      \node[black] at (0, -3.1) {$r_1$};
      \node[black, right] at (2, 0) {$r_2$};

      \node[green] at (-1,0.5) {$(f_3, f_4)$};
      \node[green] at (-1,-1) {$(f_1, f_4)$};
      \node[green] at (1,-1) {$(f_1, f_2)$};
      \node[green] at (1,0.5) {$(f_3, f_2)$};
      \node[green] at (0, 2.15) {$(f_3, {\color{red} f_1})$};
      \node[green] at (0, -2.5) {$(f_1, {\color{red} f_3})$};

      \node[slateblue] at (-1,1) {$\mathcal{C}_3$};
      \node[slateblue] at (-1,-0.5) {$\mathcal{C}_4$};
      \node[slateblue] at (1,-0.5) {$\mathcal{C}_1$};
      \node[slateblue] at (1,1) {$\mathcal{C}_2$};
      \node[slateblue] at (0, 2.65) {$\mathcal{C}_{s,3}$};
      \node[slateblue] at (0, -2) {$\mathcal{C}_{s,1}$};

      \node at (-2,3) {\fbox[rb]{$w$}
      };
    \end{tikzpicture}
    }

    \caption{The spectral network $\W$ in a neighbourhood of $w=0$, shown here for $\Re(\zeta^{-1} m) > 0$.
    The walls are asymptotic to the anti-Stokes rays, which are labelled as in \cref{fig:stokes-sectors}.
    The indicated frames $(f_i, f_j)$ diagonalize $\nabla$ in each cell (cf.\ \cref{fig:sn-labelled-sections}); those labelled in green can be obtained by restricting the sectorial frames from the classical theory.}
    \label{fig:sn-components}
\end{figure}
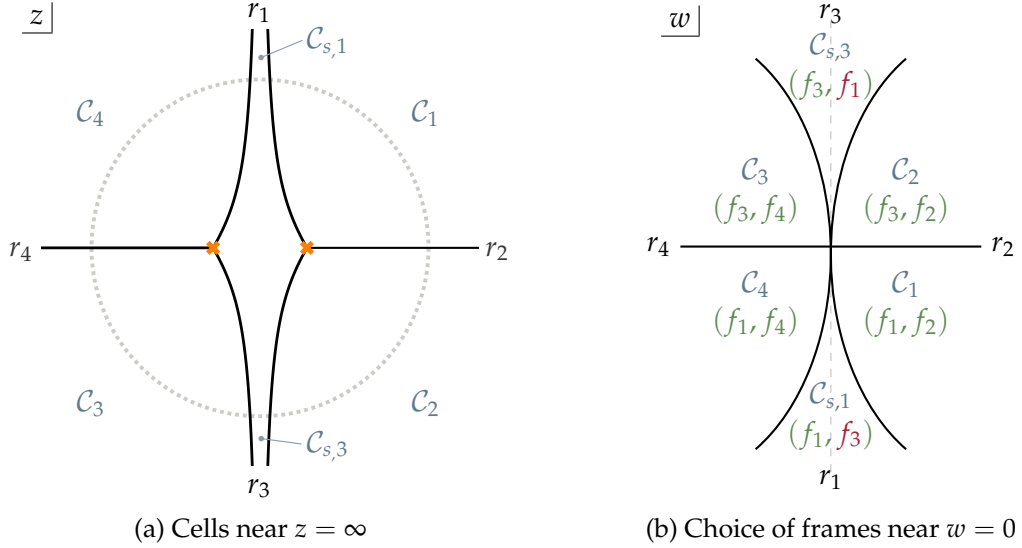

\begin{prop}[Sectorial asymptotic existence in terms of $g$] \label{prop:asymptotic-existence-g}
    In a neighbourhood of $w=0$ in each extended sector $\eSect_i$, there is an invertible matrix $\tilde{\Sigma}_i$ of smooth functions such that $\nabla_\zeta$ has the diagonal form 
    \begin{equation}\label{eq:diag-form}
            \nabla_\zeta = d + \left[ -\zeta^{-1} \frac{dw}{w^3} - \zeta \frac{d \overline{w}}{\overline{w}^3} - (\zeta^{-1} m - \frac{1}{2} m^{(3)}) \frac{dw}{w} - (\zeta \overline{m} + \frac{1}{2} m^{(3)})\frac{d \overline{w}}{\overline{w}}  \right]H
        \end{equation}
    with respect to the sectorial frame $g \cdot \tilde{\Sigma}_i $.

  Furthermore, each $\tilde{\Sigma}_i \to \bbid$ as $w \to 0$ in $\eSect_i$.
\end{prop}
\begin{proof}
    See \cref{sub:sectorial-stokes-data}.
\end{proof}

We can thereby obtain the desired frames $(f_1, f_2), (f_3, f_2), (f_3, f_4), (f_1, f_4)$ in the big cells by restricting each of the sectorial frames $g \cdot \tilde{\Sigma}_i$ from $\eSect_i$ to $\mathcal{C}_i$.
The corresponding flat sections $(s_1, s_2), (s_3, s_2), (s_3, s_4), (s_1, s_4)$ are the restrictions of the sectorially-defined sections $\Phi_i$ from the classical theory.

In fact, $f_1$ is also defined in the small cell $\mathcal{C}_{s,1}$, so it only remains to extend $f_3$ to this region (and likewise with $1$ and $3$ swapped; i.e.\ we must describe the $f_i$  labelled in red in \cref{subfig:sn-w-frames}).
To this end, we can analytically continue the corresponding flat section $s_3$ from $\mathcal{C}_3$ to $\mathcal{C}_{s,1}$ (counterclockwise, say) using the differential equation $\nabla s_3 = 0$. 
The fact that $s_1$ and $s_3$ form a basis follows from the nonvanishing of Stokes data \cite[Proposition 3.11]{Tulli:2019}, i.e.\ that the Stokes matrix element $a$ is nonzero when $m \neq 0$. 
We can then get the desired frame $(f_1, f_3)$ from $(s_1, s_3)$ by using \eqref{eq:sections-vs-frames} with
\begin{equation} \label{eq:singular-type}
    \begin{cases}
     Q   = \frac{1}{2} ( \zeta^{-1} w^{-2} + \zeta \overline{w}^{-2}) H, \\
    \Lambda  = - (\zeta^{-1} m - \frac{1}{2} m^{(3)})H, \\
    \Lambda' = - (\zeta \overline{m} + \frac{1}{2} m^{(3)})H.
    \end{cases}
\end{equation}

To summarize, we have chosen frames $(f_i, f_j)$ in a neighbourhood of $w=0$ in each cell with respect to which the connection $\nabla_\zeta$ has the diagonal form $\eqref{eq:diag-form}$.
By construction, these frames glue together to give sections of the abelianized bundle near the punctures.

\begin{cor}[Induced frame for $\L^\ab$] \label{cor:induced-frame}
    In a neighbourhood of each puncture $\infty_\pm$ of $\Sigma$, there exists a frame $g^\ab_{\zeta, \pm}$ for $\L^\ab_\zeta$ with respect to which 
  \begin{equation} \label{eq:induced-frame-form}
      \nabla^\ab_\zeta = d \pm \pi^* \left[ \zeta^{-1} \frac{dw}{w^3} + \zeta \frac{d \overline{w}}{\overline{w}^3} + (\zeta^{-1} m - \frac{1}{2} m^{(3)}) \frac{dw}{w} + (\zeta \overline{m} + \frac{1}{2} m^{(3)})\frac{d \overline{w}}{\overline{w}}   \right].
  \end{equation}
\end{cor}

This gives us the desired framed abelianization map
\begin{align} \label{eq:ab-map-set}
\begin{split}
    \ab: \Afr_\zeta &\to \abAfr_\zeta \\ 
    (E, \nabla_\zeta, g) &\mapsto (\L^\ab_\zeta, \nabla^\ab_\zeta, g^\ab_{\zeta, \pm}),
\end{split}
\end{align}
which descends to a map of moduli spaces
\begin{align} \label{eq:ab-map-moduli}
\begin{split}
    \ab: \Mfr_\zeta &\to \abMfr_\zeta \\
    [(E, \nabla_\zeta, g)] &\mapsto [(\L_\zeta^\ab, \nabla^\ab_\zeta, g^\ab_{\zeta, \pm})].
\end{split}
\end{align}

\begin{remark}[Abelian moduli space expectations]\label{abelian-moduli-expectations}
    Analogously to \cref{rem:nonabelian-moduli-expectations} about $\Mfr_\zeta$, we will carry out calculations for $\abMfr_\zeta$ using the ``obvious tangent spaces'', and we expect (but will not prove or need) that $\ab$ gives an isomorphism of moduli spaces $\Mfr_\zeta \xrightarrow{\sim} \abMfr_\zeta$.
\end{remark}

\subsection{Interlude: extending frames} 
\label{sub:extending-frames}

Soon we will want to extend the frames $g^\ab_\pm$ of the bundles $(\L, \nabla^\ab, g^\ab_\pm) \in \abAfr_\zeta$ from neighbourhoods of the punctures $\infty_\pm$ to the rest of $\Sigma$.
We will address this problem here in slightly more generality.

Let $S = \CP^1 \setminus \{p_1, \dots, p_m\}$ be an $m$-punctured sphere, and let $L$ be a complex line bundle over $S$ with a flat $\C^*$-connection $\nabla$.
Choose a trivialization $\tau_i$ of $L$ in a neighbourhood $U_i$ of each puncture, and let $\alpha_i$ denote the corresponding connection form of $\nabla$. 
Let $\gamma_i$ be a counterclockwise loop around $p_i$ in $U_i$ (see \cref{fig:punctured-sphere}).

\begin{figure}[ht]
    \centering
    \includegraphics[scale=1]{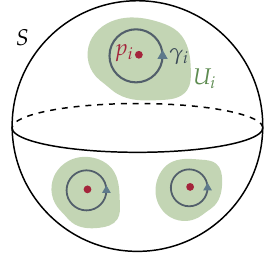}
    \caption{An $m$-punctured sphere $S$. Each puncture $p_i$ is surrounded by a loop $\gamma_i$, and a neighbourhood $U_i$ in which we have a fixed trivialization $\tau_i$ of the line bundle $L$ over $S$.}
    \label{fig:punctured-sphere}
\end{figure}

\begin{lemma}[Global extension condition]\label{lem:extend-triv}
    The local trivializations $\tau_i$ extend to a global trivialization $\tau$ of $L$ if and only if
    \begin{equation} \label{eq:winding-condition}
        \sum_i \int_{\gamma_i} \alpha_i = 0.
    \end{equation}
\end{lemma}

\begin{proof}
    Any closed $1$-form $\alpha$ which is defined on all of $S$ (e.g.\ the connection form of $\nabla$ with respect to a global trivialization) must satisfy 
    \begin{equation*}
        \sum_i \int_{\gamma_i} \alpha = 0
    \end{equation*}
    by Stokes' theorem, 
    so the condition \eqref{eq:winding-condition} is certainly necessary in order for the $\tau_i$ to extend.

    Conversely, given local trivializations satisfying \eqref{eq:winding-condition}, we can produce an extension as follows.
    Begin by choosing an arbitrary global trivialization $\tau^{(0)}$ of $L$ on $S$ (which exists since $L$ is a complex line bundle over a punctured surface), and let $\alpha^{(0)} \in \Omega^1(S)$ be the corresponding connection form of $\nabla$. 
    We want to find a gauge transformation $g: S \to \C^*$ such that $\nabla$ is of the desired form with respect to $g \tau^{(0)}$; that is, such that 
    \begin{equation*}
        (\alpha^{(0)} + d \log g)|_{U_i} = \alpha_i \quad \text{for each $i$.}
    \end{equation*}
    We split this up in two steps.

    Step 1 (matching up the integrals): there is a gauge transformation $g^{(1)}: S \to \C^*$ such that $\alpha^{(1)} \coloneqq \alpha^{(0)} + d \log g^{(1)}$ satisfies
    \begin{equation} \label{eq:gauge-same-integrals}
        \int_{\gamma_i} \alpha^{(1)} = \int_{\gamma_i} \alpha_i \quad \text{for each $i$.}
    \end{equation}
    For this we can explicitly take 
    \begin{equation*}
        g^{(1)}(z) = \prod_{i=1}^m (z-p_i)^{n_i},
    \end{equation*}
    where
    \begin{equation*}
        n_i \coloneqq \frac{1}{2 \pi i} \left( \int_{\gamma_i} \alpha_i  - \int_{\gamma_i}  \alpha^{(0)}\right) \in \Z.
    \end{equation*}
    Note that $\sum n_i = 0$ since both $\sum \int_{\gamma_i} \alpha_i = 0$ (by assumption) and $\sum \int_{\gamma_i} \alpha^{(0)} = 0$ (since $\alpha^{(0)}$ is a globally defined closed $1$-form), so $g^{(1)}$ indeed maps $S \to \C^*$.

    Step 2 (matching up the connection forms): there is a gauge transformation $g^{(2)}: S \to \C^*$ such that $\alpha^{(2)} \coloneqq \alpha^{(1)} + d \log g^{(2)}$ satisfies
    \begin{equation*}
        \alpha^{(2)}|_{U_i} = \alpha_i \quad \text{for each $i$.}
    \end{equation*}
    Indeed, in each $U_i$ we can define
    \begin{equation*}
        g^{(2)}(z)|_{U_i} = \exp \left(\int_*^z \alpha_i - \alpha^{(1)}   \right)
    \end{equation*}
    for some choice of basepoint $* \in U_i$. 
    Then
    \begin{equation*}
        \int_{\gamma_i} d \log g^{(2)} = \int_{\gamma_i} (\alpha_i - \alpha^{(1)}) = 0
    \end{equation*}
    by \eqref{eq:gauge-same-integrals}. But on the other hand 
    \begin{equation*}
        \int_{\gamma_i} d \log g^{(2)} = 2 \pi i \cdot \wind_0(g^{(2)} (\gamma_i)),
    \end{equation*}
    and so $g^{(2)} (\gamma_i) \subset \C^*$ has winding number zero around the origin for each $i$. 
    Hence we can choose an extension of $g^{(2)}$ from the $U_i$ to all of $S$.

    Combining these, the frame $\tau \coloneqq g^{(1)} g^{(2)} \tau^{(0)}$ gives us the desired global extension of the $\tau_i$.
\end{proof}

Now we will apply this to the bundles $(\L, \nabla^\ab, g^\ab_\pm) \in \abAfr_\zeta$ over the spectral curve $\Sigma$. 
In this case, by definition, we have trivializations $g^\ab_\pm$ near $\infty_\pm$ with respect to which $\nabla^\ab$ is of the prescribed form \eqref{eq:ab-framed-form}.
We can further prescribe trivializations at the branch points $b_\pm \in \Sigma$ (around which $\nabla^\ab$ has holonomy $-1$), as long as the condition \eqref{eq:winding-condition} of the above lemma is still satisfied.

\begin{cor}[Global frame for $\L$]\label{cor:global-frame}
    Given $(\L, \nabla^\ab, g^\ab_\pm) \in \abAfr_\zeta$, the frame $g^\ab_\pm$ extends to a global frame $g^\ab$ for $\L$ with respect to which:
    \begin{enumerate}
        \item near each puncture $\infty_{\pm}$, 
         \begin{equation}  \label{eq:ab-form-punctures}
         \nabla^\ab = d \pm \pi^* \left[ \zeta^{-1} \frac{dw}{w^3} + \zeta \frac{d \overline{w}}{\overline{w}^3} +(\zeta^{-1} m - \frac{1}{2} m^{(3)}) \frac{dw}{w} + (\zeta \overline{m} + \frac{1}{2} m^{(3)})\frac{d \overline{w}}{\overline{w}}   \right].
      \end{equation}

      \item near each branch point $b_\pm$,
       \begin{equation} \label{eq:ab-form-branch-points}
          \nabla^\ab = d \pm \frac{dt_\pm}{2t_\pm},
      \end{equation}
      where $t_\pm$ is a (fixed) coordinate for $\Sigma$ centred at $b_\pm$. 
    \end{enumerate}
\end{cor}

\begin{remark}
    For an abelianized bundle $(\L^\ab_\zeta, \nabla^\ab_\zeta, g^\ab_{\zeta, \pm}) = \ab (E, \nabla_\zeta, g)$, we can extend the frames $(f_i, f_j)$ in each cell of the spectral network from a neighbourhood of $w=0$  to the entire cell by pushing down the extended frame $g^\ab$.
    By construction these frames diagonalize $\nabla_\zeta$ in the entire cell. 
    However, we emphasize that we only have an explicit formula for the connection form near $w=0$.
\end{remark}

\section{Regularized Atiyah-Bott forms}
\label{sec:regularized-forms}

We have now defined our main objects of study: the framed connections $(E, \nabla, g) \in \Afr_\zeta$ over $C = \CP^1$ and their abelianizations $(\L^\ab, \nabla^\ab, g^\ab_\pm) \in \abAfr_\zeta$ over the spectral cover $\Sigma$.
In this section we will define a regularized version of the Atiyah-Bott form on each of these spaces of connections, and show that it is preserved by abelianization.

The fact that abelianization preserves the standard Atiyah-Bott form is discussed in \cite[Section 10.4]{Gaiotto:2013} for connections on a closed surface with vanishing variations near the punctures, in which case the usual Atiyah-Bott formulas converge.
In our case, the integrals
\begin{equation*}
    \int_C \dot{\nabla}_1 \wedge \dot{\nabla}_2 \quad \text{and} \quad \int_\Sigma \dot{\nabla}^\ab_1 \wedge \dot{\nabla}^\ab_2
\end{equation*}
are (logarithmically) divergent when the parameters $m$ and $m^{(3)}$ are allowed to vary, so we will need to incorporate a regularization term.

\begin{remark}[Generalizations]
    The framed connections in $\Afr_\zeta$ and $\abAfr_\zeta$ have fixed higher-order singular terms, while their first-order singular part (i.e.\ the coefficients of $\frac{dw}{w}$ and $\frac{d \overline{w}}{\overline{w}}$) can vary.
    Many of the definitions in this section would make sense more generally for similar spaces of connections (and on other punctured Riemann surfaces), but for simplicity we will just focus on the two relevant cases.
\end{remark}

\subsection{On the base curve} 
\label{sub:reg-form-base}

To start, we will restate the definition of $\Omega^\reg$ on $\Afr_\zeta$ (cf.\ \cref{def:INTRO-reg-form}).

Let $C_R \coloneqq \CP^1 \setminus \{|w| < R\}$, and write
    \begin{equation}
        \mathfrak{h} \coloneqq \C \cdot H = \left\{ \begin{pmatrix}
            a & 0 \\  
            0 & -a
    \end{pmatrix} : a \in \C \right\} \subset \mathfrak{sl}(2, \C).
    \end{equation}

Recall that the framed connections $(E, \nabla, g) \in \Afr_\zeta$ are of the form
\begin{equation}
\begin{split} \label{eq:framed-connection-form-rewrite}
     \nabla =  d &- \zeta^{-1} H \frac{dw}{w^3}  - \zeta H \frac{d \overline{w}}{\overline{w}^3} - (\zeta^{-1} m - \frac{1}{2} m^{(3)}) H \frac{dw}{w} -  (\zeta \overline{m} + \frac{1}{2} m^{(3)}) H\frac{d \overline{w}}{\overline{w}}  \\
     &+ \text{regular terms}
\end{split}
\end{equation}
near $w=0$, and so their variations can be written
\begin{align}
    \dot{\nabla} &= - (\zeta^{-1} \dot{m} - \frac{1}{2} \dot{m}^{(3)})H  \frac{dw}{w} -  (\zeta \dot{\overline{m}} + \frac{1}{2} \dot{m}^{(3)}) H \frac{d \overline{w}}{\overline{w}}  + \text{regular terms} \\
    &= - i (\zeta^{-1} \dot{m} - \dot{m}^{(3)} - \zeta \dot{\overline{m}})H\, d\theta 
    - (\zeta^{-1} \dot{m} + \zeta \dot{\overline{m}})H \, \frac{dr}{r} + \text{regular terms} \label{eq:polar-variation-form}
\end{align}
in polar coordinates $w = re^{i \theta}$.
This is of the form
\begin{equation}
    \dot{\nabla} = (\mu + \O(r)) d \theta + (\lambda + \O(r)) \frac{dr}{r}
\end{equation}
for some $\mu, \lambda \in \mathfrak{h}$, namely
\begin{equation} \label{eq:explicit-mu-lambda-coeffs}
    \mu = - i (\zeta^{-1} \dot{m} - \dot{m}^{(3)} - \zeta \dot{\overline{m}})H \quad \text{and} \quad \lambda = - (\zeta^{-1} \dot{m} + \zeta \dot{\overline{m}})H.
\end{equation}

\begin{definition}[Regularized form] \label{def:reg-form}
    Define a \emph{regularized Atiyah-Bott form} $\Omega^\reg$ on $\Afr_\zeta$ by 
    \begin{equation} \label{eq:reg-form}
        \Omega^\reg(\dot{\nabla}_1, \dot{\nabla}_2) = \lim_{R \to 0} \left[\int_{C_R} \tr (\dot{\nabla}_1 \wedge \dot{\nabla}_2) - 2 \pi \log R \cdot \mathcal{R}(\dot{\nabla}_1, \dot{\nabla}_2) \right],
    \end{equation}
    where 
    \begin{equation}
        \mathcal{R}(\dot{\nabla}_1, \dot{\nabla}_2) = \tr \left( \mu_1 \lambda_2 - \mu_2 \lambda_1 \right)
    \end{equation}
    and
    \begin{equation} \label{eq:asymptotic-polar-coeffs}
    \dot{\nabla}_i = (\mu_i + \O(r)) d \theta + (\lambda_i + \O(r)) \frac{dr}{r} \quad \text{for some } \mu_i, \lambda_i \in \mathfrak{h}
    \end{equation}
    in polar coordinates $w = re^{i \theta}$ near $w=0$.
\end{definition}

\begin{remark}[Regularization term] \label{rem:reg-term}
Using the coefficients \eqref{eq:explicit-mu-lambda-coeffs}, we can explicitly calculate the regularization term
\begin{equation} \label{eq:regularization-term}
    \begin{split}
    -2\pi \log R \cdot \tr(\mu_1 \lambda_2 - \mu_2 \lambda_1) = -4\pi i \log R [
    &(\zeta^{-1} \dot{m}_1 - \dot{m}^{(3)}_1 - \zeta \dot{\overline{m}}_1)
    (\zeta^{-1} \dot{m}_2 + \zeta \dot{\overline{m}}_2) \\
    - &(\zeta^{-1} \dot{m}_2 - \dot{m}^{(3)}_2 - \zeta \dot{\overline{m}}_2)
     (\zeta^{-1} \dot{m}_1 + \zeta \dot{\overline{m}}_1)
    ]. \end{split}
\end{equation}
\end{remark}

\begin{remark}[L.\ Jeffrey form]
    A similar bilinear form appears in \cite[Definition 3.2]{Jeffrey:1994}, under the identification of her half-open cylinder $[0, \infty) \times S^1$ with the punctured disc $\Delta^\times$ via $(t, s) \mapsto e^{-t + 2\pi i s}$.
    (More precisely, \cite{Jeffrey:1994} defines a one-parameter family of forms indexed by $r \in [0, \infty)$; our definition for $\Omega^\reg$ coincides when $r = 0$.)

    The form in \cite{Jeffrey:1994} is introduced as an auxiliary construction to prove nondegeneracy of the Atiyah-Bott form on an ``extended moduli space'' of connections on a surface with boundary.
    We will instead introduce an abelian version of the form in \cref{sub:reg-form-cover} below, and study it using a gluing construction in \cref{sec:glued-symplectic-form}.
\end{remark}

The form $\Omega^\reg$ has the following basic properties.

\begin{lemma}[Convergence] \label{lem:reg-form-convergence}
    For any variations $\dot{\nabla}_1, \dot{\nabla}_2$, the limit defining $\Omega^\reg(\dot{\nabla}_1, \dot{\nabla}_2)$ is convergent.
\end{lemma}
\begin{proof}
    Fix a sufficiently small radius $R_0$ so that the framed form \eqref{eq:polar-variation-form} holds for $|w| < R_0$.
    Consider the annulus $A_R \coloneqq \{ R \leq |w| \leq R_0 \}$, and split up
    \begin{align*}
        \Omega^\reg(\dot{\nabla}_1, \dot{\nabla}_2) = \underbrace{\int_{C_{R_0}} \tr(\dot{\nabla}_1 \wedge \dot{\nabla}_2)}_{\text{finite}} + \lim_{R \to 0} \left[\int_{A_R} \tr(\dot{\nabla}_1 \wedge \dot{\nabla}_2) - 2 \pi \log R \cdot \tr \left( \mu_1 \lambda_2 - \mu_2 \lambda_1 \right) \right].
    \end{align*}
    Since
    \begin{equation*}
        \tr (\dot{\nabla}_1 \wedge \dot{\nabla}_2) = \left[ -\tr (\mu_1 \lambda_2 - \mu_2 \lambda_1) + \O(r)  \right] \frac{dr}{r} \wedge d\theta,
    \end{equation*}
    it follows that
    \begin{equation*}
        \int_{A_R}  \tr (\dot{\nabla}_1 \wedge \dot{\nabla}_2) = 2 \pi (\log R - \log R_0) \cdot  \tr (\mu_1 \lambda_2 - \mu_2 \lambda_1 ) + \O(1),
    \end{equation*}
    and so the regularization term of $\Omega^\reg$ cancels out the divergent term as $R \to 0$.
\end{proof}

\begin{lemma}[Gauge invariance] \label{lem:reg-form-gauge-invariance}
    $\Omega^\reg$ descends to a form on the moduli space $\Mfr_\zeta$. 
\end{lemma}
\begin{proof}
    Let $\mathcal{G}$ denote the group of gauge transformations $g: C \to \GL_2(\C)$ which approach the identity near $w=0$ and preserve the framed form \eqref{eq:framed-connection-form-rewrite} of the connections $\nabla \in \Afr$.
    We must show that $\Omega^\reg$ is \emph{basic} with respect to the action of $\mathcal{G}$; that is, that:
    \begin{enumerate}[(i)]
        \item $\Omega^\reg$ is $\mathcal{G}$-invariant, and
        \item $\Omega^\reg$ vanishes on vertical tangent vectors (i.e.\ along gauge orbits).
    \end{enumerate}

    The first statement is just conjugation-invariance of the trace.
    For the second statement, the integral term of $\Omega^\reg$ vanishes along vertical tangent vectors by the usual Atiyah-Bott argument using Stokes' theorem (and the fact that $g \to \id$ as $R \to 0$).
    The regularization term also vanishes because variations in the gauge direction have $\mu = \lambda = 0$, since the gauge transformations preserve \eqref{eq:framed-connection-form-rewrite}.
\end{proof}

We will let $\Omega^\reg_\zeta$ denote the form on $\Hfr$ obtained by pulling back $\Omega^\reg$ via $\NAH_\zeta: \Hfr \to \Afr_\zeta$, as well as the induced form on the moduli space $\Xfr$ (see \cref{fig:regularized-forms}).

\begin{figure}[ht]
    \begin{tikzcd}
    &[-25pt] \text{\underline{Higgs bundles}:} & \text{\underline{Connections}:} \\[-20pt]
    \text{\underline{Sets of objects}:} & \color{green} (\Hfr, \Omega^\reg_\zeta) \arrow[d, two heads] \arrow[r, "\NAH_\zeta"] & \color{orange} (\Afr_\zeta, \Omega^\reg)\arrow[d, two heads] \\
    \text{\underline{Moduli spaces}:} & \color{green} (\Xfr, \Omega^\reg_\zeta) \arrow[r, "\NAH_\zeta"] & \color{orange} (\Mfr_\zeta, \Omega^\reg)                     
    \end{tikzcd}
    \caption{Forms induced by the regularized Atiyah-Bott form $\Omega^\reg$ on $\Afr_\zeta$.}
    \label{fig:regularized-forms}
\end{figure}

\subsection{On the spectral cover} 
\label{sub:reg-form-cover}

There is a natural analogue of \cref{def:reg-form} for the abelianized connections on the spectral cover $\Sigma$, where now the regularization term involves contributions from both of the punctures $\infty_\pm$.

Write $\Sigma_R \coloneqq \pi^{-1}(C_R)$, where $\pi: \Sigma \to C$ is the spectral cover.

\begin{definition}[Regularized abelian form] \label{def:ab-reg-form}
    Define a \emph{regularized abelian Atiyah-Bott form} $\Omega^{\reg, \ab}$ on $\abAfr_\zeta$ by
    \begin{equation} \label{eq:ab-reg-form}
        \Omega^{\reg, \ab}(\dot{\nabla}_1^\ab, \dot{\nabla}_2^\ab) = \lim_{R \to 0} \left[ \int_{\Sigma_R} \dot{\nabla}_1^\ab \wedge \dot{\nabla}_2^\ab - 2 \pi \log R \cdot \mathcal{R}^\ab(\dot{\nabla}_1^\ab, \dot{\nabla}_2^\ab)  \right],
    \end{equation}
    where
    \begin{equation} \label{eq:ab-reg-term}
        \mathcal{R}^\ab (\dot{\nabla}_1^\ab, \dot{\nabla}_2^\ab) = 2 ( \mu^\ab_1 \lambda^\ab_2 - \mu^\ab_2 \lambda^\ab_1)
    \end{equation}
    and
    \begin{equation} \label{eq:ab-asymptotic-polar-coeffs}
    \dot{\nabla}_i^\ab = (\pm \mu^\ab_i + \O(r)) d \theta + (\pm \lambda^\ab_i + \O(r)) \frac{dr}{r} \quad \text{for some } \mu^\ab_i, \lambda^\ab_i \in \C
    \end{equation}
    in polar coordinates $(r, \theta)$ centred at each puncture $\infty_\mp$.\footnote{More generally we could define $\mathcal{R}^\ab (\dot{\nabla}_1^\ab, \dot{\nabla}_2^\ab)  = ( \mu^+_1 \lambda^+_2 - \mu^+_2 \lambda^+_1) + ( \mu^-_1 \lambda^-_2 - \mu^-_2 \lambda^-_1)$ for variations of the form $\dot{\nabla}_i^\ab = (\mu^\pm_i + \O(r)) d \theta + (\lambda^\pm_i + \O(r)) \frac{dr}{r}$ near $\infty_\pm$.}
\end{definition}

\begin{remark}
    Because of the rigid prescribed framed form \eqref{eq:ab-framed-form} of the connections $\nabla^\ab \in \Afr_\zeta$ (cf.\ \cref{rem:abelian-form-rigidity}), their variations are \emph{exactly} of the form
\begin{equation}
    \dot{\nabla}_i^\ab = \pm \mu^\ab  d \theta \pm \lambda^\ab  \frac{dr}{r} \quad \text{for some } \mu^\ab_i, \lambda^\ab_i \in \C,
\end{equation}
i.e.\ the  $\O(r)$ terms in \eqref{eq:ab-asymptotic-polar-coeffs} are actually zero.
This will not be needed for the arguments below, however, and the more general definition will carry over to the semiflat setting in  \cref{sub:sf-regularized-atiyah-bott-forms}.
\end{remark}

The values $\mu^\ab_i$ and $\lambda^\ab_i$ for the abelianized variations $\dot{\nabla}_i^\ab$ are just the diagonal entries of $\mu_i$ and $\lambda_i$ for $\dot{\nabla}_i$, that is,
\begin{equation}
     \mu_i = \diag(\mu^\ab_i, -\mu^\ab_i) \quad \text{and} \quad  \lambda_i = \diag(\lambda^\ab_i, -\lambda^\ab_i),
\end{equation}
and so
\begin{equation}
     \tr \left( \mu_1 \lambda_2 - \mu_2 \lambda_1 \right) = 2 ( \mu^\ab_1 \lambda^\ab_2 - \mu^\ab_2 \lambda^\ab_1).
\end{equation}

\begin{cor}[Abelianization preserves regularization terms] \label{lem:reg-terms-coincide}
    \begin{equation}
         \mathcal{R}(\dot{\nabla}_1, \dot{\nabla}_2) = \mathcal{R}^\ab(\dot{\nabla}^\ab_1, \dot{\nabla}^\ab_2),
    \end{equation}
    i.e.\ $\Omega^{\reg, \ab}(\dot{\nabla}_1^\ab, \dot{\nabla}_2^\ab)$ and $\Omega^\reg(\dot{\nabla}_1, \dot{\nabla}_2)$ have the same regularization terms, namely given by \eqref{eq:regularization-term}.
\end{cor}

\subsection{Regularization and abelianization} 
\label{sub:reg-and-abelianization}

Now we show that as expected, abelianization preserves the regularized forms.
(This is equality \eqref{eq:INTRO-ab-symplectomorphism} in the schematic commutative diagram \eqref{eq:forms-cd}.)

\begin{prop}[``Abelianization is a symplectomorphism''] \label{prop:ab-symplectomorphism}
As forms on $\Afr_\zeta$,
\begin{equation} \label{eq:ab-symplectomorphism}
    \ab^* \Omega^{\reg, \ab} = \Omega^\reg,
\end{equation}
i.e.\
\begin{equation}
    \Omega^\reg(\dot{\nabla}_1, \dot{\nabla}_2) = \Omega^{\reg, \ab}(\dot{\nabla}^\ab_1, \dot{\nabla}^\ab_2).
\end{equation}
\end{prop}

By \cref{lem:reg-terms-coincide}, the regularization terms appearing in both of the forms are the same, so we only need to compare the two integrals
\begin{equation*}
    \int_{C_R} \tr (\dot{\nabla}_1 \wedge \dot{\nabla}_2) \quad \text{and} \quad \int_{\Sigma_R} \dot{\nabla}_1^\ab \wedge \dot{\nabla}_2^\ab
\end{equation*}
which are regularized as $R \to 0$.
This boils down to a local calculation as argued in \cite[Section 10.4]{Gaiotto:2013}.
The main idea is that the Stokes jump upon crossing a wall of the spectral network should be thought of as contributing an off-diagonal distribution term supported at the wall, whose product with the other diagonal terms of the variations is traceless.

For completeness we include a slightly expanded version of their proof below. 
We will not say anything fundamentally new, but will emphasize how some of the details fit in with our setup and ``smooth out the delta function'' to remain in the $C^\infty$ setting.

\begin{proof}
    If the variations $\dot{\nabla}_i$ have support $V$ away from the walls of the spectral network $\W$, then $\dot{\nabla}_i = \pi_* \dot{\nabla}_i^\ab$, and so $\tr(\dot{\nabla}_1 \wedge \dot{\nabla}_2)|_V$ is just the sum of $\dot{\nabla}_1^\ab \wedge \dot{\nabla}_2^\ab$ on the two sheets of $\pi^{-1}(V)$.

    Suppose the support of a variation, say $\dot{\nabla}_1$, intersects $\W$. 
    We can assume the gauge has been chosen so that $\dot{\nabla}_1$ vanishes near the branch points, so it suffices to consider the intersection with a single wall $w$.
    For simplicity, choose coordinates $(x,y)$ on $C$ so that the wall is at $y=0$.

    Let $f = (f_1, f_2)$ and $f' = (f_1', f_2')$ be frames below and above the wall (see \cref{fig:wall-frames}) which diagonalize
    \begin{equation*}
    \nabla = d + \begin{pmatrix}
        \beta^- & \\ 
         & \beta^+
        \end{pmatrix}.
    \end{equation*}
    They are related by $f' = f \cdot S$ for a unipotent and (without loss of generality) upper-diagonal Stokes matrix
    \begin{equation*}
        S = \begin{pmatrix}
            1 & \alpha \\
            0 & 1
        \end{pmatrix}.
    \end{equation*}

    In order to obtain a single frame with which to compute variations of $\nabla$, we interpolate between $f$ and $f'$ as follows.
    Let $\widetilde{w}$ be a thickening of the wall $w$ to some region $-\epsilon < y < \epsilon$ inside of which the Stokes matrix element $\alpha$ is defined.    
    Let $\eta$ be a smooth cutoff function with 
    \begin{equation*}
        \eta(x,y) = \begin{cases}
            0 & \text{for } y < -\epsilon,  \\
            1 & \text{for } y > \epsilon.
        \end{cases}
    \end{equation*}
        Then the frame $\tilde{f} \coloneqq f \cdot \tilde{S}$, where
        \begin{equation*}
        \tilde{S} = \begin{pmatrix}
            1 & \eta \alpha \\
            0 & 1
        \end{pmatrix},
    \end{equation*}
    smoothly interpolates between $f$ and $f'$.

    \begin{figure}[ht]
        \centering
        \includegraphics[scale=1]{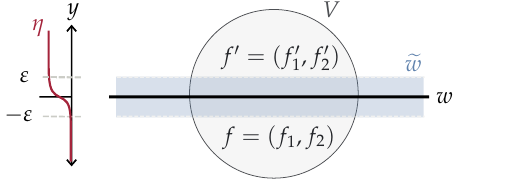}
        \caption{Here the support $V$ of the variations intersects a wall $w$ of the spectral network. 
        The frames $f$ and $f'$ are used for abelianization on each side of the wall.
        We interpolate between them using a cutoff function $\eta$ which smoothly goes from $0$ to $1$ as we move through the thickened wall $\widetilde{w}$.}
        \label{fig:wall-frames}
    \end{figure}

    With respect to $\tilde{f}$ inside $\widetilde{w}$,
    \begin{equation}
        \nabla = d + \begin{pmatrix}
            0 & d(\eta \alpha) \\
            0 & 0
        \end{pmatrix} + \begin{pmatrix}
            \beta^- & \eta \alpha (\beta^- - \beta^+) \\
            0 & \beta^+
        \end{pmatrix},
    \end{equation}
    and so variations are of the form
    \begin{equation}
        \dot{\nabla} = \begin{pmatrix}
            0 & * \\
            0 & 0
        \end{pmatrix} + \begin{pmatrix}
            \dot{\beta}^- & * \\
            0 & \dot{\beta}^+
        \end{pmatrix}.
    \end{equation}
    Crucially, the off-diagonal term does not contribute to $\tr(\dot{\nabla}_1 \wedge \dot{\nabla}_2)$, and we again get 
    \begin{equation*}
        \tr(\dot{\nabla}_1 \wedge \dot{\nabla}_2) = \dot{\beta}^-_1 \wedge \dot{\beta}^-_2 + \dot{\beta}^+_1 \wedge \dot{\beta}^+_2,
    \end{equation*}
    the sum of $\dot{\nabla}_1^\ab \wedge \dot{\nabla}_2^\ab$ on the two sheets.
\end{proof}

\begin{remark}
    To avoid potential confusion, we emphasize that the argument above does \emph{not} say that 
    \begin{equation*}
        \int_{C_R} \tr (\dot{\nabla}_1 \wedge \dot{\nabla}_2) \quad \text{and} \quad  \int_{\Sigma_R} \dot{\nabla}_1^\ab \wedge \dot{\nabla}_2^\ab
    \end{equation*}
    are equal at a finite radius $R > 0$.
    In fact, the first integral is not even invariant under gauge transformations that approach the identity near $z=\infty$ (which we utilized in the above proof).
    In order to obtain a gauge-invariant equality, we really need the full regularized integrals
    \begin{equation*}
        \lim_{R \to 0 }\left[\int_{C_R} \tr (\dot{\nabla}_1 \wedge \dot{\nabla}_2) -2 \pi \log R \cdot \mathcal{R}(\dot{\nabla}_1, \dot{\nabla}_2)\right] 
        = \lim_{R \to 0} \left[\int_{\Sigma_R} \dot{\nabla}_1^\ab \wedge \dot{\nabla}_2^\ab - 2 \pi \log R \cdot \mathcal{R}^\ab(\dot{\nabla}_1^\ab, \dot{\nabla}_2^\ab) \right].
    \end{equation*}
\end{remark}

\section{Glued symplectic form}
\label{sec:glued-symplectic-form}

In the previous section we showed that we can study the regularized Atiyah-Bott form $\Omega^\reg$ via its abelian analogue $\Omega^{\reg, \ab}$.
To relate this to the Ooguri-Vafa form, we will introduce an intermediary \emph{glued symplectic form} $\Omega^\glue$ on $\abAfr_\zeta$.
Then, using a kind of ``glued Riemann bilinear identity'' and the geometry of the relevant spectral networks, we will show that $\Xe$ and $\Xm$ are Darboux coordinates for $\Omega^\glue$.

This section is dedicated to explaining the construction of $\Omega^\glue$ and proving the following equalities of forms.  
(These are respectively \eqref{eq:INTRO-glue-equals-reg} and \eqref{eq:INTRO-glue-equals-ov} in the commutative diagram \eqref{eq:forms-cd}.)

\begin{prop}[Gluing and regularization] \label{prop:glue-equals-reg}
    \begin{equation} \label{eq:glue-equals-reg}
        \Omega^\glue = \Omega^{\reg, \ab},
    \end{equation}
    i.e.\ the glued form $\Omega^\glue$ coincides with the regularized abelian Atiyah-Bott form on $\abAfr_\zeta$.
\end{prop}

\begin{prop}[Gluing and Ooguri-Vafa] \label{prop:glue-equals-ov}
    \begin{equation} \label{eq:glue-equals-ov}
         \ab^* \Omega^\glue = -4 \pi^2 \cdot \Omega^\ov_\Stokes,
    \end{equation}
    i.e.\ the pullback of the glued form $\Omega^\glue$ to $\Afr_\zeta$ coincides with the Ooguri-Vafa form $\Omega^\ov_\Stokes$.
\end{prop}

Combining these with \cref{prop:ab-symplectomorphism}, we will obtain our first main result:
\begin{theorem}[Ooguri-Vafa and regularization] \label{thm:reg-equals-ov}
    Under the identification of spaces $\Mov \cong \Xfr$,  
    \begin{equation} \label{eq:reg-equals-ov}
         \Omega^\ov_\zeta = - \frac{1}{4\pi^2} \Omega^\reg_\zeta,
    \end{equation}
    i.e.\ the Ooguri-Vafa symplectic form coincides with the regularized Atiyah-Bott form, pulled back to $\Xfr$.
\end{theorem}

\subsection{General gluing construction on a cylinder} 
\label{sub:general-gluing-construction}

First we will describe a general construction for a closed $2$-form on a moduli space of framed flat $\C^*$-connections on a cylinder.
Roughly, it will be induced from the Atiyah-Bott symplectic form on the torus by gluing the ends of the cylinder together and choosing a family of gauge transformations to glue the connections.

\begin{enumerate}[(a)]
\item \underline{topological setup}:

Fix a topological cylinder $S$, and let $S_\top$ and $S_\bot$ denote neighbourhoods of its two boundary components.
Choose an orientation-reversing diffeomorphism $\sigma: S_\top \xrightarrow{\sim} S_\bot$ with which we can glue the ends together to form a torus $T = S/\sigma$.

\begin{figure}[ht]
    \centering
    \includegraphics[scale=1]{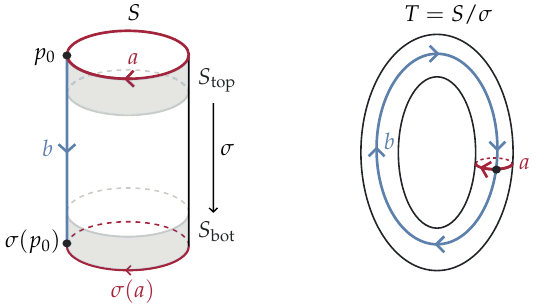}
    \caption{A cylinder $S$ with a loop $a$ and a longitudinal path $b$ from $p_0$ to $\sigma(p_0)$, glued to a torus $T$ via $\sigma$.}
    \label{fig:cylinder}
\end{figure}

Let $a$ and $b$ be paths as indicated in \cref{fig:cylinder} (so that their classes form a homology basis for $T$ after gluing), and let $p_0$
denote their point of intersection on the top boundary edge.

\item \underline{framed connections}:

Identify the space of smooth $\C^*$-connections on the trivial bundle over $S$ with $\Omega^1(S)$.
Then consider a space $\mathcal{A}_S$ of flat connections, consisting of closed $1$-forms $\alpha \in \Omega^1(S)$ which are of some prescribed ``framed form'' in a neighbourhood of the boundary $\partial S$.\footnote{Soon we will specify precise boundary conditions, but for now we are just interested in the formal setup.}
Let $\mathcal{G}_S$ denote the set of gauge transformations $g: S \to \C^*$ such that $g \equiv 1$ in a neighbourhood of $\partial S$.
Then we can define a moduli space $\M_S \coloneqq \mathcal{A}_S/\mathcal{G}_S$ of framed flat connections on $S$.

\begin{remark}[L.~Jeffrey moduli spaces] \label{rem:extended-spaces}
    Unlike the standard moduli spaces of connections on a surface with boundary, the boundary holonomies of our connections are \emph{not} fixed. 
    Our definition of $\M_S$ is similar to (an abelian version of) the ``extended moduli spaces'' of \cite{Jeffrey:1994}, which consist of framed connections whose boundary holonomies are allowed to vary, although we will prescribe a slightly different framed form.\footnote{The connections in the extended spaces from \cite[Section 2.1]{Jeffrey:1994} are of the form $\alpha = \mu \,  d \theta$ (for certain constant coefficients $\mu$) in polar coordinates near the boundary circles. In our case, we will also allow a nonzero $dr$ term.}
\end{remark}

\item \underline{gluing data}:

Connections on $S$ can be ``glued'' to $T$ by making suitable gauge transformations.

To construct the space of connections on the torus $T$, note that a $1$-form $\beta \in \Omega^1(S)$ is pulled back from $\Omega^1(T)$ if and only if $\beta|_\top = \sigma^* \beta|_\top$, so we can write
\begin{equation*}
    \mathcal{A}_T \coloneqq \{\beta \in \Omega^1(S): d\beta = 0 ,\ \beta|_\top = \sigma^* \beta|_\top\}.
\end{equation*}
Similarly, let $\mathcal{G}_T$ be the set of gauge transformations $g: S \to \C^*$ with $g|_\top = \sigma^* g|_\top$.
Then $\M_T \coloneqq \mathcal{A}_T/\mathcal{G}_T$ is the (usual) moduli space of flat $\C^*$-connections on the torus.

Given a connection form $\alpha \in \mathcal{A}_S$ (\emph{not} necessarily in $\mathcal{A}_T$), choose a smooth gauge transformation $g = e^\chi: S \to \C^*$ such that 
\begin{equation} \label{eq:gluing-condition}
     (\alpha - d \chi)|_\top = \sigma^*(\alpha - d \chi)|_\top.
\end{equation}
(We emphasize that such a gauge transformation $g$ is generally nontrivial on $\partial S$.)
Then $\alpha - d \chi \in \mathcal{A}_T$, i.e.\ we can regard $\alpha - d \chi$ as a connection form on $T$.

\begin{remark}[$\chi$-dependence]
    The following gluing construction \emph{will} depend on the choice of $\chi$, but only in a neighbourhood of $\partial S$.
    Indeed, $g = e^\chi$ extends from $\partial S$ to $S$ if $\chi$ satisfies the winding number compatibility condition
    \begin{equation} \label{eq:gluing-map-compatibility}
        \int_a d\chi = \int_a \sigma^* d\chi
    \end{equation}
    (cf.\ the proof of \cref{lem:extend-triv}), and any two such extensions will be gauge equivalent by a map in $\mathcal{G}_T$.
    Therefore it will be enough to just specify $\chi|_\top$ and $\chi|_\bot$ satisfying \eqref{eq:gluing-map-compatibility}.
\end{remark}

Now suppose we have a smoothly varying family of such maps $\chi = \chi(\alpha)$ for each $\alpha \in \mathcal{A}_S$, and package them into a ``connection gluing map'' 
\begin{align}
\begin{split}
    \Gamma_\chi: \mathcal{A}_S &\to \mathcal{A}_T \\
    \alpha &\mapsto \alpha - d \chi. 
\end{split}
\end{align}
Assume further that the choice of $\chi$ is gauge invariant in the sense that $\chi(\alpha) = \chi(\alpha')$ when $\alpha$ and $\alpha'$ are gauge equivalent.\footnote{In practice this is not a restrictive assumption:
if $\alpha$ and $\alpha'$ are gauge equivalent on $S$, then they agree in a neighbourhood of $\partial S$, and so the condition will automatically be satisfied if the choice of $\chi(\alpha)|_{\partial S}$ depends only on $\alpha|_{\partial S}$.}
Then it follows that $\Gamma_\chi$ descends to a map of moduli spaces 
\begin{equation}
    \widetilde{\Gamma}_\chi: \M_S \to \M_T.
\end{equation}

\begin{remark}[Symmetric gluing]
    We could, for instance, choose $\chi$ so that
    \begin{equation} \label{eq:sym-gluing-condition}
        d\chi = \frac{\alpha - \sigma^* \alpha}{2}
    \end{equation}
    on both the top and bottom of $S$.
    This satisfies the gluing condition \eqref{eq:gluing-condition}, with 
        \begin{equation*}
        \sigma^* (\alpha - d \chi) = \frac{\alpha + \sigma^* \alpha}{2} = \alpha - d \chi,
    \end{equation*}
    and the compatibility condition \eqref{eq:gluing-map-compatibility}, with
    \begin{equation*}
        \int_a d \chi = 0 = \int_a \sigma^* d \chi
    \end{equation*}
    (since $\int_a \alpha = \int_a \sigma^* \alpha$).
    We will call \eqref{eq:sym-gluing-condition} the \emph{symmetric gluing condition}, since the formula is the same on both boundary components.
\end{remark}
\end{enumerate}

Recall that the moduli space $\M_T$ of flat $\C^*$-connections on the torus carries the abelian Atiyah-Bott symplectic form $\widetilde{\Omega}^\AB_T$, induced from the form
\begin{equation}
    \Omega^\AB_T (\dot{\beta}_1, \dot{\beta}_2) = \int_T \dot{\beta}_1 \wedge \dot{\beta}_2
\end{equation}
on $\mathcal{A}_T$, where $\dot{\beta}_i$ are variations of $\beta \in \mathcal{A}_T$.
The desired glued form on $S$ will be obtained by pulling back the Atiyah-Bott form via $\Gamma_\chi$.

\begin{definition}[Glued form]
    Define a \emph{glued form} on $\mathcal{A}_S$ by
    \begin{equation}
        \omega_{S, \chi}^\glue \coloneqq (\Gamma_\chi)^* \Omega^\AB_T,
    \end{equation}
    i.e.
    \begin{equation}\label{eq:glued-form-0}
        \omega^\glue_{S, \chi}(\dot{\alpha}_1, \dot{\alpha}_2) = \int_S (\dot{\alpha}_1 - d \dot{\chi}_1) \wedge (\dot{\alpha}_2 - d \dot{\chi}_2),
    \end{equation}
    where $\dot{\alpha}_i$ are variations of $\alpha \in \mathcal{A}_S$ and $\dot{\chi}_i$ are the corresponding induced variations of $\chi(\alpha)$.
\end{definition}
The glued form on $\mathcal{A}_S$ descends to a form
\begin{equation}
    \widetilde{\omega}^\glue_{S, \chi} = (\widetilde{\Gamma}_\chi)^* \widetilde{\Omega}^\AB_T
\end{equation}
on the moduli space $\M_S$.
Abusing notation, we will denote both forms by $\omega^\glue_{S, \chi}$ (see \cref{fig:gluing-map-pullback}).

\begin{figure}[ht]
    \begin{tikzcd}
    &[-25pt] \text{\underline{Cylinder $S$}:} & \text{\underline{Torus $T$}:} \\[-20pt]
    \text{\underline{Sets of connections}:} &  (\A_S, \omega^\glue_{S, \chi}) \arrow[d, two heads] \arrow[r, "\Gamma_\chi"] &  (\A_T, \Omega^\AB_T)\arrow[d, two heads] \\
    \text{\underline{Moduli spaces}:} &  (\M_S, \omega^\glue_{S, \chi}) \arrow[r, "\widetilde{\Gamma}_\chi"] &  (\M_T, \widetilde{\Omega}^\AB_T)                     
    \end{tikzcd}
    \caption{Forms induced by the abelian Atiyah-Bott form $\Omega^\AB_T$ on the torus $T$.}
    \label{fig:gluing-map-pullback}
\end{figure}

The upshot of this construction is that $\omega^\glue_{S, \chi}$ can be computed using the following expressions, which can be thought of as a kind of ``Riemann bilinear identity on the cylinder'' with correction terms involving $\chi$.

\begin{prop}[Glued bilinear identity] \label{prop:glued-form-identity}
\begin{align}
    \omega^\glue_{S, \chi}(\dot{\alpha}_1, \dot{\alpha}_2)
    &=\int_S \dot{\alpha}_1 \wedge \dot{\alpha}_2 + \int_{\partial S} (\dot{\chi}_2 \dot{\alpha}_1 -  \dot{\chi}_1 \dot{\alpha}_2 + \dot{\chi}_1 d \dot{\chi}_2) \label{eq:glued-form-1} \\
    &= \int_a \dot{\alpha}_1 \left( \dot{\chi}_2(p_0) - \dot{\chi}_2(\sigma(p_0)) + \int_b \dot{\alpha}_2 \right) 
    - \int_a \dot{\alpha}_2 \left( \dot{\chi}_1(p_0) - \dot{\chi}_1(\sigma(p_0)) + \int_b \dot{\alpha}_1 \right) \label{eq:glued-form-2}
\end{align}
\end{prop}
\begin{proof}  
    The first expression is obtained by directly expanding \eqref{eq:glued-form-0} and using Stokes' theorem, along with the fact that the $\dot{\alpha}_i$ are closed.

    The second is obtained by viewing $\dot{\alpha}_i - d \dot{\chi}_i$ as forms on the torus $T$ and applying the usual Riemann bilinear identity
    \begin{align*}
      \omega^\glue_{S, \chi}(\dot{\alpha}_1, \dot{\alpha}_2) &= \int_a (\dot{\alpha}_1 - d \dot{\chi}_1) \int_b (\dot{\alpha}_2 - d \dot{\chi}_2) - \int_a (\dot{\alpha}_2 - d \dot{\chi}_2) \int_b (\dot{\alpha}_1 - d \dot{\chi}_1)
    \end{align*}
    (this also uses closedness of the $\dot{\alpha}_i$).
    Viewing the individual components as forms on $S$ again, $\int_a d \dot{\chi}_i = 0$ and $\int_b d \dot{\chi}_i = \dot{\chi}_i(\sigma(p_0)) - \dot{\chi}_i(p_0)$, so the result follows.
\end{proof}

\begin{remark}\label{rem:closedness-of-variations}
    For our later application, note that the calculations in the above proof only used the fact that the variations $\dot{\alpha}_i$ were closed forms, not the stronger condition that the original connection was flat.
\end{remark}

Once we further specialize the gluing setup, we will use \eqref{eq:glued-form-1} to identify $\Omega^\glue$  with the regularized Atiyah-Bott form on $\abAfr_\zeta$, and interpret \eqref{eq:glued-form-2} in terms of the Ooguri-Vafa twistor coordinates $\Xe$ and $\Xm$.

\subsection{\texorpdfstring{Gluing on $\Sigma$ with a cutoff}{Gluing on Sigma with a cutoff}} 
\label{sub:gluing-on-sigma}

Fix $\zeta \in \C^*$ and consider the abelianized connections in $\abAfr_\zeta$ (from \cref{def:ab-framed-connections}).
We will apply the preceding gluing construction to a subset of the spectral curve $\Sigma$ obtained by removing two discs around the punctures, namely $\Sigma_r \coloneqq \Sigma \setminus \pi^{-1}(D_r)$, where $D_r = \{|w| < r\} \subseteq \CP^1$ (see \cref{fig:cylinder-cutoff}).
We will refer to the boundary edge around $\infty_-$ as the top of $\Sigma_r$.
Assume that the cutoff radius $r$ is sufficiently small so that the trivializations \eqref{eq:ab-form-punctures} for $\nabla^\ab$ hold in neighbourhoods of the top and bottom of $\Sigma_r$.

\begin{figure}[ht]
    \centering
    \begin{tikzpicture}[scale=1]
        \fill[gray!20] (0,0) circle (2cm);
        \draw (-2,0) arc (180:360:2 and 0.3);
        \draw[dashed] (2,0) arc (0:180:2 and 0.3);
        
        \fill[white] (1.41,1.41) arc (0:180:1.41 and -0.2);
        \fill[white] (1.41,1.41) arc (0:180:1.41 and 0.6);
        \draw[red, ultra thick] (1.41,1.41) arc (0:180:1.41 and -0.2);
        \draw[red, ultra thick, ->] (1.41,1.41) arc (0:110:1.41 and -0.2) node[below=2pt] {$a$};
        \draw[red, thick, dashed] (1.41,1.41) arc (0:180:1.41 and 0.2);
        \fill[fill=black] (0,2) circle (2pt) node[above right=-2pt] {$\infty_-$};
                
        \fill[white] (1.41,-1.41) arc (0:180:1.41 and -0.6);
        \fill[gray!20] (1.41,-1.41) arc (0:180:1.41 and -0.2);
        \draw[thick] (1.41,-1.41) arc (0:180:1.41 and -0.2);
        \draw[thick, dashed] (1.41,-1.41) arc (0:180:1.41 and 0.2);
         \fill[fill=black] (0,-2) circle (2pt) node[below right] {$\infty_+$};

        \draw[blue, ultra thick] (-1.41, 1.41) arc (135:225:2) node[left, midway] {$b_r$};
        \draw[blue, ultra thick, ->] (-1.41, 1.41) arc (135:180:2);
        \fill[fill=black] (-1.41,1.41) circle (2pt) node[above left=-2pt] {$p_0$};

        \draw (0,0) circle (2cm);
        \node at (0,0) {$\Sigma_r$};

        \draw[->] (2.5,0) -- node[above] {$\pi$} (4, 0) ;

        \fill [gray!20] (4.5, -1.5) rectangle (7.5,1.5);
        \fill[white] (6,0) circle (1cm);
        \draw[->] (4.5,0) -- (7.5, 0);
        \draw[->] (6,-1.5) -- (6, 1.5);
        \draw[ultra thick, red] (6,0) circle (1cm);
        \draw[ultra thick, red, <-] (7, 0) arc (0:180:1);
        \node[above right] at (6,0) {$D_r$};
        \node at (6,-2) {$w$–plane $\subset \CP^1$};
    \end{tikzpicture}

    \caption{The cut off cylinder $\Sigma_r$ inside the spectral curve $\Sigma$, and its projection to $\CP^1$ in a neighbourhood of $w=0$.}
    \label{fig:cylinder-cutoff}
\end{figure}
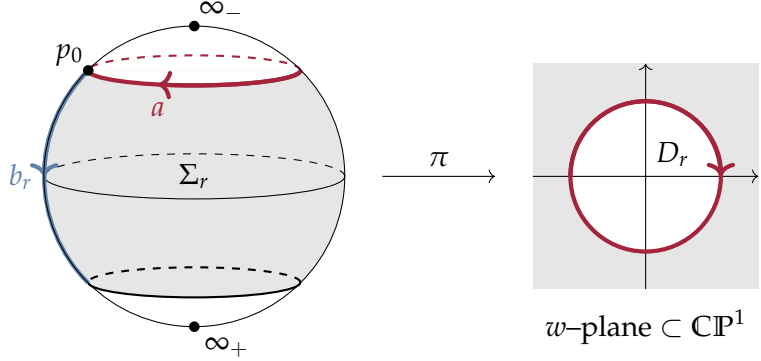

We will explain how to use the gluing construction to obtain the desired form $\Omega^\glue$ on $\abAfr_\zeta$.

\begin{remark}[Variations of abelianized connections] \label{rem:var-of-ab-connection}
There are two preliminary technicalities to address regarding variations of elements $(\L^\ab, \nabla^\ab, g^\ab_\pm) \in \abAfr_\zeta$. 
\begin{itemize}
    \item The spectral curve $\Sigma \supset \Sigma_r$ is itself defined in terms of the parameter $m$, so it varies with $\nabla^\ab$. 
    However, the resulting surfaces are diffeomorphic for all nonzero $m$, so we will identify them all with a fixed reference surface.
    Similarly, we will identify the varying bundles $\L^\ab$ over $\Sigma$ with a fixed (trivial) bundle using the global frame $g^\ab$ from \cref{cor:global-frame}.

    \item The general gluing construction was described for flat connections, but $\nabla^\ab$ is only almost-flat.
    However, its prescribed connection form $\alpha$ is constant in a neighbourhood of the branch points (see \eqref{eq:ab-form-branch-points}), so its variations are zero there, and thus $\dot{\alpha}$ is a closed $1$-form defined on all of $\Sigma_r$.
    The formulas for the glued bilinear identity in \cref{prop:glued-form-identity} therefore still apply (see \cref{rem:closedness-of-variations}).
\end{itemize}
\end{remark}

At first we will consider a slightly more general choice of path $b_r$ that winds around the cylinder $\frac{\vartheta}{2\pi}$ times.
(Eventually we will choose a specific basepoint $p_0$ and angle $\vartheta$ using the geometry of the relevant spectral network, but the construction makes sense more generally.)

We run the construction from \cref{sub:general-gluing-construction} with the following data:

\begin{enumerate}[(a)]
\item \underline{topological setup}:

Fix the cylinder $S = \Sigma_r \subseteq \Sigma$.
Consider the map $\sigma = \sigma_\vartheta: \Sigma \to \Sigma$,
\begin{equation} \label{eq:gluing-diffeo}
    \sigma(w,s) = (e^{i \vartheta} \overline{w}, -s)
\end{equation}
which restricts to an orientation-reversing diffeomorphism between neighbourhoods of the top and bottom of $\Sigma_r$.
Let $\Sigma_{r, \top} \xrightarrow{\sim} \Sigma_{r, \bot}$ be two such neighbourhoods, chosen sufficiently small so that the prescribed framed form \eqref{eq:ab-form-punctures} holds.

As in \cref{fig:cylinder-cutoff}, let $a$ be the boundary loop around $\infty_-$.
Choose a base point $p_0$ along $a$, and let $b_r=b_r(p_0, \vartheta)$ be the open path from $p_0$ to $\sigma(p_0)$, winding around the cylinder $\frac{\vartheta}{2\pi}$ times in the same direction as $a$ (see \cref{fig:twisted-cylinder-path}).

\begin{figure}[ht]
    \centering
    \includegraphics[scale=1]{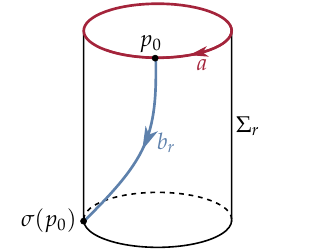}
    \caption{Path $b_r = b_r(p_0, \vartheta)$ winding around $\Sigma_r$, with $\vartheta = \pi/2.$}
    \label{fig:twisted-cylinder-path}
\end{figure}

\pagebreak

\item \underline{framed connections}:

Consider the space $\abAfr_\zeta$ of connections $\nabla^\ab = d + \alpha$, trivialized using the frames $g^\ab$ from \cref{cor:global-frame} so that near the boundary they have the prescribed form
\begin{align}
    \alpha|_\top &= \pi^* \left[ -\zeta^{-1} \frac{dw}{w^3} - \zeta \frac{d \overline{w}}{\overline{w}^3} - (\zeta^{-1} m - \frac{1}{2} m^{(3)}) \frac{dw}{w} - (\zeta \overline{m} + \frac{1}{2} m^{(3)})\frac{d \overline{w}}{\overline{w}} \right],  \label{eq:ab-form-top} \\
    \alpha|_\bot &= \pi^* \left[ \zeta^{-1} \frac{dw}{w^3} + \zeta \frac{d \overline{w}}{\overline{w}^3} + (\zeta^{-1} m - \frac{1}{2} m^{(3)}) \frac{dw}{w} + (\zeta \overline{m} + \frac{1}{2} m^{(3)})\frac{d \overline{w}}{\overline{w}} \right]. \label{eq:ab-form-bot} 
\end{align} 
Note that the allowable gauge transformations in this setting are identically equal to $1$ near $\partial \Sigma_r$ (see \cref{rem:abelian-form-rigidity}), which is consistent with the general gluing construction.

For notational purposes, let
\begin{equation}
    \alpha_0 \coloneqq -\zeta^{-1} \frac{dw}{w^3} - \zeta \frac{d \overline{w}}{\overline{w}^3} - (\zeta^{-1} m - \frac{1}{2} m^{(3)}) \frac{dw}{w} - (\zeta \overline{m} + \frac{1}{2} m^{(3)})\frac{d \overline{w}}{\overline{w}}
\end{equation}
and introduce the sign 
\begin{equation} \label{eq:epsilon-sign}
    \epsilon \coloneqq \begin{cases}
        +1 & \text{on } \Sigma_{r, \top}   \\
        -1 & \text{on } \Sigma_{r, \bot}, 
    \end{cases}
\end{equation}
so that we can write
\begin{equation} \label{eq:ab-form-sign}
    \alpha = \epsilon \cdot \pi^* \alpha_0
\end{equation}
near $\partial \Sigma_r$.
In what follows we will usually suppress $\pi^*$ from the notation.

\item \underline{gluing data}:

Explicitly, the symmetric gluing condition \eqref{eq:sym-gluing-condition} for $\chi$   becomes
\begin{equation}
    d \chi = \frac{\epsilon}{2} \left[ -(\zeta^{-1} + \zeta e^{2i \vartheta})\frac{dw}{w^3} - (\zeta + \zeta^{-1} e^{-2i \vartheta}) \frac{d \overline{w}}{\overline{w}^3}  -  (\zeta^{-1} m + \zeta \overline{m}) \left( \frac{dw}{w} + \frac{d \overline{w}}{\overline{w}} \right) \right].
\end{equation}
In order to specify our choice of $\chi$ it will be useful to introduce some more notation:
\begin{itemize}
    \item Define an antiderivative of $\alpha$ in a neighbourhood of $\partial \Sigma_r$ by 
    \begin{equation} \label{eq:lifted-antiderivative}
        A \coloneqq \epsilon \cdot \pi^* A_0,
    \end{equation}
    where
    \begin{equation*} 
        A_0(w) = \frac{1}{2} \zeta^{-1} w^{-2}  + \frac{1}{2} \zeta \overline{w}^{-2}  - (\zeta^{-1} m - \frac{1}{2} m^{(3)}) \log w - (\zeta \overline{m} + \frac{1}{2} m^{(3)}) \log \overline{w}.
    \end{equation*}
    (Note that $A_0$ is the same function \eqref{eq:normalization-antiderivative} appearing in the normalization condition for the canonical flat sections $s_i$.)
    Then $d A_0 = \alpha_0$, and so $dA = \alpha$ near $\partial \Sigma_r$.

    \item Let
    \begin{equation}
        x_e \coloneqq -2\pi i (\zeta^{-1} m - m^{(3)} - \zeta \overline{m})
    \end{equation}
    so that
    \begin{equation}
        \Xe = \exp (x_e).
    \end{equation}
    Then $A_0$ has monodromy $A_0 \to A_0 + x_e$ around $w=0$.
\end{itemize}

We will choose the gluing map
\begin{equation} \label{eq:sym-gluing-map-antiderivative}
    \boxed{\chi = \frac{A - \sigma^* A}{2} - \frac{\epsilon}{2} \cdot \frac{\vartheta}{2\pi} x_e}
\end{equation}
to satisfy the symmetric condition \eqref{eq:sym-gluing-condition}.
(Note that $\chi$ is single-valued even though $A_0$ has monodromy.)
The $- \frac{\epsilon}{2} \cdot \frac{\vartheta}{2\pi} x_e$ term was chosen so that the explicit expansion
\begin{equation} \label{eq:gluing-map}
    \boxed{ \chi = \frac{\epsilon}{2} \left[ \frac{1}{2} (\zeta^{-1} + \zeta e^{2 i \vartheta}) w^{-2} + \frac{1}{2} (\zeta + \zeta^{-1} e^{-2i \vartheta}) \overline{w}^{-2} - 2 (\zeta^{-1} m + \zeta \overline{m}) \log |w| \right] }
\end{equation}
has no constant term.
\end{enumerate}

\begin{construction}[Glued form $\Omega_r^\glue$] 
    For any choice of path $b_r = b_r(p_0, \vartheta)$ as above, the gluing construction produces a form
    \begin{equation*}
        \Omega^\glue_r \coloneqq \omega^\glue_{\Sigma_r, \chi}
    \end{equation*}
     on $\abAfr_\zeta$ (and its moduli space $\abMfr_\zeta$).
    According to \cref{prop:glued-form-identity}, it can be calculated by
    \begin{align}
    \begin{split}
       \Omega^\glue_r(\dot{\nabla}^\ab_1, \dot{\nabla}^\ab_2)
        &=\int_{\Sigma_r} \dot{\alpha}_1 \wedge \dot{\alpha}_2 
        + \int_{\partial  \Sigma_r} (\dot{\chi}_2 \dot{\alpha}_1 -  \dot{\chi}_1 \dot{\alpha}_2 + \dot{\chi}_1 d \dot{\chi}_2) \label{eq:cutoff-glued-form-1} 
    \end{split} \\
    \begin{split}
        &= \int_a \dot{\alpha}_1 \left( \dot{\chi}_2(p_0) - \dot{\chi}_2(\sigma(p_0)) + \int_{b_r} \dot{\alpha}_2 \right) \\
        &\qquad- \int_a \dot{\alpha}_2 \left( \dot{\chi}_1(p_0) - \dot{\chi}_1(\sigma(p_0)) +  \int_{b_r} \dot{\alpha}_1 \right). \label{eq:cutoff-glued-form-2}
    \end{split}
    \end{align} 
\end{construction}
In \cref{sub:interpreting-magnetic-coordinate} we will specify the correct  choice of path $b_r$, but first we will discuss how to interpret the above integrals more generally.

\subsection{Gluing and regularization} 
\label{sub:gluing-and-regularization}

The two formulas \eqref{eq:cutoff-glued-form-1} and \eqref{eq:cutoff-glued-form-2} for $\Omega^\glue_r$ can each be thought of as corresponding to a certain kind of regularization, as alluded to in \cref{ssub:summary-gluing}.

\subsubsection{Regularization for parallel transport} 
\label{ssub:regularization-for-parallel-transport}

First we will consider the expressions
\begin{equation}
    \int^{\reg, \chi}_{b_r} \alpha \coloneqq \chi(p_0) - \chi(\sigma(p_0)) + \int_{b_r} \alpha
\end{equation}
whose variations appear in \eqref{eq:cutoff-glued-form-2}.

Using the formula \eqref{eq:sym-gluing-map-antiderivative} for $\chi$ in terms of the antiderivative $A$ of $\alpha$ near $\partial \Sigma_r$, we can rewrite
\begin{equation} \label{eq:reg-integral-antiderivatives}
     \int^{\reg, \chi}_{b_r} \alpha = \left(A(p_0) - A(\sigma(p_0)) + \int_{b_r} \alpha\right) - \frac{\vartheta}{2\pi} x_e.
\end{equation}
It follows that the expressions $\int^{\reg, \chi}_{b_r} \alpha$ are independent of $r$ (assuming, as always, that $r$ is sufficiently small so that the setup for the gluing construction is defined).
This suggests an interpretation of the gluing map as a means of regularizing the divergent integral
\begin{equation*}
    \lim_{r \to 0} \int_{b_r} \alpha
\end{equation*}
and calculating the (log of the) holonomy of the framed connection $\nabla = d + \alpha$ along an open path; cf.\ the discussion of regularized parallel transports in \cref{ssub:summary-stokes-data}.
We will show in \cref{sub:interpreting-magnetic-coordinate} that by choosing $b_r$ appropriately we can identify the bracketed expression in \eqref{eq:reg-integral-antiderivatives} with the logarithm of the magnetic coordinate $\Xm$.

Taking variations, we see that none of the factors $(\dot{\chi}(p_0) - \dot{\chi}(\sigma(p_0)) + \int_{b_r} \dot{\alpha})$ or $\int_a \dot{\alpha}$
appearing in the expression \eqref{eq:cutoff-glued-form-2} for $\Omega^\glue_r$ depend on $r$.

\begin{cor}[$r$-independence of $\Omega^\glue_r$] \label{cor:glued-form-r-indep}
    The form $\Omega^\glue_r$ is independent of $r$. 
\end{cor}
We will henceforth drop the $r$ subscript and just write $\Omega^\glue$.

\subsubsection{Regularization for the Atiyah-Bott form} 
\label{ssub:regularization-atiyah-bott}

Switching to the other expression \eqref{eq:cutoff-glued-form-1} for $\Omega^\glue$, we will now consider the boundary integral 
\begin{equation}
    \mathcal{R}^\chi_r(\dot{\nabla}^\ab_1, \dot{\nabla}^\ab_2) \coloneqq \int_{\partial  \Sigma_r} (\dot{\chi}_2 \dot{\alpha}_1 -  \dot{\chi}_1 \dot{\alpha}_2 + \dot{\chi}_1 d \dot{\chi}_2).
\end{equation}

\begin{remark}[Explicit boundary integral]
In the current setting, we have explicit closed formulas for $\chi$ and $\alpha$ near $\partial \Sigma_r$, so we can directly calculate
\begin{equation} \label{eq:boundary-term}
    \begin{split}
    \mathcal{R}^\chi_r(\dot{\nabla}^\ab_1, \dot{\nabla}^\ab_2)
    = -4 \pi i \log r [
    &(\zeta^{-1} \dot{m}_1 - \dot{m}^{(3)}_1 - \zeta \dot{\overline{m}}_1)
    (\zeta^{-1} \dot{m}_2 + \zeta \dot{\overline{m}}_2) \\
    - &(\zeta^{-1} \dot{m}_2 - \dot{m}^{(3)}_2 - \zeta \dot{\overline{m}}_2)
     (\zeta^{-1} \dot{m}_1 + \zeta \dot{\overline{m}}_1)
    ].
    \end{split}
\end{equation}
This coincides with the regularization term $-2 \pi \log r \cdot \mathcal{R}^\ab(\dot{\alpha}_1^\ab, \dot{\alpha}_2^\ab)$ of the regularized abelian Atiyah-Bott form $\Omega^{\reg, \ab}$ (see \cref{lem:reg-terms-coincide}), and so it follows that $\Omega^\glue = \Omega^{\reg, \ab}$.
This proves \cref{prop:glue-equals-reg}.
\end{remark}

It will be useful for later (and perhaps more enlightening) to give a more general argument in terms of the gluing map \eqref{eq:sym-gluing-map-antiderivative}.
The same proof will carry over to the semiflat setting in \cref{ssub:sf-gluing-and-regularization}.

\begin{proof}[Alternative proof of \cref{prop:glue-equals-reg}]
First, we can rewrite
\begin{equation} \label{eq:rewritten-boundary-term}
    \mathcal{R}^\chi_r(\dot{\nabla}^\ab_1, \dot{\nabla}^\ab_2) = \int_a (\dot{\chi}_2 \cdot (\dot{\alpha}_1 + \sigma^* \dot{\alpha}_1)  -  \dot{\chi}_1 \cdot (\dot{\alpha}_2 + \sigma^* \dot{\alpha}_2)),
\end{equation}
using that $\int_{\partial S} (\cdot) = \int_a (\cdot) - \int_a \sigma^* (\cdot)$ and $\sigma^* \chi = -\chi$.
If we express the variations in polar coordinates near each puncture $\infty_\mp$ (cf.\ \eqref{eq:ab-asymptotic-polar-coeffs}) as
\begin{align} 
    \dot{\alpha}_i &= (\pm \mu^\ab_i + \O(r)) d \theta + (\pm \lambda^\ab_i + \O(r)) \frac{dr}{r} \quad \text{for some } \mu^\ab_i, \lambda^\ab_i \in \C \\
    &= \epsilon \left[(\mu^\ab_i + \O(r)) d \theta + (\lambda^\ab_i + \O(r)) \frac{dr}{r}\right],
    \end{align}
then the terms appearing in \eqref{eq:rewritten-boundary-term} are
\begin{equation}
    \dot{\alpha}_i + \sigma^* \dot{\alpha}_i = 2 \epsilon (\mu^\ab_i + \O(r)) d \theta
\end{equation}
and
\begin{equation}
    \dot{\chi}_i = \epsilon (\lambda^\ab_i + \O(r)) \log r.
\end{equation}

We can therefore write
\begin{align*}
    \mathcal{R}^\chi_r(\dot{\nabla}^\ab_1, \dot{\nabla}^\ab_2) &= 2 \log r \int_a \left[ (\lambda^\ab_2 + \O(r))  (\mu^\ab_1 + \O(r))  -  (\lambda^\ab_1 + \O(r))  (\mu^\ab_2 + \O(r)) \right] d \theta   \\
    &= -2 \pi \log r  \cdot \left[ 2 (\mu^\ab_1 \lambda^\ab_2 - \mu^\ab_2 \lambda^\ab_1) + \O(r) \right] \\
    &= -2 \pi \log r  \cdot \left[ \mathcal{R}^\ab (\dot{\nabla}^\ab_1, \dot{\nabla}^\ab_2) + \O(r) \right],
\end{align*}
using the general definition \eqref{eq:ab-reg-term} of $\mathcal{R}^\ab$.
But on the other hand we know that
\begin{equation*}
    \Omega^\glue(\dot{\nabla}^\ab_1, \dot{\nabla}^\ab_2) \defeq \int_{\Sigma_r} \dot{\alpha}_1 \wedge \dot{\alpha}_2  + \mathcal{R}^\chi_r(\dot{\nabla}^\ab_1, \dot{\nabla}^\ab_2)
\end{equation*}
is $r$-independent, so it follows that
\begin{align*}
    \Omega^\glue(\dot{\nabla}^\ab_1, \dot{\nabla}^\ab_2) &= \lim_{r \to 0} \left[ \int_{\Sigma_r} \dot{\alpha}_1 \wedge \dot{\alpha}_2  - 2 \pi \log r \cdot \mathcal{R}^\ab(\dot{\nabla}_1^\ab, \dot{\nabla}_2^\ab ) \right] \\
    &= \Omega^{\reg, \ab} (\dot{\nabla}^\ab_1, \dot{\nabla}^\ab_2). \qedhere
\end{align*}
\end{proof}

\subsection{Interpreting the electric twistor coordinate} 
\label{sub:interpreting-electric-coordinate}

Now we return to the question of interpreting $\Xe$ and $\Xm$ in terms of the integrals
\begin{equation*}
    \int_a \alpha \quad \text{and} \quad \int^{\reg, \chi}_{b_r} \alpha = \chi(p_0) - \chi(\sigma(p_0)) + \int_{b_r} \alpha.
\end{equation*}
We will start with the electric twistor coordinate $\Xe$ as a quick sanity check.
Using the expression \eqref{eq:ab-form-top} for $\alpha|_\top$, we can directly compute 
\begin{align*}
    \int_a \alpha  &= \int_a  \left(-\zeta^{-1} \frac{dw}{w^3} - \zeta \frac{d \overline{w}}{\overline{w}^3} - (\zeta^{-1} m - \frac{1}{2} m^{(3)}) \frac{dw}{w} - (\zeta \overline{m} + \frac{1}{2} m^{(3)})\frac{d \overline{w}}{\overline{w}}\right) \\
    &=  2 \pi i (\zeta^{-1} m - \zeta \overline{m} - m^{(3)}) \\
    &= -x_e,
\end{align*}
which is consistent with the interpretation of $\Xe$ in terms of the formal monodromy of $\nabla_\zeta$ \cite{Tulli:2019}.

We will henceforth call $a = \gamma_e$, so that we have the following integral interpretation of $\Xe$ as a holonomy along $\gamma_e$.

\begin{prop}[$\gamma_e$ integral] \label{prop:electric-integral}
    For a connection $\nabla \in \Afr_\zeta$ with abelianization $\nabla^\ab = d + \alpha$,
    \begin{equation}
        \exp\left( -\int_{\gamma_e} \alpha \right) = \Xe(\nabla).
    \end{equation}
\end{prop}

\subsection{Interpreting the magnetic twistor coordinate} 
\label{sub:interpreting-magnetic-coordinate}

\subsubsection{\texorpdfstring{Choice of path $\gamma_m$}{Choice of path gamma\_m}} 
\label{ssub:choice-of-path}

In order to relate
\begin{equation*}
    \int^{\reg, \chi}_{b_r} \alpha
\end{equation*}
to the magnetic coordinate $\Xm$, the desired path $\gamma_m = b_r(p_0, \vartheta)$ will be described in terms of the spectral network on the cut off double cover $\Sigma_r$, as shown in \cref{fig:cutoff-paths} below.

\begin{figure}[ht]
    \subcaptionbox{$\Re(\zeta^{-1} m) > 0$}{\includegraphics[scale=1]{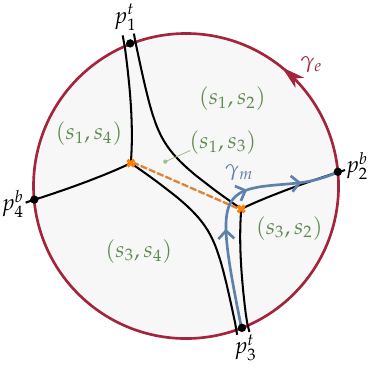}}
    \hspace{1cm}
    \subcaptionbox{$\Re(\zeta^{-1} m) < 0$}{\includegraphics[scale=1]{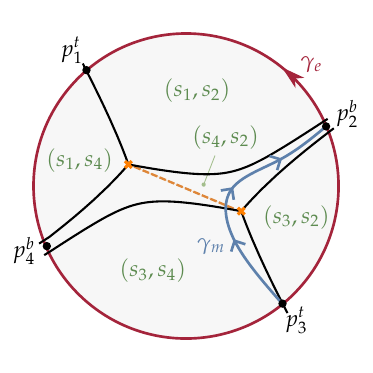}}
    \caption{The relevant paths $\gamma_e, \gamma_m$ and spectral network $\W_\zeta$ on the cut off double cover $\Sigma_r$, with four points $p_i^t$ or $p_i^b$ lying on the indicated sheets and anti-Stokes rays.}
    \label{fig:cutoff-paths}
\end{figure}

Recall that $\gamma_e = a$ is the boundary loop around $\infty_-$, i.e.\ the circle $|w| = r$ on the top sheet.
Define the four points
\begin{equation*}
    p_1^t, p_2^b, p_3^t, p_4^b \in \partial \Sigma_r
\end{equation*}
that lie on the anti-Stokes ray indicated by their subscripts and the sheet indicated by their superscripts.
(We will sometimes omit the superscripts if we do not need to emphasize them.)

Take $\gamma_m$ to be a path from $p_3^t$ to $p_2^b$ that winds around the branch point in the triangle 123 (cf.\ \cite[Section 9.4.3]{Gaiotto:2013a}).
We will choose the gluing angle $\vartheta$ so that the map $\sigma(w,s)  =(e^{i \vartheta} \overline{w}, -s)$ sends the basepoint
\begin{equation*}
    p_0 \coloneqq p_3^t
\end{equation*}
to $p_2^b$ (see \cref{fig:gluing-angle}).

\begin{figure}[ht]
    \centering
    \includegraphics[scale=1]{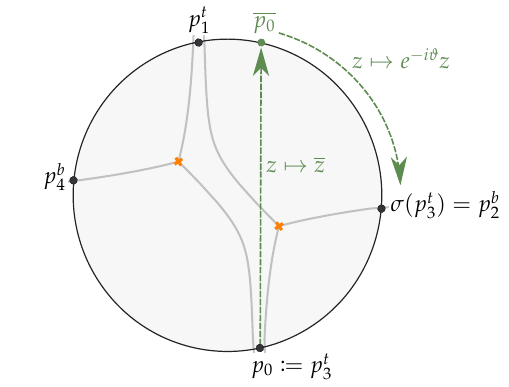} 
    \caption{We choose the gluing angle $\vartheta$ so that the orientation-reversing map $\sigma(w,s)  =(e^{i \vartheta} \overline{w}, -s)$ sends the basepoint $p_3^t$ to $p_2^b$.
    This is illustrated above in $z$-coordinates (when $\Re(\zeta^{-1} m) > 0$).}
    \label{fig:gluing-angle}
\end{figure}

With such a choice of $p_0$ and $\theta$ we will now study the integral
\begin{equation}
    \int^{\reg, \chi}_{\gamma_m} \alpha = \chi(p_0) - \chi(\sigma(p_0)) + \int_{\gamma_m} \alpha.
\end{equation}

\begin{remark}[Variations of spectral network and gluing angle]
    As we vary the parameter $m$, the spectral network also varies, and hence so do the boundary points $p_i$ and gluing angle $\vartheta$. 
    However, the topology of this picture does not change for nearby $m$, and hence we will treat this data as fixed when computing variations of the connection $\nabla^\ab$ (cf.\ \cref{rem:var-of-ab-connection}).
\end{remark}

\subsubsection{\texorpdfstring{Regularized holonomies along $\gamma_m$}{Regularized holonomy along gamma\_m}} 
\label{ssub:regularized-holonomies}

We will prove that $\Xm$ has the following integral interpretation (cf.\ \cref{prop:electric-integral} for $\Xe$).

\begin{prop}[$\gamma_m$ glued integral] \label{prop:magnetic-integral}
    For a connection $\nabla \in \Afr_\zeta$ with abelianization $\nabla^\ab = d + \alpha$,
    \begin{equation}
        \exp \left( -\int^{\reg, \chi}_{\gamma_m} \alpha  \right) 
        = \Xm(\nabla) \cdot (\Xe(\nabla))^{\frac{\vartheta}{2\pi}}.
    \end{equation}
\end{prop}

For our current choice of path $\gamma_m = b_r(p_3^t, \vartheta)$, \eqref{eq:reg-integral-antiderivatives} says that
\begin{align*}
     \exp \left( -\int^{\reg, \chi}_{\gamma_m} \alpha  \right) &= \exp\left( - A(p_3^t) + A(p_2^b) - \int_{\gamma_m} \alpha + \frac{\vartheta}{2\pi} x_e \right) \\
    &= \exp\left( -A(p_3^t) + A(p_2^b) - \int_{\gamma_m} \alpha  \right) \cdot (\Xe)^{\frac{\vartheta}{2\pi}},
\end{align*}
so it suffices to prove the following:
\begin{lemma}[$\Xm$ as regularized holonomy] \label{lem:reg-holonomy-xm}
    \begin{equation}\label{eq:reg-holonomy-xm}
    \exp \left(-A(p_3^t) + A(p_2^b) - \int_{\gamma_m} \alpha \right) = \Xm
\end{equation}
\end{lemma}

We will use the interpretation \eqref{eq:magnetic-coord} of the magnetic twistor coordinate from \cite{Tulli:2019}: 
\begin{equation}
    \Xm(\zeta) = \begin{cases}
     a(\zeta) \equiv \dfrac{s_3 \wedge s_1}{s_2 \wedge s_1}  & \text{when } \Re(\zeta^{-1} m) > 0,  \\\\
    -\dfrac{1}{b(\zeta)} \equiv - \dfrac{s_3 \wedge s_2}{s_4 \wedge s_2} & \text{when } \Re(\zeta^{-1} m) < 0.
    \end{cases}
\end{equation}
The sectorial sections $s_i$ can be described in terms of the abelianized bundle over the spectral cover.
Write $\nabla^\ab = d + \alpha$ with respect to the global frame $g^\ab$ of \cref{cor:global-frame}. 
Near each of the marked boundary points $p_i$ we can define a flat section $\hat{s}_i$ of $\L^\ab$ with respect to $g^\ab$ by
\begin{equation}\label{eq:upstairs-sections}
    \hat{s}_i(q) = c_i \exp\left( -\int_{p_i}^q \alpha \right), 
\end{equation}
where $c_i$ is a constant depending on $p_i$ (and therefore also the cutoff radius $r$).

\begin{lemma}[Normalization for upstairs sections] \label{lem:normalization-constants}
    If we choose the constants
    \begin{equation} \label{eq:section-normalization-constants}
        c_i = \exp(-A(p_i))
    \end{equation}
    in \eqref{eq:upstairs-sections}, then the pushforward of $\hat{s}_i$ to the extended sector $\eSect_i$ coincides with the section $s_i$.
\end{lemma}
\begin{proof}
    The pushed-forward sections have no jumps and are $\nabla$-flat by the abelianization construction,
    so we must only check that they have the right asymptotics \eqref{eq:sectorial-normalization} with respect to the original frame $g$.

    For each $i$, let $t_i$ denote the pushed-forward section $\pi_*(\hat{s}_i)$ and let $\vec{t}_i$ denote its vector representation with respect to the frame $g$, so that $t_i = g \cdot \vec{t}_i$.
    In the cases $i=1$ and $2$ (the others are analogous), we must verify that
    \begin{equation} \label{eq:r1-normalization-condition}
        \vec{t}_1(w) \cdot e^{A_0(w)} \to \begin{pmatrix} 1 \\ 0 \end{pmatrix} 
        \quad \text{as $w \to 0$ in $\eSect_1$}
    \end{equation}
    and
    \begin{equation} \label{eq:r2-normalization-condition}
        \vec{t}_2(w) \cdot e^{-A_0(w)} \to \begin{pmatrix} 0 \\ 1 \end{pmatrix} 
        \quad \text{as $w \to 0$ in $\eSect_2$.}
    \end{equation}
    By \cref{rem:asymptotics-uniqueness}, it suffices to verify these asymptotics as $w \to 0$ in the ``big cells'' $\mathcal{C}_1$ and $\mathcal{C}_2$ shown in \cref{fig:sn-components}.

    We have an explicit antiderivative $A$ for $\alpha$ in a neighbourhood of $\infty_-$ (containing $p_1^t$), so we can write
    \begin{equation*}
        \hat{s}_1(q) = c_1 \exp \left( - \int_{p_1^t}^q \alpha  \right) = c_1 \exp \left( A(p_1^t) - A(q) \right)
    \end{equation*}
    for $q$ near $p_1^t$.
    If we choose $c_1 = \exp(-A(p_1^t))$, the pushed-forward section will be of the form
    \begin{equation*}
        t_1(w) = f_1 \cdot \exp \left(-A_0(w) \right)
    \end{equation*}
    for $w = \pi(q)$ near $0$, where $f_1$ is part of the frame used for abelianization in $\mathcal{C}_1$ (see \cref{subfig:sn-w-frames}).

    To calculate the asymptotics in terms of the original frame $g$, recall that by \cref{prop:asymptotic-existence-g} we can write
    \begin{align*}
        (f_1, f_2) = g \cdot \tilde{\Sigma}_1,
    \end{align*}
    where $\tilde{\Sigma}_1 \to \bbid$ as $w \to 0$ in $\eSect_1$.
    Let $\tilde{\sigma}_1$ denote the first column of $\tilde{\Sigma}_1$.
    Then
    \begin{equation*}
        t_1(w) = g \cdot \underbrace{\tilde{\sigma}_1 \exp \left(-A_0(w) \right)}_{\vec{t}_1(w)},
    \end{equation*}
    and so
    \begin{equation*}
        \vec{t}_1(w) \cdot e^{A_0(w)} = \tilde{\sigma}_1(w) \to \begin{pmatrix} 1 \\ 0 \end{pmatrix},
    \end{equation*}
    as required by \eqref{eq:r1-normalization-condition}.

    The argument for $\hat{s}_2$ is the same, except that
    \begin{equation*}
        t_2(w) = f_2 \cdot \exp (A_0(w))
    \end{equation*}
    due to the flipped signs on the bottom sheet (see \eqref{eq:lifted-antiderivative}).
\end{proof}

We will use these normalizations of the sections to relate the magnetic coordinate to the desired regularized integral.
The formula for $\hat{s}_i$ extends to define a (multivalued) flat section of $\L^\ab$ on the rest of $\Sigma' \supset \Sigma_r$, i.e.\ the punctured surface minus the branch points\footnote{The value of $\hat{s}_i$ depends on the choice of path going around a branch point, since $\nabla^\ab$ is almost-flat.}, which can be used to compute parallel transports.

The Stokes matrix element
\begin{equation*}
    a(\zeta) = \frac{s_3 \wedge s_1}{s_2 \wedge s_1}
\end{equation*}
is naturally described in terms of parallel transport along $\gamma_m$ when $\Re(\zeta^{-1} m) > 0$.
It will be useful to introduce a related path $\gamma_m'$ from $p_4^b$ to $p_3^t$, winding around the same branch point as $\gamma_m$ (see \cref{fig:gamma-m-comparisons}), in order to describe
\begin{equation*}
    b(\zeta) = \frac{s_4 \wedge s_2}{s_3 \wedge s_2}
\end{equation*}
when $\Re(\zeta^{-1} m) < 0$.

\begin{figure}[ht]
    \subcaptionbox{$\Re(\zeta^{-1} m) > 0$}
        {\includegraphics[scale=1]{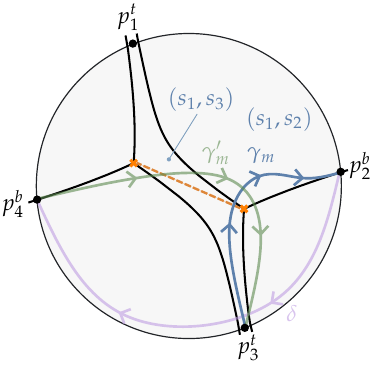}}
    \subcaptionbox{$\Re(\zeta^{-1} m) < 0$}
        {\includegraphics[scale=1]{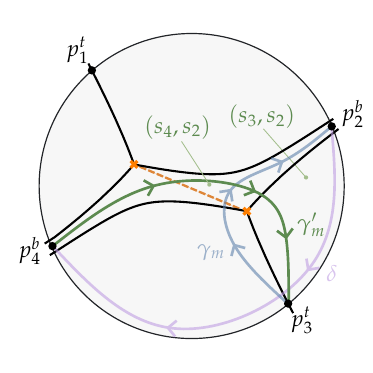}}
    \caption{$\gamma_m$ is related to a path $\gamma_m'$, which can be used to compute $b(\zeta)$ when $\Re(\zeta^{-1} m) < 0$.}
    \label{fig:gamma-m-comparisons}
\end{figure}

\begin{lemma}[Ratios as regularized parallel transport]
For \emph{any} choice of normalization constants $c_i$ and corresponding sections $s_i$,
    \begin{equation} \label{eq:a-wedge-as-integral}
    \frac{c_3}{c_2} \exp\left(-\int_{\gamma_m} \alpha \right) = \frac{s_3 \wedge s_1}{s_2 \wedge s_1}  \quad \textup{when } \Re(\zeta^{-1} m) > 0
    \end{equation} 
    and
    \begin{equation} \label{eq:b-wedge-as-integral}
        \frac{c_4}{c_3} \exp\left(-\int_{\gamma_m'} \alpha \right) = \frac{s_4 \wedge s_2}{s_3 \wedge s_2} \quad \textup{when } \Re(\zeta^{-1} m) < 0.
    \end{equation}

\end{lemma}
\begin{proof}
    For the first case, consider the parallel transport of $\hat{s}_3(p_3^t)$ to $p_2^b$ along the path $\gamma_m$. 
    Since $\hat{s}_3$ is flat, we can directly use its definition \eqref{eq:upstairs-sections} to write the resulting vector as
    \begin{equation*}
        \hat{s}_3(p_2^b) = c_3 \exp \left(- \int_{\gamma_m}\alpha  \right).
    \end{equation*}
    On the other hand, the transported vector can also be written as
    \begin{equation*}
        \frac{s_3 \wedge s_1}{s_2 \wedge s_1} \cdot \underbrace{\hat{s}_2(p_2^b)}_{c_2} 
    \end{equation*}
    by the parallel transport formula in \cref{lem:ab-parallel-transport}, since $\gamma_m$ crosses the wall of the network from the $(s_1, s_3)$ cell to the $(s_1, s_2)$ cell. 
    Setting these equal gives \eqref{eq:a-wedge-as-integral}.
    The other argument is identical.
\end{proof}

\begin{cor}
    If we choose the normalization constants $c_i = \exp(-A(p_i))$ as in \cref{lem:normalization-constants}, then
\begin{equation} \label{eq:regularized-a-wedge-as-integral}
    \exp\left( -A(p_3^t) + A(p_2^b)  - \int_{\gamma_m} \alpha \right) = \frac{s_3 \wedge s_1}{s_2 \wedge s_1} \equiv a(\zeta) \quad \textup{when } \Re(\zeta^{-1} m) > 0
\end{equation}
and
\begin{equation} \label{eq:regularized-b-wedge-as-integral}
    \exp\left( -A(p_4^b) + A(p_3^t)  - \int_{\gamma_m'} \alpha  \right) = \frac{s_4 \wedge s_2}{s_3 \wedge s_2}  \equiv b(\zeta) \quad \textup{when } \Re(\zeta^{-1} m) < 0.
\end{equation}
\end{cor}

To express the second case \eqref{eq:regularized-b-wedge-as-integral} in terms of $\gamma_m$, note that the concatenation
\begin{equation*}
    \gamma_m' + \gamma_m + \delta
\end{equation*}
is homotopic to a small loop around the branch point, 
where $\delta$ is the path from $p_2^b$ to $p_4^b$ along $\partial \Sigma_r$ shown in \cref{fig:gamma-m-comparisons}.
Therefore
\begin{equation}
    \exp\left( \int_{\gamma_m' + \gamma_m + \delta} \alpha  \right) = -1,
\end{equation}
and so
\begin{align*}
    \exp\left( -A(p_3^t) + A(p_2^b)  - \int_{\gamma_m} \alpha \right) 
    &= -\exp\left( -A(p_3^b) + A(p_2^t)  + \int_\delta \alpha + \int_{\gamma_m'} \alpha  \right) \\
    &= -\exp\left( A(p_4^b) - A(p_3^t) +  \int_{\gamma_m'} \alpha  \right) \\
    &= -\frac{1}{b(\zeta)} \quad \text{when } \Re(\zeta^{-1} m) < 0.
\end{align*}
We conclude that
\begin{equation*}
    \exp\left( -A(p_3^t) + A(p_2^b)  - \int_{\gamma_m} \alpha \right) = \Xm(\zeta)
\end{equation*}
in both cases, proving \cref{lem:reg-holonomy-xm} and hence also \cref{prop:magnetic-integral}.

\subsection{Summary of twistor interpretations and completing the proof} 
\label{sub:twistor-summary}

\cref{prop:electric-integral,prop:magnetic-integral} say that for a connection $\nabla \in \Afr_\zeta$ with abelianization $\nabla^\ab = d + \alpha$,
\begin{equation}\label{eq:gamma-e-integral}
     \int_{\gamma_e} \alpha  = - \log \Xe(\nabla) \pmod{2 \pi i}
\end{equation}
and
\begin{equation}\label{eq:gamma-m-integral}
    \chi(p_0) - \chi(\sigma(p_0)) + \int_{\gamma_m} \alpha = - \log \Xm(\nabla) - \frac{\vartheta}{2\pi} \log \Xe(\nabla) \pmod{2 \pi i}.
\end{equation}

Taking variations, we see that the corresponding integrals appearing in the expression \eqref{eq:glued-form-2} for the glued form $\Omega^\glue$ are
\begin{equation*}
    \int_{\gamma_e} \dot{\alpha} = - d \log \Xe (\dot{\nabla})
\end{equation*}
and
\begin{equation*}
    \dot{\chi}(p_0) - \dot{\chi}(\sigma(p_0)) + \int_{\gamma_m} \dot{\alpha} = \left(- d\log \Xm - \frac{\vartheta}{2\pi} d\log \Xe\right) (\dot{\nabla}).
\end{equation*}
It follows that
\begin{align*}
\begin{split}
    \ab^* \Omega^\glue(\dot{\nabla}_1, \dot{\nabla}_2) 
    &\defeq \int_{\gamma_e} \dot{\alpha}_1 \left(\dot{\chi}_2(p_0) - \dot{\chi}_2(\sigma(p_0)) + \int_{\gamma_m} \dot{\alpha}_2 \right) \\
    &\qquad - \int_{\gamma_e} \dot{\alpha}_2 \left( \dot{\chi}_1(p_0) - \dot{\chi}_1(\sigma(p_0)) + \int_{\gamma_m} \dot{\alpha}_1 \right)
\end{split}\\
\begin{split}
    &=  d \log \Xe \wedge d \log \Xm (\dot{\nabla}_1, \dot{\nabla}_2) 
\end{split} \\
\begin{split}
    &\defeq -4 \pi^2 \cdot \Omega^\ov (\dot{\nabla}_1, \dot{\nabla}_2),
\end{split}
\end{align*}
which proves \cref{prop:glue-equals-ov}.

Finally, by combining the results as indicated in the diagram \eqref{eq:forms-cd}, we get the equality of forms
\begin{equation}
    \Omega^\ov = -\frac{1}{4 \pi^2} \ab^* \Omega^\glue = -\frac{1}{4 \pi^2} \ab^* \Omega^{\glue, \reg} = -\frac{1}{4 \pi^2} \Omega^\reg
\end{equation}
on $\Hfr$.
Pulling back to $\Xfr$ via $\NAH_\zeta$ gives
\begin{equation*}
    \Omega^\ov_\zeta = -\frac{1}{4 \pi^2} \Omega^\reg_\zeta,
\end{equation*}
completing the proof of \cref{thm:reg-equals-ov}.

\section{Semiflat analysis} 
\label{sec:sf-analysis}

In this section we will carry out a version of the preceding analysis in order to study the \emph{semiflat} Ooguri-Vafa form $\Omega^{\ov, \sf}_\zeta$.
All of the main objects and spaces defined earlier have natural semiflat analogues (see \cref{fig:summary-of-sf-spaces}).
The overall argument will be of a similar flavour, but with some key technical differences; in particular, we will now have to work directly with the magnetic angle $\theta_m$ instead of just the magnetic twistor coordinate.

\begin{figure}[ht]
    \begin{equation*}
    \begin{tikzcd}
    & & & \color{red} \sfabAfr_\zeta \\
    & & \color{red} \abAfr_\zeta & \\
    & \color{green} \sfHfr \arrow[rr, "\phantom{\NAH_\zeta}" near start] & & \color{orange} \sfAfr_\zeta \arrow[uu, "\phantom{\ab}"'] \\
    \color{green} \Hfr \arrow[rr, "\NAH_\zeta"] \arrow[ru, "\iota^\sf"] & & \color{orange} \Afr_\zeta \arrow[uu, "\ab"' {yshift=-5pt}, crossing over] &                           
    \end{tikzcd}
    \end{equation*}
    \caption{The semiflat analogues of the spaces in \cref{fig:summary-of-spaces}, to be defined below.}
    \label{fig:summary-of-sf-spaces}
\end{figure}
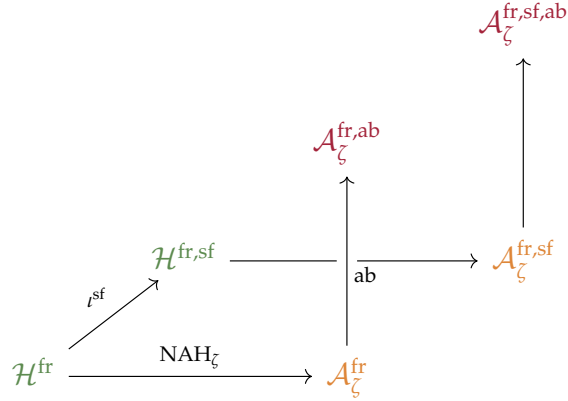

Our main result in this section, analogous to \cref{thm:reg-equals-ov}, will be an identification
\begin{equation*}
    \Omega^{\ov, \shift}_\zeta = - \frac{1}{4\pi^2} \Omega^{\reg, \sf}_\zeta
\end{equation*}
of a \emph{shifted} version of the semiflat Ooguri-Vafa form (\cref{def:shifted-magnetic}) with the regularized Atiyah-Bott form for the corresponding semiflat connections---this is \cref{thm:sf-reg-equals-shifted-ov} below.
The shift involves a modified integral in the definition of the magnetic angle $\theta_m$, which we will study further in \cref{sec:duality-on-hitchin-section} (on the Hitchin section).

The four subsections below correspond to \crefrange{sec:setup}{sec:glued-symplectic-form}.
We will follow our previous steps but try to avoid repetition, instead emphasizing the novel aspects of the calculations.
We begin with a brief summary of the relevant semiflat objects.

\subsection{Overview of semiflat constructions} 
\label{sub:overview-of-sf-constructions}

\subsubsection{Semiflat metrics and harmonic bundles} 
\label{ssub:sf-metrics}

There are various ways to describe the semiflat hyperkähler metric $g_{L^2}^\sf$ on Higgs moduli spaces, such as a twistorial construction in \cite{Gaiotto:2010}, and a more general construction involving the theory of special Kähler manifolds and algebraic integrable systems \cite{Freed:1999a}.
For our purposes it will be most useful to use a characterization in terms of \emph{limiting configurations} $(E, \theta, h_\sf)$ for Hitchin's equation, following \cite{Fredrickson:2022} (see also \cite{Mazzeo:2019,Fredrickson:2020}).\footnote{Recently, \cite{He:2025} has further studied the deformation theory of the moduli space of limiting configurations over singular fibres.}

\begin{remark}[Irregular singularities]
The approach in \cite{Fredrickson:2022} applies in the setting of Higgs bundles with simple poles, where it was proved that the natural $L^2$-metric on the moduli space of limiting configurations coincides with the semiflat metric $g_{L^2}^\sf$ defined via the integrable system structure.
However, the literature is not as developed for higher-order poles.

We will extend the construction of limiting configurations $(E, \theta, h_\sf)$ in a natural way for our setting with irregular singularities, and take this as our \emph{definition} of the ``semiflat harmonic metric'' $h_\sf$.
This approach will allow us to define semiflat analogues of the various objects and spaces considered in the previous sections.
We will not discuss the corresponding $L^2$-metric on the moduli space, but we expect that it coincides with other constructions for the semiflat metric.
\end{remark}

Given a Higgs bundle $(E, \theta)$ over $C = \CP^1$---say, underlying an element $(E, \theta, h, g) \in \Hfr$---consider its associated spectral Higgs line bundle $(\L, \lambda)$ over $\Sigma$, where $\lambda$ is the canonical $1$-form. (Here $\L = \L_\theta$ comes from abelianizing the Higgs bundle; it should not be confused with the line bundles $\L^\ab$ from the previous sections which came from abelianizing the corresponding connections $\nabla_\zeta$.)

As before, let $\Sigma'$ denote $\Sigma$ with the branch points removed.
The semiflat harmonic metric $h_\sf$ will be obtained by pushing forward a  metric from $(\L, \lambda)|_{\Sigma'}$.

\begin{construction}[Semiflat harmonic metric]
    Equip $\L$ with parabolic weights $-\frac{1}{2}$ at the branch points and $\pm m^{(3)}$ at the punctures $\infty_\mp$.
    Then, there is a unique hermitian metric $h_\L$ on $\L|_{\Sigma'}$ such that:
    \begin{enumerate}[(i)]
    \item $F_{D_{h_\L}} = 0$ (i.e.\ the Chern connection corresponding to $h_\L$ is flat).

    \item $h_\L$ is adapted to the parabolic structure (i.e.\  near each branch point or puncture $p$ with parabolic weight $\alpha_p$, there is a local coordinate $w_p$ centred at $p$ and local holomorphic frame $e_p$ for $\L$ such that $h_\L(e_p, e_p) = |w_p|^{2 \alpha_p}$).
    \end{enumerate}
    ($h_\L$ is a solution to the abelian version of the Hitchin equation \eqref{eq:hitchin} with suitable boundary conditions.)
    Let $h_\sf$ denote the orthogonal pushforward of $h_\L$ to $E \cong \pi_* \L$, where as usual $\pi: \Sigma \to C$.
    We call $h_\sf$ the \emph{semiflat harmonic metric} for $(E, \theta)$.
\end{construction}

We emphasize that $h_\sf$ is singular at the branch points of $C$, unlike $h$.
Also note that by construction, the $\theta$-eigenspaces are orthogonal with respect to $h_\sf$, and $h_\sf$ solves the decoupled Hitchin equations
\begin{equation} \label{eq:decoupled-hitchin}
    F_{D_{h_\sf}} = [\theta, \theta^{\dagger_{h_\sf}}] = 0.
\end{equation} 

Now we can define semiflat versions of the spaces of harmonic bundles from \cref{sub:framed-harmonic-bundles}.
Instead of belabouring the point by fully repeating the definitions, we will just stress the main differences.

\begin{definition}[Semiflat harmonic bundles in $\H_\sf$]
    Define a set $\H_\sf$ of \emph{semiflat harmonic bundles} $(E, \theta, h_\sf)$ by making the appropriate replacements of $h$ with $h_\sf$ in \cref{def:harmonic-bundles} for $\H$.
    Now, these are bundles over $C' \setminus \{\infty \}$, where $C'$ denotes $C$ with the branch points removed.
\end{definition}

\begin{definition}[Framed semiflat harmonic bundles in $\sfHfr$]
    Define a set $\sfHfr$ of \emph{compatibly framed semiflat harmonic bundles} $(E, \theta, h_\sf, g)$ by similarly modifying \cref{def:framed-harmonic} for $\Hfr$.
    Now, we will say that a frame $g$ at $\infty$ is compatible if it extends to an \emph{$h_\sf$-unitary frame of eigenvectors} for $\theta$, with respect to which
    \begin{equation} \label{eq:sf-higgs-framed-form}
        \theta = \begin{pmatrix}
            \sqrt{z^2+2m} & 0 \\
            0 & - \sqrt{z^2+2m} 
        \end{pmatrix} dz =  -H \frac{dw}{w^3} - mH \frac{dw}{w} + \text{diagonal regular terms}.
    \end{equation}
    (cf.\ the compatibly framed form \eqref{eq:higgs-framed-form} in $\Hfr$, where $g$ was an $h$-unitary frame diagonalizing just the singular part of $\theta$.)

    Let $\sfXfr$ denote the set of isomorphism classes in $\sfHfr$.
\end{definition}

\begin{lemma}[Comparison map] \label{lem:comparison-map}
    There is a natural map
    \begin{align} \label{eq:comparison-map}
    \begin{split}
        \iota^\sf: \Hfr &\to \sfHfr \\
        (E, \theta, h, g) &\mapsto (E, \theta, h_\sf, g),
    \end{split}
    \end{align}
    which descends to a map of moduli spaces
    \begin{align}
    \begin{split}
        \iota^\sf: \Xfr &\to \sfXfr \\
        [(E, \theta, h, g)] &\mapsto [(E, \theta, h_\sf, g)].
    \end{split}
    \end{align}
\end{lemma}
\begin{proof}
    Unwinding the definitions, the main nontrivial statement here is that the original frame $g$ admits the desired extension to an $h_\sf$-unitary eigenframe for $\theta$; this is proved in \cref{lem:extension-to-eigenframe}.
\end{proof}

The map $\iota^\sf$ can be used to pull back semiflat objects (e.g.\ symplectic forms) to $\Hfr$ or $\Xfr$, where they can be compared with their standard counterparts.
The analogous (unframed) map in \cite{Mazzeo:2019} gives a diffeomorphism of moduli spaces, and is used to compare the Hitchin metric $g_{L^2}$ to the semiflat metric $g_{L^2}^\sf$.

\subsubsection{Semiflat connections} 
\label{ssub:sf-connections}

Now we define analogues of the framed connections from \cref{sub:framed-connections}.

Given $(E, \theta, h_\sf, g) \in \sfHfr$ and $\zeta \in \C^*$, we can consider the corresponding connection
\begin{equation} \label{eq:sf-connection}
    \nabla^\sf_\zeta \coloneqq \zeta^{-1} \theta + D_{h_\sf} + \zeta \theta^{\dagger_{h_\sf}}.
\end{equation}
Note that $\nabla^\sf_\zeta$ is singular at the two branch points as well as the puncture at $\infty$, but it is flat over their complement in $C = \CP^1$ (since $h_\sf$ satisfies the decoupled Hitchin equations \eqref{eq:decoupled-hitchin}).

With respect to the (extension of the) frame $g$ near $w=0$, the Chern connection for $h_\sf$ is of the form
\begin{equation} \label{eq:sf-chern-framed-form}
    D_{h_\sf} = d +  \frac{m^{(3)}}{2} H \left(  \frac{dw}{w} - \frac{d \overline{w}}{\overline{w}} \right),
\end{equation}
and hence
\begin{align}
    \begin{split}
        \nabla^\sf_\zeta &= d + \left[\zeta^{-1} \left(- \frac{dw}{w^3} - m \frac{dw}{w}  \right) + \frac{m^{(3)}}{2}  \left( \frac{dw}{w} - \frac{d \overline{w}}{\overline{w}} \right) + \zeta \left(- \frac{d \overline{w}}{\overline{w}^3} - \overline{m} \frac{d \overline{w}}{\overline{w}}\right)\right]H  \\
    &\qquad + \text{diagonal regular terms}
    \end{split} \\
    \begin{split} \label{eq:sf-framed-connection-form}
        &= d + \left[- \zeta^{-1} \frac{dw}{w^3}  - \zeta \frac{d \overline{w}}{\overline{w}^3} - (\zeta^{-1} m - \frac{1}{2} m^{(3)}) \frac{dw}{w} -  (\zeta \overline{m} + \frac{1}{2} m^{(3)})\frac{d \overline{w}}{\overline{w}}  \right]H \\
        &\qquad+ \text{diagonal regular terms}.
    \end{split}
\end{align}
We emphasize that although these are the same singular terms that appeared in our original calculations, we are now working with a different metric $h_\sf$ and extension of the frame $g$.

\begin{definition}[Framed semiflat bundles in $\sfAfr_\zeta$]
    For fixed $\zeta \in \C^*$, let $\sfAfr_\zeta$ denote the set of \emph{$\zeta$-compatibly framed semiflat bundles} $(E, \nabla^\sf, g)$ with framing of the form \eqref{eq:sf-framed-connection-form} near $\infty$, analogously to \cref{def:framed-connections} for $\Afr_\zeta$.
\end{definition}

As before, we will denote the nonabelian Hodge map by
\begin{align} \label{eq:sf-NAH-map-set}
\begin{split}
    \NAH_\zeta: \sfHfr &\to \sfAfr_\zeta \\ 
    (E, \theta, h_\sf, g) &\mapsto (E, \nabla^\sf_\zeta, g).
\end{split}
\end{align}

Before discussing the abelian counterpart of $\sfAfr_\zeta$, we will fill in some gaps from our earlier discussion of the semiflat Ooguri-Vafa form.

\subsubsection{Semiflat Ooguri-Vafa form and the magnetic angle} 
\label{ssub:sf-ov-form}

As promised in \cref{sub:ov-correspondence}, we will now describe the explicit formula for the magnetic angle $\theta_m$ from \cite{Tulli:2019}.
Given a framed Higgs bundle $(E, \theta, h, g) \in \Hfr$ with $m \neq 0$, let $\phi_2 = -\det \theta = (z^2 + 2m) dz^2$ be the corresponding quadratic differential.
Choose a square root
\begin{equation} \label{eq:phi-square-root}
    \lambda_0 \coloneqq \sqrt{\phi_2} = \sqrt{z^2 + 2m} \, dz
\end{equation}    
with branch cut between $z=\pm \sqrt{-2m}$, so that near $w=0$
\begin{equation} \label{eq:phi-square-root-expansion}
    \lambda_0 = - \frac{dw}{w^3} - m \frac{dw}{w} + \text{regular terms}.
\end{equation}
Following the setup in \cite[Section 3.6.1]{Tulli:2019}:

\begin{itemize}
    \item Let $\gamma$ be a WKB curve for $\phi_2$ with phase $e^{i \arg(m)}$,\footnote{i.e.\ a $\vartheta$-trajectory of phase $\vartheta = \arg(m)$, in the language of \cref{def:theta-traj}} such that $\lambda_0(\dot{\gamma}) = e^{i \arg (m)}$.
    \begin{itemize}
        \item Such a curve crosses the branch cut for $\lambda_0$.
        As $t \to \infty$ and for $\Re(\zeta^{-1} m) > 0$, $\gamma(t)$ lies in either
        \begin{equation*}
          \eSect_1(\zeta) \cap \eSect_4(\zeta)  \quad \text{or} \quad \eSect_2(\zeta) \cap \eSect_4(\zeta), 
        \end{equation*}
        i.e.\ the sector centred around the anti-Stokes ray $r_1$ or $r_3$ (see \cref{subfig:sector-intersections}).

        We assume that $\gamma$ is oriented from $\eSect_2(\zeta) \cap \eSect_4(\zeta)$ to $\eSect_1(\zeta) \cap \eSect_4(\zeta)$ .

         \item In fact, we can (and will) choose $\gamma$ to be a straight line of the form 
        \begin{equation} \label{eq:wkb-curve-choice}
            \gamma_\WKB(t) = \rho(t) e^{i \arg(m)/2}
        \end{equation}
        for an appropriate real-valued function $\rho$ (see \cref{subfig:wkb-direction}).

        \item For later use, let $\Gamma_\WKB$ denote the lift of $\gamma$ to the spectral cover $\Sigma$ from $\infty_+$ to $\infty_-$. (We choose this ordering of sheets in order to be consistent with the choice of square root branch along $\gamma$ in \cite{Tulli:2019}.)
    \end{itemize}

    \begin{figure}[h]
    \subcaptionbox{$\gamma(t)$ lies in the indicated sectors as $t \to \pm \infty$.\label{subfig:sector-intersections}}[0.33\textwidth]
    {\begin{tikzpicture}[scale=1]
    \node at (-2,2.5) {\fbox[rb]{$w$}};

    \draw[fill=green!50] (0,0) -- (45:-1.7) arc(45:135:-1.7) -- cycle;
    \draw[fill=red!50] (0,0) -- (45:1.7) arc(45:135:1.7) -- cycle;

    \node[below right, green] at (0.5, -1.7) {$\eSect_1 \cap \eSect_4$};
    \node[above right, red] at (0.5, 1.7) {$\eSect_2 \cap \eSect_3$};

    \draw[thick] (-2, 0) -- (2, 0);
    \draw[thick] (0, -2) -- (0, 2);
    \draw[thick, dashed, gray] (-1.8, -1.8) -- (1.8, 1.8);
    \draw[thick, dashed, gray] (-1.8, 1.8) -- (1.8, -1.8);

    \draw[very thick, green] (0,0) -- (45:-1.7) arc(45:135:-1.7) -- cycle;
    \draw[very thick, red] (0,0) -- (45:1.7) arc(45:135:1.7) -- cycle;

    \node[black, above, red] at (0, 2) {$r_3$};
    \node[black, left] at (-2, 0) {$r_4$};
    \node[black, below, green] at (0, -2) {$r_1$};
    \node[black, right] at (2, 0) {$r_2$};
    \end{tikzpicture}}
    \hspace{1cm}
    \subcaptionbox{The chosen WKB curve $\gamma$ and possible directions for $r_1$ and $r_3$.\label{subfig:wkb-direction}}[0.4\textwidth]
    {\includegraphics[scale=1]{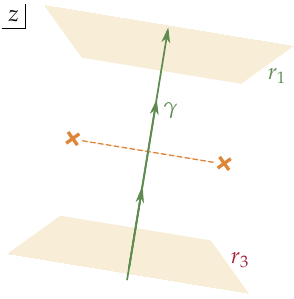}}

    \caption{The WKB curve $\gamma$ is chosen to be a straight line $\rho(t) e^{i \arg(m)/2}$, oriented so that it goes from the sector centred around $r_3$ in \subref{subfig:sector-intersections} to the sector around $r_1$ when $\Re(\zeta^{-1} m) > 0$.
    Conversely, the anti-Stokes directions $r_1$ and $r_3$ must lie somewhere within the yellow shaded region in \subref{subfig:wkb-direction}.}
    \label{fig:wkb-curve}
    \end{figure}

    \item Choose a frame $(\eta_1, \eta_2)$ of $\theta$-eigenvectors along $\gamma$ such that
    \begin{equation} \label{eq:wkb-eigenframe-normalization}
        (\eta_1, \eta_2)|_{\gamma(t)} \to \begin{cases}
            (e_1, e_2) & \text{as } t \to \infty,  \\
            (e_2, -e_1) & \text{as } t \to -\infty,
        \end{cases}
    \end{equation}
    where $g = (e_1, e_2)$ is the specified frame for $E$ at $\infty$.
    We will make a particular choice of $(\eta_1, \eta_2)$ in \cref{rem:pushforward-frame} below.

    \item Let $A_h$ denote the connection form of the Chern connection $D_h$ with respect to $(\eta_1, \eta_2)$.
\end{itemize} 
\begin{definition}[$\theta_m$ on $\Hfr$] \label{def:magnetic-angle}
    With setup as above, the magnetic angle $\theta_m$ on $\Hfr$ is given by
    \begin{equation} \label{eq:magnetic-angle}
        \theta_m(E, \theta, h, g) = m^{(3)} \arg(-m) + \pi + \Im \int_\gamma (A_h)_{11} \pmod {2\pi}.
    \end{equation}
\end{definition}
This formula is isomorphism invariant and hence descends to the moduli space $\Xfr \cong \Mov$, where it agrees with the Ooguri-Vafa magnetic coordinate under the identification in \cite[Section 4]{Tulli:2019}.

For reference we also note the following fact from \cite{Tulli:2019}.
\begin{prop}[$U(1)$-action and $\theta_m$] \label{prop:u1-action-magnetic-angle}
    The $U(1)$-action of $e^{i \vartheta}$ on $\Xfr$ shifts $\theta_m \to \theta_m + \vartheta$; that is,
    \begin{equation}
        \theta_m(e^{i \vartheta} \cdot [(E, \theta, h, g)]) = \theta_m([(E, \theta, h, g)]) + \vartheta.
    \end{equation}
\end{prop}

We now have a fully explicit definition of the semiflat Ooguri-Vafa form
\begin{equation*}
    \Omega^{\ov, \sf}_\zeta = -\frac{1}{4 \pi^2} d \log \Xe(\zeta) \wedge d \log \Xm^{\sf}(\zeta).
\end{equation*}
In terms of the parameters on $\Hfr$, the semiflat magnetic twistor coordinate \eqref{eq:sf-ov-magnetic-coord} can be written
\begin{equation} \label{eq:sf-magnetic-coord-abbrv} 
    \Xm^{\sf}(\zeta) = \exp \left( \zeta^{-1} Z_B    + i \theta_m + \zeta  \overline{Z_B} \right),
\end{equation}
where
\begin{equation} \label{eq:magnetic-period}
    Z_B \coloneqq -m \log \left(\frac{m}{-2e} \right).
\end{equation}

We will interpret the twistor coordinates $\Xe$ and $\Xm^\sf$ in terms of certain integrals of abelian connections, as in \cref{sub:twistor-summary}.

\subsection{Semiflat abelianization and framing} 
\label{sub:sf-abelianization-and-framing}

Continuing with our argument from before, the next step is to abelianize the framed connections in $\sfAfr_\zeta$.

Let $(E, \theta, h_\sf, g) \in \sfHfr$ and $\zeta \in \C^*$.
By construction, the connection
\begin{equation*} 
    \nabla^\sf_\zeta = \zeta^{-1} \theta + D_{h_\sf} + \zeta \theta^{\dagger_{h_\sf}}
\end{equation*}
is the pushforward of the abelian connection 
\begin{equation} \label{eq:sf-abelian-connection}
    \nabla^{\sf, \ab}_\zeta \coloneqq \zeta^{-1} \lambda + D_{h_\L} + \zeta \overline{\lambda}
\end{equation}
on the spectral line bundle $\L$ over $\Sigma$.
In \cref{sec:sf-glued-symplectic-form} we will apply the gluing procedure to these connections.

We must first explain how to lift the frame $g$ from $E$ to $\L$, as in \cref{sub:framing-near-the-punctures}. 
But now this is automatic: the compatible frame $g$ extends (by definition) to a frame of $\theta$-eigenvectors, which define sections of the spectral line bundle $\L$.

Recall that with respect to the frame $g$ near $w=0$,
\begin{align*}
    \theta &= \begin{pmatrix}
        \lambda_0 & 0 \\ 
        0 & - \lambda_0
    \end{pmatrix} =  -H \frac{dw}{w^3} - m H \frac{dw}{w} + \text{diagonal regular terms}
\end{align*}
where $\lambda_0 = \sqrt{z^2 + 2m} \, dz$ as in \eqref{eq:phi-square-root},  
and
\begin{align*}
    D_{h_\sf} &= d +  \frac{m^{(3)}}{2} H \left(  \frac{dw}{w} - \frac{d \overline{w}}{\overline{w}} \right).
\end{align*}
Since these are already diagonal with respect to $g$, we do not need to make any additional modifications to the frame.

\begin{cor}[Induced frame for $\L$] \label{cor:induced-sf-frame}
    In a neighbourhood of each puncture $\infty_\pm$ of $\Sigma$, there exists a frame $g^\L_\pm$ for $\L$ with respect to which 
    \begin{equation} \label{eq:induced-sf-frame-form}
        \nabla^{\sf, \ab}_\zeta = d \pm \pi^*\left[ -\zeta^{-1} \lambda_0 -  \left(  \frac{m^{(3)}}{2} \left(  \frac{dw}{w} - \frac{d \overline{w}}{\overline{w}} \right)  \right) - \zeta \overline{\lambda_0} \right].
    \end{equation}
\end{cor}

\begin{definition}[Framed abelian semiflat bundles in $\sfabAfr_\zeta$]
    For fixed $\zeta \in \C^*$, let $\sfabAfr_\zeta$ denote the set of abelian semiflat bundles $(\L, \nabla^{\sf, \ab}, g^\L_\pm)$ with framing of the form \eqref{eq:induced-sf-frame-form} near $\infty_\pm$, analogously to \cref{def:ab-framed-connections} for $\abAfr_\zeta$.
\end{definition}

\begin{remark}[Framed semiflat form]
    The framed form \eqref{eq:induced-sf-frame-form} for the connections in $\sfabAfr_\zeta$ is also quite rigid, but in a slightly different way than $\abAfr_\zeta$ (cf.\ \cref{rem:abelian-form-rigidity}):
    \begin{itemize}
        \item For $\abAfr_\zeta$, the connections had prescribed singular terms with no additional regular terms. 

        \item For $\sfabAfr_\zeta$, the connections now \emph{do} have additional regular terms, but they are fully explicit (in terms of $\lambda_0$).
    \end{itemize}
\end{remark}

Next, as in \cref{sub:extending-frames}, we extend these local frames to a global frame $g^\L$ defined away from the punctures and branch points.
Note that over each sheet of $\Sigma$, the frame $g^\L$ pushes down to $E$ to give an eigensection for $\theta$ with corresponding eigenvalue $\pm \lambda_0$. 

\begin{remark}[Pushed-forward frame along $\gamma$] \label{rem:pushforward-frame}
    The pushforward of $g^\L$ along the path $\Gamma_\WKB$ gives an eigenframe $(\eta_1, \eta_2)$ along the WKB curve $\gamma$ which satisfies the normalization condition \eqref{eq:wkb-eigenframe-normalization}.
    We therefore can (and will) use this frame to compute the integral appearing in the formula \eqref{eq:magnetic-angle} for the magnetic angle $\theta_m$.
\end{remark}

\subsection{Regularized semiflat Atiyah-Bott forms} 
\label{sub:sf-regularized-atiyah-bott-forms}

We can define regularized Atiyah-Bott forms $\Omega^\reg$ and $\Omega^{\reg, \ab}$ on the spaces of semiflat connections $\sfAfr_\zeta$ and $\sfabAfr_\zeta$ exactly as we did in \cref{sec:regularized-forms}, namely by
 \begin{align}
    \Omega^\reg(\dot{\nabla}^\sf_1, \dot{\nabla}^\sf_2) &= \lim_{R \to 0} \left[\int_{C_R} \tr (\dot{\nabla}^\sf_1 \wedge \dot{\nabla}^\sf_2) - 2 \pi \log R \cdot \mathcal{R}(\dot{\nabla}^\sf_1, \dot{\nabla}^\sf_2) \right], \\
    \mathcal{R}(\dot{\nabla}^\sf_1, \dot{\nabla}^\sf_2) &= \tr \left( \mu^\sf_1 \lambda^\sf_2 - \mu^\sf_2 \lambda^\sf_1 \right), \\
    \dot{\nabla}^\sf_i &= (\mu^\sf_i + \O(r)) d \theta + (\lambda^\sf_i + \O(r)) \frac{dr}{r} \quad \text{for some } \mu^\sf_i, \lambda^\sf_i \in \mathfrak{h},
\end{align}
and
\begin{align}
    \Omega^{\reg, \ab}(\dot{\nabla}^{\sf, \ab}_1, \dot{\nabla}^{\sf, \ab}_2) &= \lim_{R \to 0} \left[ \int_{\Sigma_R} \dot{\nabla}^{\sf, \ab}_1 \wedge \dot{\nabla}^{\sf, \ab}_2 - 2 \pi \log R \cdot \mathcal{R}^\ab(\dot{\nabla}^{\sf, \ab}_1, \dot{\nabla}^{\sf, \ab}_2)  \right], \\
    \mathcal{R}^\ab (\dot{\nabla}^{\sf, \ab}_1, \dot{\nabla}^{\sf, \ab}_2) &= 2( \mu^{\sf, \ab}_1 \lambda^{\sf, \ab}_2 - \mu^{\sf, \ab}_2 \lambda^{\sf, \ab}_1), \\
    \dot{\nabla}^{\sf, \ab}_i &= (\pm \mu_i^{\sf, \ab} + \O(r))  d \theta + (\pm \lambda_i^{\sf, \ab} + \O(r)) \frac{dr}{r} \quad \text{for } \mu_i^{\sf, \ab}, \lambda_i^{\sf, \ab} \in \C.
\end{align}
Note that the forms themselves are given by the same formulas as before---only the underlying spaces of connections are different---so we will still denote them by $\Omega^\reg$ and $\Omega^{\reg, \ab}$ without an ``$\sf$'' subscript.

We define the \emph{regularized semiflat Atiyah-Bott form} $\Omega^{\reg, \sf}_\zeta$ on $\Hfr$ by pulling back $\Omega^\reg|_{\sfAfr_\zeta}$ via the composition
\begin{equation}
    \NAH^\sf_\zeta \coloneqq {\NAH_\zeta} \circ {\iota^\sf} \ : \ \Hfr \xlongrightarrow{\iota^\sf} \sfHfr \xrightarrow{\NAH_\zeta} \sfAfr_\zeta,
\end{equation}
where $\iota^\sf$ is the comparison map from \cref{lem:comparison-map}.
This allows us to compare $\Omega^\reg_\zeta$ and $\Omega^{\reg, \sf}_\zeta$ on the same underlying space (see \cref{fig:sf-regularized-forms}). 
By the usual abuse of notation we will also let $\Omega^{\reg, \sf}_\zeta$ denote the induced form on the moduli space $\Xfr$.

\begin{figure}[ht]
    \begin{equation*}
    \begin{tikzcd}
    & \color{green} \sfHfr \arrow[rr, "\NAH_\zeta"] && \color{orange} (\sfAfr_\zeta,  \Omega^\reg) \\
    \color{green} (\Hfr, \Omega^\reg_\zeta,   \Omega^{\reg, \sf}_\zeta ) \arrow[rr, "\NAH_\zeta"] \arrow[ru, "\iota^\sf"] && \color{orange} (\Afr_\zeta, \Omega^\reg ) &        
    \end{tikzcd}
    \end{equation*}
    \caption{Forms induced on $\Hfr$ by the regularized forms $\Omega^\reg|_{\Afr_\zeta}$ and $\Omega^\reg|_{\sfAfr_\zeta}$.}
    \label{fig:sf-regularized-forms}
\end{figure}

\begin{remark}[Semiflat regularization terms]
    If $\dot{\nabla}_i$ and $\dot{\nabla}_i^\sf$ are variations of $\nabla_\zeta = \NAH_\zeta(E, \theta, h, g)$ and $\nabla^\sf_\zeta = \NAH^\sf_\zeta(E, \theta, h, g)$ respectively, induced by variations of $(E, \theta, h, g) \in \Hfr$,
    then
    \begin{equation}
         \mathcal{R}(\dot{\nabla}_1, \dot{\nabla}_2) = \mathcal{R}(\dot{\nabla}^\sf_1, \dot{\nabla}^\sf_2),
    \end{equation}
    i.e.\, the regularization terms of $\Omega^\reg_\zeta$ and $\Omega^{\reg, \sf}_\zeta$ coincide.
\end{remark}

We also get the expected analogue of \cref{prop:ab-symplectomorphism}:

\begin{prop}[``Semiflat abelianization is a symplectomorphism''] \label{prop:sf-ab-symplectomorphism}
As forms on $\sfAfr_\zeta$,
\begin{equation}
    \ab^* \Omega^{\reg, \ab} = \Omega^\reg,
\end{equation}
i.e.\
\begin{equation}
    \Omega^\reg(\dot{\nabla}^\sf_1, \dot{\nabla}^\sf_2) = \Omega^{\reg, \ab}(\dot{\nabla}^{\sf, \ab}_1, \dot{\nabla}^{\sf, \ab}_2).
\end{equation}
\end{prop}
\begin{proof}
    This time the result is essentially immediate, since the semiflat connections $\nabla^\sf_\zeta$ are the pushforwards of $\nabla^{\sf, \ab}_\zeta$.
    (Alternatively we could follow the proof of \cref{prop:ab-symplectomorphism}, except now the Stokes matrices $S$ are just the identity.)
\end{proof}

\subsection{Semiflat glued symplectic form} 
\label{sec:sf-glued-symplectic-form}

Our goal here, analogously to \cref{sec:glued-symplectic-form}, is to produce a semiflat glued form $\Omega^{\glue, \sf}$ on $\sfabAfr_\zeta$ as an intermediary between the regularized abelian Atiyah-Bott form  and the semiflat Ooguri-Vafa form.
Once again the glued form will coincide with the regularized one, but its identification with the Ooguri-Vafa form will be slightly modified as a result of replacing $h$ with $h_{\sf}$ in the definitions.

We start by introducing a semiflat counterpart of the magnetic angle $\theta_m$ from \cref{def:magnetic-angle}.

\begin{definition}[Shifted magnetic angle and form] \label{def:shifted-magnetic}
    \begin{enumerate}
        \item[] 
    \item Define the \emph{magnetic angle} $\theta_m$ on $\sfHfr$ by
    \begin{equation} \label{eq:magnetic-angle-on-sf}
         \theta_m(E, \theta, h_\sf, g) \coloneqq  m^{(3)} \arg(-m)  + \pi + \Im \int_\gamma (A_{h_\sf})_{11} \pmod {2\pi},
     \end{equation}
    where $A_{h_\sf}$ is the connection form of the semiflat Chern connection $D_{h_\sf}$ with respect to the eigenframe $(\eta_1, \eta_2)$ from \cref{ssub:sf-ov-form}. 
    (The analogous \cref{def:magnetic-angle} for $\theta_m$ on $\Hfr$ uses $A_h$ instead of $A_{h_\sf}$.)

    \item Define a \emph{shifted magnetic angle} $\theta_m^\shift$ on $\Hfr$ by the pullback
    \begin{equation}
        \theta_m^\shift \coloneqq (\iota^\sf)^* (\theta_m|_{\sfHfr}),
    \end{equation}
    i.e.\
    \begin{equation} \label{eq:shifted-magnetic-angle}
         \theta_m^\shift(E, \theta, h, g) =  m^{(3)} \arg(-m)  + \pi + \Im \int_\gamma (A_{h_\sf})_{11} \pmod {2\pi}.
     \end{equation}

    \item  Define the corresponding \emph{shifted semiflat magnetic twistor coordinate} on $\Hfr$ by
    \begin{equation} \label{eq:shifted-magnetic-coord} 
        \Xms(\zeta) \coloneqq \exp \left( \zeta^{-1} Z_B    + i \theta_m^\shift + \zeta  \overline{Z_B} \right)
    \end{equation}
    and \emph{shifted semiflat Ooguri-Vafa form}
    \begin{equation}
        \Omega^{\ov, \shift}_\zeta \coloneqq -\frac{1}{4 \pi ^2} d \log \Xe(\zeta) \wedge d \log \Xms(\zeta).
    \end{equation}
    \end{enumerate}
\end{definition}

\begin{remark}[$\shift$ vs $\sf$]
    The (perhaps subtle) choice of notation above is deliberate.
    The definition \eqref{eq:magnetic-angle-on-sf} of $\theta_m$ on $\sfHfr$ is really the natural analogue of $\theta_m$ on $\Hfr$, given by the same formula but with a different underlying harmonic metric.
    In this sense, $\Xms$ is the natural semiflat magnetic coordinate from the point of view of $\Hfr$.

    However, it is not \emph{a priori} obvious that $\Xms$ does (or should) coincide with the semiflat Ooguri-Vafa coordinate $\Xm^\sf$ (defined by \eqref{eq:sf-magnetic-coord-abbrv}, in terms of $\theta_m$).
    We use the ``$\shift$'' superscript to emphasize this distinction.
\end{remark}

In \cref{prop:sf-glue-equals-reg,prop:sf-glue-equals-ov} below we will show that

\begin{equation*}
    \Omega^{\glue, \sf} = \Omega^{\reg, \ab},
\end{equation*}
and
\begin{equation*}
    (\ab \circ \NAH^\sf_\zeta)^* \Omega^{\glue, \sf} = -4 \pi^2 \cdot \Omega^{\ov, \shift}_\zeta,
\end{equation*}
which will give the following analogue of \cref{thm:reg-equals-ov}:

\begin{theorem}[Shifted semiflat Ooguri-Vafa and regularization] \label{thm:sf-reg-equals-shifted-ov}
    Under the identification of spaces $\Mov \cong \Xfr$,  
        \begin{equation} \label{eq:sf-reg-equals-shifted-ov}
             \Omega^{\ov, \shift}_\zeta = - \frac{1}{4\pi^2} \Omega^{\reg, \sf}_\zeta,
        \end{equation}
    i.e.\ the shifted semiflat Ooguri-Vafa form coincides with the regularized semiflat Atiyah-Bott form, pulled back to $\Xfr$.
\end{theorem}

\begin{remark}[$\theta_m$ vs $\theta_m^\shift$] 
More explicitly,
\begin{equation} \label{eq:shifted-form} 
    \Omega^{\ov, \shift} = \Omega^{\ov, \sf} + \frac{i}{4 \pi ^2} d \log \Xe \wedge d(\theta_m - \theta_m^\shift),
\end{equation}
where
\begin{equation}
    \theta_m - \theta_m^\shift =  \Im \int_{\gamma} (A_h)_{11} - \Im \int_\gamma(A_{h_\sf})_{11}.
\end{equation}
We will postpone the issue of comparing these integrals until \cref{sec:duality-on-hitchin-section}. 
\end{remark}

\subsubsection{Setup for the semiflat gluing procedure} 
\label{ssub:sf-gluing-setup}

Now we will follow the gluing procedure as in \cref{sub:gluing-on-sigma}.
We will use the same
\begin{itemize}
    \item cut off cylinder $\Sigma_r \subseteq \Sigma$,
    \item gluing diffeomorphism $\sigma(w, s) = (e^{i \vartheta} \overline{w}, -s)$, and
    \item paths $\gamma_e$ and $\gamma_m$ (as shown in \cref{fig:cutoff-paths}),
\end{itemize}
but must specify a new map $\chi^\sf$ for the gluing gauge transformation.

Write $\nabla^{\sf, \ab}_\zeta = d + \alpha^\sf$ with respect to the global frame $g^\L$. 
Also let $A_{h_\L}$ denote the connection form of the Chern connection $D_{h_\L}$ with respect to this frame, so that we can (globally) write
\begin{equation} \label{eq:sf-abelian-connection-form}
    \alpha^\sf = \zeta^{-1} \lambda + A_{h_\L} + \zeta \overline{\lambda}.
\end{equation}

With notation as in \cref{sub:gluing-on-sigma} (in particular suppressing the pullback $\pi^*$), near the punctures we have
\begin{equation}
    A_{h_\L} = \epsilon \cdot \frac{m^{(3)}}{2} \left(  \frac{dw}{w} - \frac{d \overline{w}}{\overline{w}} \right) 
\end{equation}
and
\begin{equation}
    \alpha^\sf =  \epsilon \cdot \left[ \zeta^{-1} \lambda_0 +  \frac{m^{(3)}}{2} \left(  \frac{dw}{w} - \frac{d \overline{w}}{\overline{w}} \right) + \zeta \overline{\lambda_0} \right].
\end{equation}

The above formula almost coincides with the earlier expression \eqref{eq:ab-form-sign} for $\alpha$, except now the series expansion \eqref{eq:phi-square-root-expansion} of $\lambda_0$ contains additional regular $\O(w)$ terms, and so we can write
\begin{equation}
    \alpha^\sf = \alpha + \O(|w|).
\end{equation}
It follows that there is a unique choice of gluing map $\chi^\sf$ satisfying the symmetric gluing condition \eqref{eq:sym-gluing-condition} such that
\begin{equation} \label{eq:sf-vs-regular-gluing}
    \boxed{\chi^\sf = \chi + \O(|w|)}
\end{equation}
near $\partial \Sigma_r$, where $\chi$ denotes the original choice of gluing map \eqref{eq:gluing-map}.

In fact, we can once again specify $\chi^\sf$ completely explicitly:

\begin{itemize}
    \item First choose the antiderivative
\begin{equation} \label{eq:lambda-antiderivative}
    \Lambda_0(z) = \frac{z}{2} \sqrt{z^2 + 2m} + m \log (z + \sqrt{z^2 + 2m}) - \frac{m}{2} - m \log 2
\end{equation}
of $\lambda_0$, as in \cite[Equation (3.104)]{Tulli:2019}, so that $\Lambda_0(w) = \frac{1}{2} w^{-2} - m \log w +  \O(w)$.

    \item Define an antiderivative of $\alpha^\sf$ in a neighbourhood of $\partial \Sigma_r$ by
    \begin{equation} \label{eq:sf-lifted-antiderivative}
        A^\sf \coloneqq \epsilon \cdot \pi^* \left(\zeta^{-1} \Lambda_0 + C_0 + \zeta \overline{\Lambda_0}\right),
    \end{equation}
    where
    \begin{equation} \label{eq:sf-chern-antiderivative}
        C_0(w) = \frac{m^{(3)}}{2} (\log w - \log \overline{w})
    \end{equation}
    (so $C_0$ is an antiderivative for the Chern connection term, i.e.\ $d (\epsilon \cdot \pi^* C_0) = A_{h_\L}$ near $\partial \Sigma_r$).
\end{itemize}

Analogously to \eqref{eq:sym-gluing-map-antiderivative}, we will choose the gluing map
\begin{equation}
    \boxed{\chi^\sf = \frac{A^\sf - \sigma^* A^\sf}{2} - \frac{\epsilon}{2} \cdot \frac{\vartheta}{2\pi} x_e}
\end{equation}
(which is indeed consistent with \eqref{eq:sf-vs-regular-gluing}).

\begin{construction}[Semiflat glued form $\Omega^{\glue, \sf}$]
    With this choice of $\chi^\sf$, the gluing construction produces a form $\Omega^{\glue, \sf}$ on $\sfabAfr$ (and its moduli space).
    It can be calculated by
    \begin{align}
    \begin{split}
       \Omega^{\glue, \sf}(\dot{\nabla}^{\sf, \ab}_1, \dot{\nabla}^{\sf, \ab}_2)
        &=\int_{\Sigma_r} \dot{\alpha}^\sf_1 \wedge \dot{\alpha}^\sf_2 
        + \int_{\partial \Sigma_r} (\dot{\chi}^\sf_2 \dot{\alpha}^\sf_1 -  \dot{\chi}^\sf_1 \dot{\alpha}^\sf_2 + \dot{\chi}^\sf_1 d \dot{\chi}^\sf_2) \label{eq:sf-glued-form-1} 
    \end{split} \\
    \begin{split}
        &= \int_{\gamma_e} \dot{\alpha}^\sf_1 \left( \dot{\chi}^\sf_2(p_0) - \dot{\chi}^\sf_2(\sigma(p_0)) + \int_{\gamma_m} \dot{\alpha}^\sf_2 \right) \\
        &\qquad- \int_{\gamma_e} \dot{\alpha}^\sf_2 \left( \dot{\chi}^\sf_1(p_0) - \dot{\chi}^\sf_1(\sigma(p_0)) +  \int_{\gamma_m} \dot{\alpha}^\sf_1 \right). \label{eq:sf-glued-form-2}
    \end{split}
    \end{align} 
\end{construction}

\subsubsection{Semiflat gluing and regularization} 
\label{ssub:sf-gluing-and-regularization}

Just as in \cref{ssub:regularization-for-parallel-transport}, the regularized integrals
\begin{align}
    \int^{\reg, \chi^\sf}_{\gamma_m} \alpha^\sf \coloneqq{}& \chi^\sf(p_0) - \chi^\sf(\sigma(p_0)) + \int_{\gamma_m} \alpha^\sf  \\
    ={}& \! \left( A^\sf(p_0) - A^\sf(\sigma(p_0)) + \int_{\gamma_m} \alpha^\sf  \right) - \frac{\vartheta}{2\pi} x_e \label{eq:sf-reg-integral-antiderivatives}
\end{align}
are $r$-independent, and hence so is the form $\Omega^{\glue, \sf}$.
By the same argument as in \cref{ssub:regularization-atiyah-bott}, we obtain the analogue of \cref{prop:glue-equals-reg}.

\begin{prop}[Semiflat gluing and regularization] \label{prop:sf-glue-equals-reg}
    As forms on $\sfabAfr_\zeta$,
    \begin{equation} \label{eq:sf-glue-equals-reg}
        \Omega^{\glue, \sf} = \Omega^{\reg, \ab},
    \end{equation}
    i.e.\ the semiflat glued form $\Omega^{\glue, \sf}$ coincides with the regularized abelian Atiyah-Bott form.
\end{prop}

\subsubsection{Interpreting the semiflat twistor coordinates} 
\label{ssub:interpreting-sf-coordinates}

Next we will interpret the integrals 
\begin{equation*}
    \int_{\gamma_e} \alpha^\sf \quad \text{and} \quad \int^{\reg, \chi^\sf}_{\gamma_m} \alpha^\sf
\end{equation*}
in terms of the semiflat twistor coordinates, as in \cref{sub:interpreting-electric-coordinate,sub:interpreting-magnetic-coordinate}.

As another sanity check, note that the integral over $\gamma_e$ has the same interpretation in terms of $\Xe$ as before:
\begin{align*}
    \int_{\gamma_e} \alpha^\sf &= \int_{\gamma_e} \left( \zeta^{-1} \lambda_0 +  \frac{m^{(3)}}{2} \left(  \frac{dw}{w} - \frac{d \overline{w}}{\overline{w}} \right) + \zeta \overline{\lambda_0} \right) \\
    &= 2 \pi i (\zeta^{-1} m - \zeta \overline{m} - m^{(3)}) \\
    &= - \log \Xe(E, \theta, h, g) \pmod{2 \pi i}.
\end{align*}
(This is consistent with the fact that $\Xe^\sf = \Xe$.)

We now turn to the more complicated regularized integral over $\gamma_m$.
We will obtain an analogue of \eqref{eq:gamma-m-integral} for the corresponding integral in $\Omega^\glue$, but involving the shifted version of the magnetic angle.

\begin{prop}[Semiflat $\gamma_m$ integral] \label{prop:sf-magnetic-integral}
    For a semiflat connection $\nabla^\sf = \NAH^\sf_\zeta(E, \theta, h, g) \in \sfAfr_\zeta$, pushed down from $\nabla^{\sf, \ab} = d + \alpha^\sf$,
    \label{prop:sf-gamma-m-integral}
    \begin{equation} \label{eq:sf-gamma-m-integral}
        \int^{\reg, \chi^\sf}_{\gamma_m} \alpha^\sf = - \log \Xms(E, \theta, h, g) - \frac{\vartheta}{2\pi} \log \Xe(E, \theta, h, g) \pmod{2\pi i}.
    \end{equation}
\end{prop}
The calculation of this regularized integral involves several steps, but it provides some geometric insight into the terms appearing in $\theta_m^\shift$.
For notational concreteness we will assume that $\Re(\zeta^{-1} m) > 0$, but the argument for the other case is the same.\footnote{Unlike the full magnetic coordinate $\Xm(\zeta)$, the semiflat coordinate does not involve any jumps in $\zeta$.}

\begin{proof}
Using the formulas \eqref{eq:sf-reg-integral-antiderivatives} for $\int^{\reg, \chi^\sf}_{\gamma_m} \alpha^\sf$ and \eqref{eq:shifted-magnetic-coord} for $\Xms$, we need to match up
\begin{equation}  \label{eq:sf-gamma-m-integral-lhs}
    I_m^\sf \coloneqq A^\sf(p_3^t) - A^\sf(p_2^b) + \int_{\gamma_m} \alpha^\sf 
\end{equation}
with
\begin{equation} \label{eq:sf-gamma-m-integral-rhs}
    -\log \Xms = 
    -\zeta^{-1} Z_B  \underbrace{- i m^{(3)} \arg(-m) - \pi i - \int_{\Gamma_\WKB} A_{h_\L}}_{-i \theta_m^\shift}   - \zeta  \overline{Z_B}  \pmod{2\pi i}.
\end{equation}
In the expansion \eqref{eq:shifted-magnetic-angle} of $\theta_m^\shift$ we have rewritten
\begin{equation}
    \int_\gamma (A_{h_\sf})_{11} = \int_{\Gamma_\WKB} A_{h_\L},
\end{equation}
where $\Gamma_\WKB$ is the lift of $\gamma$ to $\Sigma$ from \cref{ssub:sf-ov-form}.
(Note that this integral is purely imaginary.)

To start, we can use the expressions \eqref{eq:sf-abelian-connection-form} for $\alpha^\sf$ and \eqref{eq:sf-lifted-antiderivative} for $A^\sf$ to split up 
\begin{equation} 
     \begin{split}
         I_m^\sf &= \zeta^{-1} \left(  \Lambda_0(p_3^t) + \Lambda_0(p_2^b) + \int_{\gamma_m} \lambda  \right) + \left( C_0(p_3^t) + C_0(p_2^b) + \int_{\gamma_m} A_{h_\L} \right) \\
         &\qquad + \zeta \left( \overline{\Lambda_0}(p_3^t) + \overline{\Lambda_0}(p_2^b) + \int_{\gamma_m} \overline{\lambda}  \right).
     \end{split}
\end{equation}
The $\zeta^{-1}$ and $\zeta$ terms can be calculated explicitly.

\begin{lemma}[Regularized integral of $\lambda$ along $\gamma_m$] \label{lem:regularized-lambda-integrals}
    \begin{equation}
        \Lambda_0(p_3^t) + \Lambda_0(p_2^b) + \int_{\gamma_m} \lambda = -Z_B
    \end{equation}
\end{lemma}
\begin{subproof}
    Essentially the same calculation appears in \cite[Section 9.4.3]{Gaiotto:2013a} and \cite[Lemma 3.13]{Tulli:2019}.
    Begin by deforming the path $\gamma_m$ so that it passes through the (preimage of the) branch point $z = \sqrt{-2m}$, and write
    \begin{equation}
        \int_{\gamma_m} \lambda = \int_{p_3^t}^{\sqrt{-2m}} \lambda + \int_{\sqrt{-2m}}^{p_2^b} \lambda.
    \end{equation}
    These two paths respectively lie on the top/bottom sheets of $\Sigma$, where $\lambda|_\top =  \lambda_0$ and $\lambda|_\bot = - \lambda_0$, so we can write
    \begin{align*}
        \int_{\gamma_m} \lambda &= \int_{p_3^t}^{\sqrt{-2m}} \lambda_0 + \int_{\sqrt{-2m}}^{p_2^b} (- \lambda_0) \\
        &= \left(\Lambda_0(\sqrt{-2m}) - \Lambda_0(p_3^t) \right) + \left(\  \Lambda_0(\sqrt{-2m}) - \Lambda_0(p_2^b) \right).
    \end{align*}
    Therefore
    \begin{align*}
        \Lambda_0(p_3^t) + \Lambda_0(p_2^b) +   \int_{\gamma_m} \lambda   &= 2 \Lambda_0(\sqrt{-2m})  = m \log \left(\frac{m}{-2e}\right) = -Z_B,
    \end{align*}
    where the second equality is a direct calculation using the definition \eqref{eq:lambda-antiderivative} of $\Lambda_0$.
\end{subproof}
It follows that
\begin{equation*}
    I_m^\sf =  -\zeta^{-1} Z_B  + \left( C_0(p_3^t) + C_0(p_2^b) + \int_{\gamma_m} A_{h_\L}\right) - \zeta \overline{Z_B}
\end{equation*}
This is starting to look like the desired expression \eqref{eq:sf-gamma-m-integral-rhs}. 
Next we need to incorporate the lifted WKB curve $\Gamma_\WKB$.

\begin{lemma}[$\gamma_m$ vs $\Gamma_\WKB$] \label{lem:br-vs-gamma}
    \begin{equation}
        C_0(p_3^t) + C_0(p_2^b) + \int_{\gamma_m} A_{h_\L}  = -i \theta_m^\shift \pmod {2\pi i} 
    \end{equation}
\end{lemma}
\begin{subproof}
We must show that
\begin{equation} \label{eq:br-vs-gamma-expanded}
    C_0(p_3^t) + C_0(p_2^b) + \int_{\gamma_m} A_{h_\L} = -i  m^{(3)} \arg(-m) - \int_{\Gamma_\WKB} A_{h_\L} + \pi i \pmod {2\pi i}.
\end{equation}
    We will compare the integrals in terms of the paths indicated in the triptych below (\cref{fig:br-vs-gamma-triptych}).

    \begin{figure}[ht]
        \subcaptionbox{\label{subfig:br-and-gamma}}{\includegraphics[scale=1]{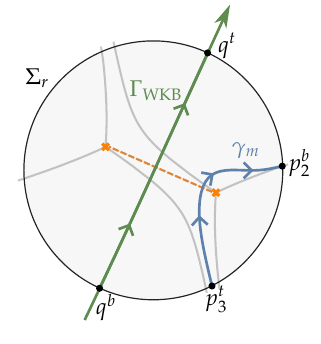}}
        \subcaptionbox{\label{subfig:gamma-decomp}}{\includegraphics[scale=1]{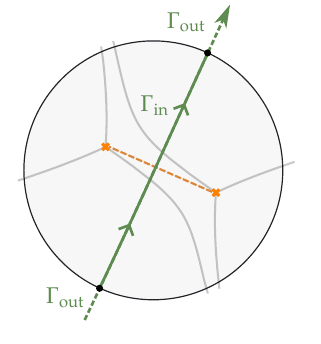}}
        \subcaptionbox{\label{subfig:br-gamma-loop}}{\includegraphics[scale=1]{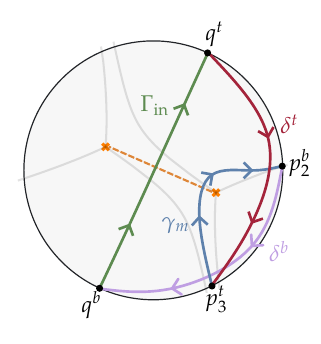}}
        \caption{Comparison of the lifted WKB curve $\Gamma_\WKB$ and $\gamma_m$, in terms of paths $\delta^t$ and $\delta^b$ near the boundary of $\Sigma_r$.}
        \label{fig:br-vs-gamma-triptych}
    \end{figure}

    Let $q^b$ and $q^t$ denote the points where $\Gamma_\WKB$ intersects $\partial \Sigma_r$, lying on the sheets indicated by their superscripts, and decompose
    \begin{equation*}
        \Gamma_\WKB = \Gamma_\ins + \Gamma_\out
    \end{equation*}
    into components lying inside and outside $\Sigma_r$.

    Recall that near the punctures---so in particular, near $\partial \Sigma_r$ and along $\Gamma_\out$---we have
    \begin{align*}
        A_{h_\L} &= \epsilon \cdot \frac{m^{(3)}}{2} \left(  \frac{dw}{w} - \frac{d \overline{w}}{\overline{w}} \right) \\
        &= \epsilon \cdot i m^{(3)} d\theta.
    \end{align*}
    Since $d \theta(\dot{\Gamma}_\WKB) = 0$ (i.e.\ $\Gamma_\WKB$ is a straight line), it follows that
    \begin{equation*}
        \int_{\Gamma_\out} A_{h_\L} = 0,
    \end{equation*}
    and so 
    \begin{equation}
        \int_{\Gamma_\WKB}  A_{h_\L} = \int_{\Gamma_\ins}  A_{h_\L}.
    \end{equation}

    Now consider the two paths $\delta^t$ and $\delta^b$ shown in \cref{subfig:br-gamma-loop}, going along $\partial \Sigma_r$ from $q^t$ to $p_3^t$ and $p_2^b$ to $q^b$ respectively.
    The concatenation
    \begin{equation*}
        \Gamma_\ins + \delta^t + \gamma_m + \delta^b
    \end{equation*}
    is homotopic to a small loop around one of the branch points, so
    \begin{equation}
        \exp\left( \int_{\Gamma_\ins + \delta^t + \gamma_m + \delta^b} A_{h_\L} \right) = -1,
    \end{equation}
    and hence
    \begin{equation}
        \int_{\gamma_m} A_{h_\L} + \int_{\delta^t} A_{h_\L} + \int_{\delta^b} A_{h_\L} = -\int_{\Gamma_\WKB} A_{h_\L} + \pi i  \pmod {2\pi i}.
    \end{equation}

    Since $d (\epsilon \cdot C_0) = A_{h_\L}$ near $\partial \Sigma_r$, we can write
    \begin{equation*}
        \int_{\delta^t} A_{h_\L} = C_0(p_3^t) - C_0(q^t) 
    \end{equation*}
    on the top sheet
    \begin{equation*}
        \int_{\delta^b} A_{h_\L} 
        = C_0(p_2^b) - C_0(q^b)
    \end{equation*}
    and on the bottom.
    Finally, we can use the angle \eqref{eq:wkb-curve-choice} of the WKB curve to compute\footnote{Note that $\theta$ is the angle in $w$-coordinates, whereas \cref{fig:br-vs-gamma-triptych} and the formula \eqref{eq:wkb-curve-choice} are in terms of $z$.}
    \begin{align*}
        C_0(q^t) + C_0(q^b) &= -i m^{(3)} \frac{\arg(m)}{2} -im^{(3)} \left( \frac{\arg(m)}{2} + \pi \right) \\
        &= -im^{(3)} \arg(-m),
    \end{align*}
    from which \eqref{eq:br-vs-gamma-expanded} follows.
\end{subproof}
Therefore
\begin{equation*}
    I_m^\sf =  -\zeta^{-1} Z_B  -i \theta_m^\shift - \zeta \overline{Z_B} = - \log \Xms \pmod {2\pi i},
\end{equation*}
which proves \cref{prop:sf-gamma-m-integral}.
\end{proof}

\subsubsection{Summary of semiflat twistor interpretations and completing the proof} 
\label{ssub:sf-twistor-summary}

Similarly to \cref{sub:twistor-summary}, we see that for variations $\dot{\mathcal{E}}_i$ of $\mathcal{E} = (E, \theta, h, g) \in \Hfr$, with induced variations of the semiflat connections $\nabla^\sf_\zeta$ and $\nabla^{\sf, \ab}_\zeta = d + \alpha^\sf$,
\begin{align*}
\begin{split}
    (\ab \circ \NAH^\sf_\zeta)^* \Omega^{\glue, \sf}(\dot{\mathcal{E}}_1, \dot{\mathcal{E}}_2) 
    &\defeq \int_{\gamma_e} \dot{\alpha}^\sf_1 \left(\dot{\chi}^\sf_2(p_0) - \dot{\chi}^\sf_2(\sigma(p_0)) + \int_{\gamma_m} \dot{\alpha}^\sf_2 \right) \\
    &\qquad - \int_{\gamma_e} \dot{\alpha}^\sf_2 \left( \dot{\chi}^\sf_1(p_0) - \dot{\chi}^\sf_1(\sigma(p_0)) + \int_{\gamma_m} \dot{\alpha}^\sf_1 \right)
\end{split} \\
\begin{split}
    &=  d \log \Xe \wedge d \log \Xms (\dot{\mathcal{E}}_1, \dot{\mathcal{E}}_2)
\end{split} \\
\begin{split}
    &\defeq -4 \pi^2 \cdot \Omega^{\ov, \shift}_\zeta (\dot{\mathcal{E}}_1, \dot{\mathcal{E}}_2),
\end{split}
\end{align*}
which gives the following analogue of \cref{prop:glue-equals-ov}:

\begin{prop}[Semiflat gluing and shifted Ooguri-Vafa] \label{prop:sf-glue-equals-ov}
    \begin{equation} \label{eq:sf-glue-equals-ov}
         (\ab \circ \NAH^\sf_\zeta)^* \Omega^{\glue, \sf} = -4 \pi^2 \cdot \Omega^{\ov, \shift}_\zeta.
    \end{equation}
\end{prop}

Combining our results, we get the equality of forms
\begin{align*}
    \Omega^{\ov, \shift} &= -\frac{1}{4 \pi^2} (\ab \circ \NAH^\sf_\zeta)^* \Omega^{\glue, \sf} \\
    &= -\frac{1}{4 \pi^2}(\NAH^\sf_\zeta)^* \ab^* (\Omega^{\reg, \ab}|_{\sfabAfr_\zeta}) \\
    &= -\frac{1}{4 \pi^2} (\NAH^\sf_\zeta)^*  \Omega^\reg|_{\sfAfr} \\
    &= -\frac{1}{4 \pi^2} \Omega^{\reg, \sf},
\end{align*}
completing the proof of \cref{thm:sf-reg-equals-shifted-ov}.

\section{Duality and the magnetic angles on the Hitchin section} 
\label{sec:duality-on-hitchin-section}

It remains to compare the usual semiflat Ooguri-Vafa form $\Omega^{\ov, \sf}$ to the shifted version $\Omega^{\ov, \shift}$ appearing in \cref{thm:sf-reg-equals-shifted-ov}.
One might hope that the \emph{a priori} different magnetic angles $\theta_m$ and $\theta_m^\shift$ actually coincide, which would mean that $\Xm^\sf = \Xms$ and hence $\Omega^{\ov, \sf} = \Omega^{\ov, \shift}$.
We will prove that this holds on a ``framed Hitchin section'' $\Bfr \subset \Xfr$ (\cref{def:framed-hitchin-section}).

\begin{theorem}[Vanishing angles on the Hitchin section] \label{thm:hitchin-section-shift}
    Restricted to the Hitchin section $\Bfr \subset \Xfr$,
    \begin{equation}
         \theta_m|_{\Bfr} \equiv 0 \equiv  \theta_m^\shift|_{\Bfr}.
     \end{equation}
    Consequently
    \begin{equation}
        \Omega^{\ov, \sf}_\zeta|_{\Bfr}  = -\frac{1}{4 \pi^2} \Omega^{\reg, \sf}_\zeta|_{\Bfr},
    \end{equation}
    i.e.\ the (usual) semiflat Ooguri-Vafa form coincides with the regularized semiflat Atiyah-Bott form.
\end{theorem}

Comparing the angles $\theta_m$ and $\theta_m^\shift$ amounts to comparing the integrals 
\begin{equation*}
    \int_\gamma (A_h)_{11} \quad \text{and} \quad \int_\gamma (A_{h_\sf})_{11}.
\end{equation*}
These integrals are generally hard to describe explicitly---they respectively involve solutions to Hitchin's equations \eqref{eq:hitchin} and \eqref{eq:decoupled-hitchin}---but by restricting to the Hitchin section we will be able take advantage of an additional symmetry of the bundles.

Similarly to the standard Hitchin section, we will consider a family of bundles of the form $E = K_C^{-1/2} \oplus K_C^{1/2}$ with Higgs field
\begin{equation*}
    \theta =\begin{pmatrix}
            0 & 1 \\ 
            (z^2 + 2m)  dz^2 & 0
        \end{pmatrix}, \quad m \in \C^*,
\end{equation*}
but we will need to take some extra care to specify the parabolic structure and framing.

As in the unframed case \cite{Hitchin:1992a}, these bundles should exhibit a kind of self-duality related to their structure as real Higgs bundles (cf.\ \cite{Neitzke:2024} for a similar wild setting, but in rank 3).
We will discuss this duality more generally before returning to the Hitchin section, where we will use it to show that the magnetic angles vanish.

\subsection{Framed duality} 
\label{sub:framed-duality}

There is a natural duality $\D^\ov$ on the Ooguri-Vafa space $\Mov$ defined by negating the angles on each torus fibre,
 i.e.\ acting on the coordinates by
    \begin{align} \label{eq:ov-duality}
        \D^\ov: \begin{cases}
            z &\mapsto z \\
            \theta_e &\mapsto -\theta_e \\ 
            \theta_m &\mapsto -\theta_m. \\
        \end{cases}
    \end{align}
Our goal in this subsection is to describe the corresponding duality $\D^\fr$ on the space $\Xfr$.
Note that in terms of the correspondence $\Mov \cong \Xfr$ (see \cref{tab:dictionary}), $\D^\ov$ sends $m \mapsto m$ and $m^{(3)} \mapsto -m^{(3)}$.

\subsubsection{Duality for unrestricted framed harmonic bundles} 
\label{ssub:duality-for-unrestricted-framed-harmonic-bundles}

A natural first guess for $\D^\fr$ is to send each component of $(E, \theta, h, g)$ to its dual object.
This is almost correct, but we will also need to shift the dual parabolic weights when $m^{(3)} = \frac{1}{2}$ to ensure that they remain in the half-open interval $(-\frac{1}{2}, \frac{1}{2}]$.
We will proceed towards the full definition in steps, starting with the duality
\begin{equation} \label{eq:unframed-duality}
    (E, \theta, h)^* \coloneqq (E^*, \theta^t, h^*)
\end{equation}
on the space $\H$ of (unframed) harmonic bundles.
\begin{lemma}[Dual harmonic bundle]
    $(E, \theta, h)^*$ is in fact a harmonic bundle, i.e.\ the dual metric $h^*$ satisfies the ``dual Hitchin equation''
    \begin{equation} \label{eq:dual-hitchin}
        F_{D_{h^*}} + [\theta^t, (\theta^t)^{\dagger_{h^*}}] = 0.
    \end{equation}
\end{lemma}
\begin{proof}
    By general properties of duality, $(\theta^t)^{\dagger_{h^*}} = (\theta^{\dagger_h})^t$, and so 
    \begin{align*}
        [\theta^t, (\theta^t)^{\dagger_{h^*}}] &= [\theta^t, (\theta^{\dagger_h})^t] = -[\theta, \theta^{\dagger_h}]^t \\
        &=(F_{D_h})^t \tag*{since $h$ satisfies the Hitchin equation \eqref{eq:hitchin}} \\
        &= - F_{D_{h^*}}. \tag*{\qedhere}
    \end{align*}
\end{proof}
Also note that the parameter $m$ is preserved, since $\det \theta^t = \det \theta = -(z^2+2m) dz^2$.

\begin{remark}
    One could also reasonably define $(E, \theta, h)^* = (E^*, -\theta^t, h^*)$ as in e.g.\ \cite{Simpson:1992}, but the version without the minus sign will work better with our framing conventions.
\end{remark}

In order to extend this duality to the space $\Hfr$ of framed bundles we will also need to take into account the parameter $m^{(3)} \in (-\frac{1}{2}, \frac{1}{2}]$.
To simplify matters we will temporarily introduce a larger unrestricted space of bundles.

\begin{definition}[Unrestricted framed harmonic bundles]
    Let $\uHfr$ be the space of framed harmonic bundles satisfying the same conditions as $\Hfr$ in \cref{def:framed-harmonic}, but without any restriction on $m^{(3)}$, i.e.\ replacing \eqref{eq:holo-str-framed-form} with
    \begin{equation} \label{eq:unrestricted-holo-str-framed-form}
        \bar{\partial}_E = \bar{\partial} - \frac{m^{(3)}}{2} H \frac{d \overline{w}}{\overline{w}} + \text{regular terms} \quad \text{for some $ m^{(3)} \in \R$.} 
    \end{equation}
\end{definition}
This extended space is the natural setting for the ``naive duality''
\begin{equation} \label{eq:dual-framed-bundle}
    (E, \theta, h, g)^* \coloneqq (E^*, \theta^t, h^*, g^*).
\end{equation} 

\begin{lemma}[Duality and $m,m^{(3)}$] \label{lem:dual-framed-params}
    If $(E, \theta, h, g) \in \uHfr$ has associated parameters $m$ and $m^{(3)}$, then its dual $(E, \theta, h, g)^*$ also belongs to $\uHfr$ and has associated parameters $m$ and $-m^{(3)}$.
\end{lemma}
\begin{proof}
    With respect to the dual frame $g^*$ near $w=0$,
    \begin{equation}
        \theta^t = -H \frac{dw}{w^3} - m H \frac{dw}{w} + \text{regular terms}
    \end{equation}
    and
    \begin{equation}
        \bar{\partial}_{E^*} = \bar{\partial} + \frac{m^{(3)}}{2} H \frac{d \overline{w}}{\overline{w}} + \text{regular terms}
    \end{equation}
    (cf.\ the expressions \eqref{eq:higgs-framed-form} for $\theta$ and \eqref{eq:holo-str-framed-form} for $\bar{\partial}_E$ with respect to $g$).
    Therefore $g^*$ is a compatible frame for $(E^*, \theta^t, h^*)$ exhibiting the desired properties.
\end{proof}

This definition of duality is natural with respect to many of the other constructions on $\uHfr$.
In particular:
\begin{itemize}
    \item The parabolic structure of $(E, \theta, h, g)^* \in \uHfr$ is dual to that of $(E, \theta, h, g)$ (see \cref{sub:filtered-duality}).

    \item The semiflat constructions from \cref{sec:sf-analysis} are compatible with duality, in the sense that:
    \begin{itemize}
        \item The spectral Higgs line bundle of $(E, \theta)^*$ is the dual of that of $(E, \theta)$, and their corresponding induced frames (as in \cref{sub:sf-abelianization-and-framing}) are also dual.
        \item The semiflat harmonic metric for $(E, \theta)^*$ is the dual of that for $(E, \theta)$.
    \end{itemize}
\end{itemize}

We can define magnetic angles on the larger space $\uHfr$ using the exact same formulas as for $\Hfr$:
\begin{equation} \label{eq:magnetic-angle-hat}
    \hat{\theta}_m \coloneqq m^{(3)} \arg(-m) + \pi + \Im \int_\gamma (A_h)_{11}  \pmod {2\pi}
\end{equation}
and
\begin{equation} \label{eq:shifted-magnetic-angle-hat}
    \hat{\theta}_m^\shift \coloneqq m^{(3)} \arg(-m)  + \pi + \Im \int_\gamma (A_{h_\sf})_{11}  \pmod {2\pi}.
\end{equation}

\begin{prop}[Magnetic angles under duality on $\uHfr$] \label{prop:unrestricted-magnetic-angles-under-duality}
    \begin{equation}
        \hat{\theta}_m((E, \theta, h, g)^*) = -\hat{\theta}_m(E, \theta, h, g)
    \end{equation}
    and
    \begin{equation}
        \hat{\theta}_m^\shift((E, \theta, h, g)^*) = -\hat{\theta}_m^\shift(E, \theta, h, g).
    \end{equation}
\end{prop}
\begin{proof}
    By \cref{lem:dual-framed-params} the first term $m^{(3)} \arg(-m)$ is negated, so we must show that the same is true of
    \begin{equation*}
        \Im \int_\gamma (A_h)_{11} \quad \text{and} \quad \Im \int_{\gamma} (A_{h_\sf})_{11}.
    \end{equation*}

    For the first integral, recall that $A_h$ denotes the connection form of the Chern connection $D_h$ with respect to a $\theta$-eigenframe $(\eta_1, \eta_2)$ satisfying the normalization condition \eqref{eq:wkb-eigenframe-normalization}.
    In order to compute the corresponding integral for $(E, \theta, g)^*$, we can use the dual frame $(\eta_1^*, \eta_2^*)$, with respect to which $A_h$ is replaced by $-A_h^t$ and hence $(A_h)_{11}$ by $-(A_h)_{11}$.
    Therefore the first integral is negated.
    
    The same argument works for the second integral, by the compatibility of the semiflat construction with duality.
\end{proof}

\subsubsection{Hecke modifications and magnetic angles} 
\label{ssub:hecke-modifications-and-magnetic-angles}

Shifting the weight $m^{(3)}$ of an element of $\uHfr$ can be described in terms of certain Hecke modifications.

In \cref{sub:hecke-modifications} we recall a standard definition of Hecke modifications, which involves making a choice of a holomorphic trivialization.
If the chosen holomorphic frame is compatible with the parabolic structure, then the Hecke modification naturally shifts the associated parabolic weights.
However, the bundles in $\uHfr$ come equipped with \emph{unitary} rather than holomorphic frames, so it will be convenient to give an alternative definition adapted to this setting.

\begin{definition}[Unitary Hecke modification] \label{def:unitary-hecke-mod}
    Let $(E, \theta, h, g) \in \uHfr$, and write $g = (e_1, e_2)$ for the frame in a neighbourhood of $\infty$.
    A \emph{unitary Hecke modification of $(E, \theta, h, g)$ of type $n \in \Z$} is the new bundle $(\tilde{E}, \theta, \tilde{h}, \tilde{g}) \in \uHfr$ such that:
    \begin{enumerate}
        \item The underlying harmonic bundle satisfies
        \begin{equation}
            (\tilde{E}, \theta, \tilde{h})|_{\CP^1 \setminus \{\infty\}} = (E, \theta, h)|_{\CP^1 \setminus \{\infty\}}.
        \end{equation}
        \item The modified frame is
        \begin{equation} \label{eq:hecke-modified-frame} 
            \tilde{g}|_{\CP^1 \setminus \{\infty\}} = \left( \frac{w^{n}}{|w|^{n}} \cdot e_1, \frac{w^{-n}}{|w|^{-n}} \cdot e_2 \right).
        \end{equation}
        (This defines the required unitary extension of $(\tilde{E}, \theta, \tilde{h})|_{\CP^1 \setminus \{\infty\}}$ over $\infty$.)
    \end{enumerate}
\end{definition}

\begin{remark}
    If $g$ was obtained by orthonormalizing a holomorphic frame as in \cref{constr:frame-from-eigenframe}, then \cref{def:unitary-hecke-mod} is compatible with our definition of a parabolic Hecke modification in \cref{sub:hecke-modifications}. 
\end{remark}

It is clear from the definition that unitary Hecke modifications define a $\Z$-action on $\uHfr$.
\begin{lemma}[Hecke modifications and $m^{(3)}$]
    A unitary Hecke modification of type $n$ shifts the weight $m^{(3)}$ of an element of $\uHfr$ to $m^{(3)} + n$.
\end{lemma}
\begin{proof}
    Calculate \eqref{eq:unrestricted-holo-str-framed-form} under the change of frame $g \to \tilde{g}$.
\end{proof}

By making an appropriate Hecke modification, the weight $m^{(3)}$ associated to $(E, \theta, h, g) \in \uHfr$ can always therefore be (uniquely) shifted to lie in $(-\frac{1}{2}, \frac{1}{2}]$.

\begin{cor}
    The quotient of the space $\uHfr$ by the action of unitary Hecke modifications is $\Hfr$.
\end{cor}

This allows us to complete the definition of duality on $\Hfr$.

\begin{definition}[Duality on $\Hfr$]
    Given a framed bundle $(E, \theta, h, g) \in \Hfr$ with weight $m^{(3)}$, define its dual $\D^\fr (E, \theta, h, g) \in \Hfr$ by 
    \begin{enumerate}
        \item taking the dual $(E, \theta, h, g)^*$ in $\uHfr$, and 

        \item performing a Hecke modification of type $1$ if $m^{(3)} = \frac{1}{2}$.
    \end{enumerate}
    The resulting bundle has weight
    \begin{equation*}
        \begin{cases}
            -m^{(3)} & \text{if } m^{(3)} \in (-\frac{1}{2}, \frac{1}{2}), \\
            -m^{(3)} + 1 = \frac{1}{2} & \text{if } m^{(3)} = \frac{1}{2},
        \end{cases}
    \end{equation*}
    and hence still lies in $\Hfr$.
\end{definition}

\begin{prop}
    The magnetic angle $\hat{\theta}_m$ and shifted magnetic angle $\hat{\theta}_m^\shift$ on $\uHfr$ are invariant under unitary Hecke modifications.
\end{prop}

\begin{proof}
    It suffices to consider a modification $(E, \theta, g) \to (\tilde{E}, \theta, \tilde{g})$ of type 1.
    We must show that the integrals in \eqref{eq:magnetic-angle-hat} and \eqref{eq:shifted-magnetic-angle-hat} shift by $-\arg(-m)$ to compensate for the increment $m^{(3)} \to m^{(3)} + 1$.

    Starting with \eqref{eq:magnetic-angle-hat}, let 
    \begin{equation*}
        I_m(g) \coloneqq \Im \int_\gamma (A_h)_{11} \quad \text{and} \quad I_m(\tilde{g}) \coloneqq \Im \int_\gamma (\tilde{A}_h)_{11} 
    \end{equation*}
    denote the integrals for $\hat{\theta}_m(E, \theta, g)$ and $\hat{\theta}_m(\tilde{E}, \theta, \tilde{g})$.
    The first integral $I_m(g)$ is computed with respect to an eigenframe $(\eta_1, \eta_2)$ along $\gamma$ that satisfies the normalization condition \eqref{eq:wkb-eigenframe-normalization} involving $g = (e_1, e_2)$:
    \begin{equation*}
        (\eta_1, \eta_2)|_{\gamma(t)} \to \begin{cases}
            (e_1, e_2) & \text{as } t \to \infty,  \\
            (e_2, -e_1) & \text{as } t \to -\infty.
        \end{cases}
    \end{equation*}

    The Hecke modified bundle is defined by the frame
    \begin{equation*}
        \tilde{g} = \left( \frac{w}{|w|} e_1, \frac{w^{-1}}{|w|^{-1}} e_2  \right) = \left( \frac{z^{-1}}{|z|^{-1}} e_1, \frac{z}{|z|} e_2 \right)
    \end{equation*}
    near $\infty$.
    By the choice \eqref{eq:wkb-curve-choice} of $\gamma$ as a straight line $\gamma(t) = \rho(t) e^{i \arg(m)/2}$,
    \begin{equation}
        \tilde{g}|_{\gamma(t)} = \begin{cases}
            (e^{-i \arg(m)/2} \cdot  e_1, \ e^{i \arg(m)/2} \cdot e_2) & \text{for } t \gg 0,   \\
            (-e^{-i \arg(m)/2} \cdot  e_1, \ -e^{i \arg(m)/2} \cdot e_2) & \text{for } t \ll 0.
        \end{cases}
    \end{equation}
    The corresponding normalization condition for the eigenframe $(\tilde{\eta}_1, \tilde{\eta}_2)$ used to calculate $I_m(\tilde{g})$ is
    \begin{equation}
        (\tilde{\eta}_1, \tilde{\eta}_2)|_{\gamma(t)} \to \begin{cases}
            (e^{-i \arg(m)/2} \cdot  e_1, \ e^{i \arg(m)/2} \cdot e_2) & \text{as } t \to \infty,  \\
            (-e^{i \arg(m)/2} \cdot e_2, \ e^{-i \arg(m)/2} \cdot e_1) & \text{as } t \to -\infty.
        \end{cases}
    \end{equation}
    We can therefore take $\tilde{\eta}_i = c_i \eta_i$, where the coefficients $c_i$ are chosen to approach the phases above as $t \to \pm \infty$, e.g.
    \begin{equation*}
        c_1|_{\gamma(t)} \to \begin{cases}
            e^{-i \arg(m)/2} & \text{as } t \to  \infty,  \\
            -e^{i \arg(m)/2} & \text{as } t \to -\infty.
        \end{cases}
    \end{equation*}

    It follows that
    \begin{align*}
        I_m(\tilde{g}) &= \Im \int_\gamma (\tilde{A}_h)_{11} = \Im \int_\gamma ((A_h)_{11} +  d \log c_1) \\
        &= I_m(g) + \underbrace{\arg(c_1|_{\gamma(\infty)})}_{-\arg(m)/2} - \underbrace{\arg(c_1|_{\gamma(-\infty)})}_{\arg(m)/2+\pi} \\
        &= I_m(g) -\arg(-m),
    \end{align*}
    as required.
    The argument for the shifted magnetic angle \eqref{eq:shifted-magnetic-angle-hat} is identical.
\end{proof}

\begin{cor}[Magnetic angles under duality on $\Hfr$] \label{cor:magnetic-angles-under-duality}
    The duality $\D^\fr$ on $\Hfr$ negates the magnetic angles $\theta_m$ and $\theta_m^\shift$.
\end{cor}

We will also write $\D^\fr$ to denote the induced duality $\D^\fr([E, \theta, h, g]) = [\D^\fr(E, \theta, h, g)]$ on the space $\Xfr$ of isomorphism classes.

\begin{cor}[$\D^\fr \leftrightarrow \D^\ov$]
    Under the identification of $\Xfr$ with $\Mov$, the duality $\D^\fr$ coincides with the natural duality $\D^\ov$ defined by \eqref{eq:ov-duality}.
\end{cor}

\subsection{Framed Hitchin section and self-duality} 
\label{sub:framed-hitchin-section}

Now we will construct a framed version of the Hitchin section and show that it is self-dual under $\D^\fr$.
Certain technical details are relegated to \cref{sub:self-dual-frames}, but we summarize the main points below.

\begin{definition}[Hitchin section for $\Xfr$] \label{def:framed-hitchin-section}
    Consider the set of quadratic differentials 
    \begin{equation}
        \mathcal{Q} \coloneqq \{(z^2 + 2m)  dz^2 : m \in \C^* \}.
    \end{equation}
    Define a map
    \begin{equation}
        \iota^\fr: \mathcal{Q} \to \Xfr
    \end{equation}
    sending $(z^2 + 2m ) dz^2$ to the isomorphism class $[(E, \theta, g_0)]$ of the framed bundle with
    \begin{itemize}
        \item underlying vector bundle
        \begin{equation}
             E|_{\CP^1 \setminus \{\infty\}} = K_C^{-1/2} \oplus K_C^{1/2}|_{\CP^1 \setminus \{\infty\}},
         \end{equation} 
         where $K_C^{1/2}$ is a spin structure\footnote{i.e.\ a choice of line bundle such that $K_C^{1/2} \otimes K_C^{1/2} \cong K_C$, the canonical bundle on $C$. For $C = \CP^1$ there is a unique choice, namely $K_C^{1/2} = \O(-1)$.},

        \item Higgs field
        \begin{equation}
            \theta =\begin{pmatrix}
            0 & 1 \\ 
            (z^2 + 2m)  dz^2 & 0
        \end{pmatrix},
        \end{equation}

        \item parabolic weight $m^{(3)} = \frac{1}{2}$ (with multiplicity $2$) at $\infty$, as in \cref{constr:bundle-from-parameters}, and
         
        \item framing $g_0 = (e_1, e_2)$ at $\infty$ as specified in \cref{def:hitchin-section-frame}, that is, obtained by orthonormalizing the eigenframe $(\eta_1, w\eta_2)$ where
        \begin{equation*}
        \eta_1 = \sqrt{\frac{i z}{2 \sqrt{P}}} \cdot
        \begin{pmatrix}
        1 \\
        \sqrt{P}
        \end{pmatrix}
        \quad \text{and} \quad
        \eta_2 = \sqrt{\frac{-i z}{2 \sqrt{P}}} \cdot
        \begin{pmatrix}
        -1 \\
        \sqrt{P}
        \end{pmatrix}.
    \end{equation*}

    \end{itemize}
    We will refer to the image $\Bfr \coloneqq \iota^\fr(\B) \subseteq \Xfr$ as the \emph{(framed) Hitchin section} for $\Xfr$.
\end{definition}

\begin{prop}[Self-duality on $\Bfr$] \label{prop:self-duality}
    Isomorphism classes of bundles in the Hitchin section $\Bfr \subseteq \Xfr$ are fixed under the duality $\D^\fr$.
\end{prop}
\begin{proof}
    Unpacking the definitions, we must show that
    $(E, \theta, h, g) \cong (E^*, \theta^t, h^*, g^*)$ up to a unitary Hecke modification of type 1.

    We claim that the desired isomorphism is given by
    \begin{equation}
        S = \begin{pmatrix}
            0 & i \\ 
            i & 0
        \end{pmatrix}: E \xrightarrow{\sim} E^*.
    \end{equation}

    It is straightforward to compute that:
    \begin{enumerate}
        \item $S$ sends $\theta$ to $\theta^t$, i.e.\ $S \theta S^{-1} = \theta^t$.

        \item $S$ sends $h$ to $h^*$, since on the Hitchin section the harmonic metric is diagonal with respect to the direct sum decomposition of $E$.
    \end{enumerate}
    It remains to show that $S g^* = g$ up to a Hecke modification. 
    This boils down to a linear algebra computation which we give in \cref{lem:dual-frame-calculation}.
\end{proof}

Finally, we conclude that the magnetic angles vanish on the Hitchin section (\cref{thm:hitchin-section-shift}).

\begin{proof}[Proof of \cref{thm:hitchin-section-shift}]
    Since $\D^\fr [(E, \theta, h, g)] = [(E, \theta, h, g)]$ on $\Bfr$,
    \begin{equation*}
        \theta_m [(E, \theta, h, g)] = \theta_m (\D^\fr[(E, \theta, h, g)]) = -\theta_m [(E, \theta, h, g)],
    \end{equation*}
    and likewise for $\theta_m^\shift$.
    Therefore
    \begin{equation*}
        \theta_m|_{\Bfr} \equiv 0 \equiv \theta_m^{\shift}|_{\Bfr}. \qedhere
    \end{equation*}
\end{proof}

\section{Metric on the Hitchin section} 
\label{sec:metric-on-the-hitchin-section}

The natural remaining question is how to translate the equality of the symplectic forms $\Omega^\reg_\zeta$ and $\Omega^\ov_\zeta$ to the corresponding metrics.

At least formally, we would like to say something like
\begin{align*}
    g_{L^2}^\reg(\cdot, \cdot) &= -\Re  \, (\NAH_{\zeta = 1})^* \Omega^{\reg}(\cdot, I \cdot) \tag{cf.\  \eqref{eq:metric-and-form}} \\
    & \defeq -\Re  \,  \Omega^{\reg}_{\zeta = 1}(\cdot, I \cdot)
\end{align*}
for a suitably regularized version of $g_{L^2}$.

The main difficulty here comes from understanding the complex structure $I$ on the space $\Xfr$ of framed Higgs bundles.
It is not \emph{a priori} clear how the complex structure should act on variations of the parabolic weights or framing.\footnote{Of course, the identification $\Mov \cong \Xfr$ from \cite{Tulli:2019} induces a complex structure $I$ on $\Xfr$, but it is not obvious how/if it is related to the ``natural one'' on $\Xfr$ (whatever that may be).}

This difficulty is also related to the analytic question of describing the tangent space. 
As mentioned before, we do not have a gauge-theoretic construction of these moduli spaces.
In the standard (compact, unframed) setting, the tangent space to a point in a Hitchin moduli space $\M^\Hit$ can be identified with the space of variations $(\dot{A}, \dot{\theta})$ which
satisfy the linearizations of Hitchin's equations
\begin{equation*}
    \begin{cases}
        F_{D_h} + [ \theta, \theta^{\dagger_h}] = 0, \\
        \bar{\partial}_E \phi= 0,
    \end{cases}
\end{equation*}
and lie in the orthocomplement of the linearized gauge orbit.
In our setting, a variation of the parabolic weight corresponds to changing the boundary conditions of the relevant PDEs. 
We leave these analytic considerations for future work.

However, these technical difficulties can be circumvented by restricting to the Hitchin section $\Bfr \subset \Xfr$, along which the weights and framing do not change.

\subsection{Polynomial Hitchin sections} 
\label{sub:background-polynomial-hitchin-section}

An explicit integral for the $L^2$-metric on certain ``polynomial Hitchin sections'' is given in \cite{Dumas:2020} (using the computations from \cite{Dumas:2019} for the compact surface case).
Our situation will be described by a regularized version of this integral.
We will briefly summarize the relevant definitions and results from \cite{Dumas:2020} below, and then apply them to our setting in the next section.

Let $\B_d$ denote the space of quadratic differentials $\phi$ on $\C$ of the form $\phi = P(z) \, dz^2$, where $P$ is a polynomial of degree $d$.
(By abuse of notation we will sometimes just write $P \in \mathcal{B}_d$.)
The differentials $\phi \in \B_d$ parametrize a \emph{polynomial Hitchin section of degree $d$}, consisting of Higgs bundles $(E, \theta_\phi)$ where $E$ is the trivial\footnote{The canonical bundle of $C=\C$ is trivial, so we fix a trivialization $E = K_C^{-1/2} \oplus K_C^{1/2} \cong \O^2$ to simplify the notation.} rank $2$ bundle over $\C$ and
\begin{equation}
     \theta_\phi = \begin{pmatrix}
        0 & 1 \\
        P & 0
    \end{pmatrix} \, dz.
\end{equation}

The harmonic metric $h = h_\phi$ for $(E, \theta_\phi)$ is diagonal and can be written as
\begin{equation}
    h = \begin{pmatrix}
        e^u & 0 \\
        0 & e^{-u}
    \end{pmatrix}
\end{equation}
for a scalar function $u$ that satisfies:
\begin{enumerate}
    \item the self-duality equation
    \begin{equation}
        \Delta u = 4(e^{2u} - e^{-2u} |P|^2).
    \end{equation}
    (This is equivalent to Hitchin's equation \eqref{eq:hitchin}.) 

    \item the growth rate condition
    \begin{equation}
        u \sim \frac{1}{2} \log |P| \quad \text{as } |z| \to \infty.
    \end{equation}
    (This is equivalent to the choice of parabolic weights $m^{(3)} = \frac{1}{2}$ in our setting.)
\end{enumerate}

Tangent vectors to $\mathcal{B}_d$ are just represented by polynomials $\dot{P}$, but one must restrict their degrees as follows in order to obtain a convergent $L^2$-integral.

Given a polynomial $h(z)$ of degree $d$, let $\B_{d, h} \subseteq \B_d$ denote the space of polynomials 
\begin{equation}
    P(z) = h(z) + \ell(z)
\end{equation}
with $\deg \ell < \frac{d}{2} - 1$.
Then tangent vectors to $P \in \B_{d, h}$ are polynomials $\dot{P} = \dot{\ell}$ with 
\begin{equation}  \label{eq:degree-condition}
    \deg \dot{P} < \frac{d}{2} - 1. 
\end{equation}

As computed in \cite[Section 4]{Dumas:2019}, the norm of $\dot{P}$ with respect to Hitchin's metric is given by
\begin{equation} \label{eq:hitchin-section-integral}
    g_{L^2}(\dot{P}, \dot{P}) 
    = \int_\C 4e^{-2u} (|\dot{P}| - \Re(F P \dot{\overline{P}})) \, dx dy,
\end{equation}
where $F$ is the solution on $\C$ to the ``complex variation equation''
\begin{equation}
    (\Delta - 8(e^{2u} + e^{-2u} |P|^2))F + 8e^{-2u} \overline{P} \dot{P} = 0.
\end{equation}
(This integral is the norm $\int_\C \|\dot{A} \|^2 + \|\dot{\Theta} \|^2$ of the Coulomb gauge representative $(\dot{A}, \dot{\Theta})$ corresponding to the variation $\dot{P}$, as described in \cref{sub:l2-metric}.)

For large $|z|$, 
\begin{equation}
    u \sim u_\sf \coloneqq \frac{1}{2} \log |P| \quad \text{and} \quad F \sim F_\sf \coloneqq \frac{1}{2} \frac{\dot{P}}{P},
\end{equation}
so the integrand in \eqref{eq:hitchin-section-integral} grows like $|\dot{P}|^2/|P| \sim |z|^{2 \deg \dot{P} - d}$.
Therefore the integral converges if and only if the degree condition \eqref{eq:degree-condition} is satisfied.

In addition, replacing $u$ and $F$ with their semiflat approximations $u_\sf$ and $F_\sf$ gives the formula for the norm with respect to the semiflat metric:
\begin{equation}  \label{eq:hitchin-section-sf-integral}
    g_{L^2}^\sf(\dot{P}, \dot{P})
    = \int_\C 2 \frac{|\dot{P}|^2}{|P|} \, dx dy.
\end{equation}

\subsection{Regularization} 
\label{sub:metric-regularization}

Our framed Hitchin section $\Bfr \subset \Xfr$ (\cref{def:framed-hitchin-section}) is a polynomial Hitchin section of degree $d = 2$.
We are interested in polynomials of the form
\begin{equation*}
    P(z) = z^2 + 2m,
\end{equation*}
but we wish to allow variations $\dot{P} = 2\dot{m} \in \C$.
This violates the degree condition \eqref{eq:degree-condition}, so we will need to regularize the integral \eqref{eq:hitchin-section-integral}.

For large $|z|$, the integrand is asymptotic to
\begin{equation}
    2 \frac{|\dot{P}|^2}{|P|} \sim \frac{8 |\dot{m}|^2}{|z|^2},
\end{equation}
so the integral diverges logarithmically like $8 |\dot{m}|^2 \cdot 2 \pi \log R$ as $|z| = R \to \infty$.

\begin{definition}[Regularized $L^2$-metric on $\Bfr$]
Define a regularized version of the metric \eqref{eq:hitchin-section-integral} on $\Bfr$ by
\begin{equation} \label{eq:reg-hitchin-section-integral}
    g_{L^2}^\reg(\dot{P}, \dot{P})
    \coloneqq  \lim_{R \to \infty} \left[  \left(\int_{C_R} 4e^{-2u} (|\dot{P}| - \Re(F P \dot{\overline{P}})) \, dx dy\right) - 16 \pi \log R |\dot{m}|^2  \right],
\end{equation}
where $C_R \coloneqq \{z \in \C: |z| \leq R \}$.
\end{definition}

The natural complex structure $I$ on $\Bfr$ acts by $I(\dot{P}) = i \dot{P}$, i.e.\
\begin{equation}
    I(\dot{m}) = i \dot{m}.
\end{equation}
(This is the restriction of the natural Higgs complex structure \eqref{eq:I-on-higgs-variation} which sends $I(\dot{\theta}) = i \dot{\theta}.)$

\begin{prop}[Metric and form compatibility] \label{prop:reg-compatible-form}
    On the framed Hitchin section $\Bfr \subset \Xfr$,
    \begin{equation} \label{eq:reg-compatible-form}
        g_{L^2}^\reg(\cdot, \cdot) = -\Re \Omega_{\zeta=1}^\reg (\cdot, I \cdot),
    \end{equation}
    i.e.\ the regularized $L^2$-metric is compatible with the regularized Atiyah-Bott form (from \cref{def:reg-form}) with respect to the complex structure $I$.
\end{prop}

To spell out the notation, write
\begin{align*}
    \nabla \coloneqq&{} \NAH_{\zeta=1} (E, \theta_\phi, h_\phi) \\
    ={}& \theta_\phi + D_{h_\phi} + \theta_\phi^{\dagger_{h_\phi}}
\end{align*}
for the connection associated to $\phi = (z^2 + 2m) \, dz^2 \in \Bfr$.
Then a variation $\dot{P}$ of the polynomial $z^2 + 2m$ induces a variation $\dot{\nabla}_{\dot{P}}$ of $\nabla$, and \eqref{eq:reg-compatible-form} says that
\begin{equation*}
     g_{L^2}^\reg(\dot{P}, \dot{P}) \defeq  \lim_{R \to \infty} \left[  \left(\int_{C_R} 4e^{-2u} (|\dot{P}| - \Re(F P \dot{\overline{P}})) \, dx dy\right) - 16 \pi \log R |\dot{m}|^2  \right] \\
\end{equation*}
is equal to
\begin{equation*}
    - \Re \Omega^\reg (\dot{\nabla}_{\dot{P}}, \dot{\nabla}_{i\dot{P}}) \defeq \lim_{R \to \infty} \Re \left[ \left(- \int_{C_R} \tr (\dot{\nabla}_{\dot{P}} \wedge \dot{\nabla}_{i\dot{P}})\right) -  2 \pi \log R \cdot \mathcal{R}(\dot{\nabla}_{\dot{P}}, \dot{\nabla}_{i\dot{P}}) \right],
\end{equation*}
where as above $C_R = \{|z| \leq R \}$.

\begin{proof}
    The integral terms match up just as they do in the usual (unregularized) case, so we only need to check that the regularization terms coincide.
    Since $\dot{m}^{(3)} \equiv 0$ on the Hitchin section, the regularization term \eqref{eq:regularization-term} for $\Omega^\reg(\dot{\nabla}_1, \dot{\nabla}_2)$ takes the simpler form
\begin{equation*}
    - 2 \pi \log R \cdot \mathcal{R}(\dot{\nabla}_1, \dot{\nabla}_2) = 16 \pi \log R \Im(\dot{m}_1 \dot{\overline{m}}_2).
\end{equation*}
For $\Omega^\reg (\dot{\nabla}_{\dot{P}}, \dot{\nabla}_{i\dot{P}})$, 
this becomes
\begin{equation*}
    16 \pi \log R \Im(\dot{m} \, {\overline{i \dot{m}}}) = -16\pi \log R |\dot{m}|^2,
\end{equation*}
which is the same as the regularization term for the $L^2$-norm.
\end{proof}

\begin{theorem}[Ooguri-Vafa metric and regularization] \label{thm:reg-equals-ov-metric}
    On the polynomial Hitchin section parametrized by $P(z) = z^2 + 2m$ for $m \in \C$,
    \begin{equation}
        g^\reg_{L^2} = 4 \pi^2 \cdot g^\ov,
    \end{equation}
    i.e.\ the regularized version of Hitchin's metric coincides with the Ooguri-Vafa metric.
\end{theorem}
\begin{proof}

    \begin{align*}
        g_{L^2}^\reg (\cdot, \cdot) &= -\Re \Omega^\reg_{\zeta = 1}(\cdot, I \cdot) 
        \tag*{by \cref{prop:reg-compatible-form}} \\
        &= 4\pi^2 \cdot \Re  \,  \Omega^{\ov}_{\zeta = 1}(\cdot, I^\ov \cdot) 
        \tag*{by \cref{thm:reg-equals-ov}} \\
        &= 4\pi^2 \cdot g^\ov(\cdot, \cdot) \tag*{since $g^\ov$ is hyperkähler.}
    \end{align*}
    The second equality uses the fact that the identification $\Xfr \cong \Mov$ preserves the complex structures on the Hitchin section. 
    Indeed, $I^\ov$ acts by multiplication by $+i$ on variations $\dot{z}$ of the Ooguri-Vafa coordinate (see \cref{sec:ooguri-vafa-construction}), and $I$ acts by $+i$ on the corresponding variations $-2\pi \dot{m}$.
\end{proof}

\subsection{Semiflat metric} 
\label{sub:semiflat-metric}

The same argument works for the semiflat metric on the Hitchin section.
Here we can define
\begin{equation} \label{eq:reg-hitchin-section-sf-integral}
    g_{L^2}^{\reg, \sf}(\dot{P}, \dot{P}) 
    \coloneqq  \lim_{R \to \infty} \left[  \left(\int_{C_R} 2 \frac{|\dot{P}|^2}{|P|} \, dx dy \right) - 16 \pi \log R |\dot{m}|^2  \right]
\end{equation}
and make use of the identification
\begin{equation*}
    \Omega^{\ov, \sf}_\zeta|_{\Bfr}  = -\frac{1}{4 \pi^2} \Omega^{\reg, \sf}_\zeta|_{\Bfr}
\end{equation*}
from \cref{thm:hitchin-section-shift}.

\begin{cor}[Semiflat Ooguri-Vafa metric and regularization] \label{thm:sf-reg-equals-sf-ov-metric}
    On the polynomial Hitchin section parametrized by $P(z) = z^2 + 2m$ for $m \in \C$,
    \begin{equation}
        g^{\reg, \sf}_{L^2} = 4 \pi^2 \cdot g^{\ov, \sf},
    \end{equation}
    i.e.\ the regularized semiflat $L^2$-metric coincides with the semiflat Ooguri-Vafa metric.
\end{cor}

In fact, everything here is so explicit that the calculation can be done by hand.
In elliptic coordinates $(\mu, \nu)$ with  $x = \sqrt{-2m} \cosh \mu \cos \nu$ and $y = \sqrt{-2m} \sinh \mu \sin \nu$, there is a neat expression for the integral
\begin{equation}
    \int_{C_{\mu_0}} 2 \frac{|\dot{P}|^2}{|P|} \, dx dy = 16 \pi |\dot{m}|^2 \mu_0
\end{equation}
over $C_{\mu_0} \coloneqq \{ (\mu, \nu) : \mu \leq \mu_0 \}$.

For large $\mu_0$,
\begin{equation*}
    R = |x+iy| \approx \left|\sqrt{-2m} \cdot  \frac{e^{\mu_0}}{2} \right| = \sqrt{\frac{|m|}{2}} \, e^{\mu_0},
\end{equation*}
and so 
\begin{equation*}
    \log R \approx  \frac{1}{2} \log \frac{|m|}{2}  + \mu_0.
\end{equation*}
We can therefore rewrite the regularized semiflat metric \eqref{eq:reg-hitchin-section-sf-integral} as
\begin{align*}
    g_{L^2}^{\reg, \sf}(\dot{P}, \dot{P})
    &=  \lim_{\mu_0 \to \infty} \left[  \left(\int_{C_{\mu_0}} 2 \frac{|\dot{P}|^2}{|P|} \, dx dy \right) - 16 \pi \left(\frac{1}{2} \log \frac{|m|}{2}  + \mu_0 \right) |\dot{m}|^2  \right] \\
    &= -8\pi  \log \left( \frac{|m|}{2} \right) |\dot{m}|^2 \\
    &= 4\pi^2 \cdot g^{\ov, \sf} (\dot{P}, \dot{P}),
\end{align*}
where the last equality uses the explicit formula for the semiflat Ooguri-Vafa metric
\begin{equation*}
    g^{\ov, \sf} = 4 V^\sf \|dm \|^2
\end{equation*}
on the Hitchin section from \cref{sub:explicit-metric-formulas}.
(Recall that $V^\sf$ depends on the Ooguri-Vafa cutoff $\Lambda$:
\begin{equation*}
    V^\sf(m) = -\frac{1}{2\pi} \log \left(\frac{2|m|}{|\Lambda|} \right).
\end{equation*}
The above calculation perhaps gives some geometric insight into the mysterious-looking choice $\Lambda = 4 i$.)

\begin{remark}
    Although the (non-regularized) integrals defining $g_{L^2}$ and $g_{L^2}^\sf$ are both divergent, their difference
    \begin{equation*}
        g_{L^2} - g_{L^2}^\sf = g_{L^2}^\reg - g_{L^2}^{\reg, \sf}
    \end{equation*}
    is convergent (the regularization terms cancel).

\begin{cor}[Instanton part]
On the polynomial Hitchin section parametrized by $P(z) = z^2 + 2m$ for $m \in \C$,
    \begin{equation}
        g_{L^2} - g_{L^2}^\sf  = 4\pi^2 \cdot \underbrace{(g^\ov  - g^{\ov, \sf})}_{\ds g^{\ov, \inst}},
    \end{equation}
    i.e.\ the difference of the (divergent) $L^2$-metrics is given by the instanton part of the Ooguri-Vafa metric
    \begin{equation}
        g^{\ov, \inst} = V^\inst(m) \| dm \|^2,
    \end{equation}
    where $V^\inst$ is the sum of Bessel functions in \eqref{eq:v-sf-inst}.
\end{cor}

This was originally anticipated in a numerical calculation by Andy Neitzke (relayed in a personal communication), which this project set out to explain.
\end{remark}

\appendix

\section{Background on Hitchin's hyperkähler metric} 
\label{sec:hyperkahler-metric}

In order to establish notation and conventions, we will briefly recall some salient facts about the hyperkähler metric on a moduli space of Higgs bundles.
For simplicity we will specialize to $\SL(2)$-Higgs bundles of degree $0$ on a compact Riemann surface $C$, as first studied in \cite{Hitchin:1987}.\footnote{In the main text we will consider surfaces with punctures (i.e.\ allow the Higgs bundles to have singularities); see \cref{sec:setup} for the relevant definitions.}

In this case the Hitchin moduli space $\M^\Higgs$ parametrizes \emph{harmonic bundles} $(E, \theta, h)$ where:
\begin{itemize}
    \item $E$ is a holomorphic rank $2$ vector bundle over $C$.
    \item $\theta$, the \emph{Higgs field}, is a traceless holomorphic bundle map
    \begin{equation}
        \theta: E \to E \otimes K_C,
    \end{equation}
    where $K_C$ is the canonical bundle of holomorphic $1$-forms on $C$.

     \item $h$, the \emph{harmonic metric}, is a hermitian metric on $E$ satisfying Hitchin's equation
      \begin{equation} \label{eq:INTRO-hitchin}
          F_{D_h} + [ \theta, \theta^{\dagger_h}] = 0.
      \end{equation}
    Here $D_h$ is the canonical Chern connection associated to $(E, h)$ and $F_{D_h}$ is its curvature.
\end{itemize}

\subsection{\texorpdfstring{$L^2$-metric}{L²-metric}} 
\label{sub:l2-metric}

There is a naturally defined (but highly transcendental) Riemannian metric $g_{L^2}$ on $\M^\Hit$, involving solutions to Hitchin's equation \eqref{eq:INTRO-hitchin}.
The metric can be written down as follows (see the presentation in \cite{Dumas:2019}).

A tangent vector $v$ to $\M^\Hit$ is represented by a $1$-parameter family of harmonic bundles $(E_t, \theta_t, h_t)$, up to gauge transformations.
By identifying the underlying hermitian bundles $(E_t, h_t)$ we obtain a family of unitary connections $D_t = d + A_t$ and $1$-forms $\Theta_t \coloneqq \theta_t - \theta_t^{\dagger_h} \in \Omega^1(\mathfrak{su}(E))$, with variations
\begin{equation}
    (\dot{A}, \dot{\Theta}) = \left.\frac{d}{dt}\right|_{t=0} (D_t, \Theta_t) \in \Omega^1(\mathfrak{su}(E))^2.
\end{equation}
The $L^2$-norm of $v$ is given by
\begin{equation} \label{eq:l2-norm}
    g_{L^2}(v, v) = \int_C \|\dot{A} \|^2 + \|\dot{\Theta} \|^2
\end{equation}
for the choice of representative $(\dot{A}, \dot{\Theta})$ which minimizes the norm (or equivalently, is $L^2$-orthogonal to the gauge orbit through $(D_0, \Theta_0)$).
Such a representative is said to be \emph{in Coulomb gauge}.

\subsection{Hyperkähler structure}

The metric $g_{L^2}$ is in fact \emph{hyperkähler}, meaning it is compatible with three complex structures $I, J, K$ satisfying the quaternion relations. 

The first complex structure $I$ is the natural complex structure on $\M^\Hit$ viewed as a space of Higgs bundles.
In particular, it acts on a variation $\dot{\theta}$ of a Higgs field by
\begin{equation} \label{eq:I-on-higgs-variation}
    I(\dot{\theta}) = i \dot{\theta}.
\end{equation}
The other complex structures naturally come from viewing $\M^\Hit$ as a space of flat connections, as we will describe later.

If $w_I, w_J, w_K$ denote the three Kähler forms, then
\begin{equation}
    \Omega_I \coloneqq \omega_J + i \omega_K
\end{equation}
is a holomorphic symplectic form with respect to $I$, and likewise with the indices cyclically permuted.
The metric can be recovered by
\begin{equation} \label{eq:metric-from-holo-form}
    g_{L^2}(\cdot, \cdot) = \omega_I(\cdot, I\cdot) = \Re \Omega_K(\cdot, I \cdot)
\end{equation}
(and cyclic permutations thereof).
    
The above data can be packaged into a \emph{twistor family} of complex structures $I_\zeta$ for $\zeta \in \CP^1$ with corresponding holomorphic symplectic forms $\Omega_\zeta$ (see \cite{Hitchin:1987b,Hitchin:1992b}, or \cite[Section 3]{Gaiotto:2010} for a summary).
This comes from the fact that $a I + bJ + cK$ is also a complex structure for any
\begin{equation*}
    (a,b,c) \in S^2 \leftrightarrow \zeta \in \CP^1,
\end{equation*}
identified stereographically so that
\begin{equation} \label{eq:zeta-complex-structures}
     I \leftrightarrow I_{\zeta = 0}, \quad J \leftrightarrow I_{\zeta = -i}, \quad K \leftrightarrow I_{\zeta = 1}.
\end{equation}

\subsection{Flat connections and nonabelian Hodge} 
As mentioned in \cref{sec:introduction}, for each $\zeta \in \C^*$ the nonabelian Hodge map
\begin{equation*}
    \NAH_\zeta: (E, \theta, h) \mapsto (E, \nabla_\zeta =  \zeta^{-1} \theta + D_h + \zeta \theta^{\dagger_h})
\end{equation*}
identifies $(\M^\Hit, I_\zeta)$ with the moduli space $\M^\dR$ of flat $\SL(2, \C)$-connections on $C$ (in its natural complex structure).
Furthermore, the holomorphic symplectic form $\Omega_\zeta$ for $(\M^\Hit, I_\zeta)$ satisfies
\begin{equation} \label{eq:higgs-sympl-form}
    \Omega_\zeta = -(\NAH_\zeta)^* \Omega^\AB,
\end{equation}
where 
\begin{equation*} 
    \Omega^\AB (\dot{\nabla}_1, \dot{\nabla}_2) = \int_C \tr (\dot{\nabla}_1 \wedge \dot{\nabla}_2)
\end{equation*}
is the holomorphic symplectic \emph{Atiyah-Bott form} on $\M^\dR$ \cite{Atiyah:1983}.\footnote{The negative sign in \eqref{eq:higgs-sympl-form} is just due to our choice of conventions; we could instead have absorbed it into the definition of $\Omega^{\AB}$.}

In particular, with notation as above,
\begin{equation}
    \Omega_K \equiv \Omega_{\zeta = 1} = -(\NAH_{\zeta = 1})^* \Omega^\AB,
\end{equation}
and so the hyperkähler metric satisfies
\begin{equation} \label{eq:metric-and-form}
    g_{L^2}(\cdot, \cdot) = -\Re  \, (\NAH_{\zeta = 1})^* \Omega^\AB(\cdot, I \cdot).
\end{equation}

\section{Construction of the Ooguri-Vafa space} 
\label{sec:ooguri-vafa-construction}

In this appendix we will briefly explain how the Ooguri-Vafa space $\Mov$ can be constructed using the Gibbons-Hawking ansatz.
The construction itself is not crucial for understanding our results, but we include an overview for completeness. 
(The essential features of $\Mov$ that we will use are summarized in \cref{ssub:summary-ov-space}.)
Many more technical details and proofs can be found in \cite{Gross:2000}; 
we will mostly follow the presentation of the summaries in \cite{Gaiotto:2010,Tulli:2019}.

\subsection{Gibbons-Hawking ansatz} 
\label{sub:gibbons-hawking-ansatz}

The \emph{Gibbons-Hawking ansatz} is a recipe for producing a hyperkähler metric on a $U(1)$-bundle $X$ over an open subset of $\R^3$, using a positive harmonic function $V$ (which satisfies a certain integrality condition).
The Ooguri-Vafa space \cite{Ooguri:1996} can be constructed as a quotient of such a space $X$, for a particular choice of harmonic function $V$ with an additional $\Z$-symmetry.

\begin{construction}[Gibbons-Hawking]
    \begin{enumerate}
        \item[]
        \item Fix an open subset $U \subseteq \R^3$ with coordinates $x_1, x_2, x_3$.
        Let $V: U \to \R_{>0}$ be a positive harmonic function such  that the homology class $[\frac{i}{2\pi} F] \in H^2(U, \R)$ is integral, where $F \coloneqq 2 \pi i {\star} V$ and $\star$ denotes the usual Hodge star on $\R^3$.
        
        \item By the above integrality assumption, $[\frac{i}{2\pi} F]$ is the first Chern class of a $U(1)$-bundle $X \to U$, i.e.\ there exists a $U(1)$-connection $\Theta \in \Omega^1(X, \mathfrak{u}(1) = i \R)$ on $X$ with curvature $d \Theta = \pi^* F$.
        To simplify the notation, we will consider the real form $\tilde{\Theta} \coloneqq \frac{\Theta}{2 \pi i}$ and omit pullbacks $\pi^*$.

        \item Define three (real) symplectic forms by
        \begin{gather}
            \omega_1 \coloneqq - \tilde{\Theta} \wedge d x_1 + V dx_2 \wedge dx_3, \\
            \omega_2 \coloneqq - \tilde{\Theta} \wedge d x_2 + V dx_3 \wedge dx_1, \\
            \omega_3 \coloneqq - \tilde{\Theta} \wedge d x_3 + V dx_1 \wedge dx_2.
        \end{gather}
    
    \item The complex form
    \begin{equation}
        \Omega_1 \coloneqq  \omega_2 + i \omega_3
    \end{equation}
    can be written as $\Omega_1 = V \alpha_1 \wedge d z_1$, where
    \begin{gather}
        \alpha_1 \coloneqq - \tilde{\Theta} V^{-1} + i d x_1, \\ 
        z_1 \coloneqq x_2 + i x_3.
    \end{gather}
    There exists a unique integrable complex structure $I_1$ on $X$ for which $\Omega_1$ is a holomorphic symplectic form.
    This structure $I_1$ acts by $+i$ on $ dz_1$ and $\alpha_1$ (so in particular $z_1: X \to \C$ is an $I_1$-holomorphic map).

    \item The other two complex structures $I_2, I_3$ are obtained by cyclically permuting all of the formulas above.
    They satisfy $I_1 I_2 = I_3$ and define a hyperkähler structure with corresponding metric
    \begin{equation} \label{eq:gibbons-hawking-metric}
        g = V^{-1} \tilde{\Theta}^2 + V || d \vec{x} ||^2.
    \end{equation}
\end{enumerate}
\end{construction}

\begin{notation}[Complex structures]
To harmonize the convention
\begin{equation}
    I_\zeta = \frac{i (-\zeta + \overline{\zeta}) I_1 - (\zeta + \overline{\zeta})I_2 + (1 - |\zeta|^2) I_3 }{1+|\zeta|^2}
\end{equation}
from \cite{Gaiotto:2010} with our stereographic conventions \eqref{eq:zeta-complex-structures} for $I, J, K$, we will call
\begin{equation}
    I \coloneqq I_3, \quad J \coloneqq -I_1, \quad K \coloneqq -I_2.
\end{equation}
\end{notation}

\subsection{Ooguri-Vafa from Gibbons-Hawking} 
\label{sub:ooguri-vafa-from-gibbons-hawking}

The Ooguri-Vafa space is obtained by applying a slightly modified version of the Gibbons-Hawking construction with the harmonic function
\begin{equation}
    V(x_1, x_2, x_3) \coloneqq \frac{1}{4 \pi} \left( \sum_{n=-\infty}^\infty \frac{1}{\sqrt{x_1^2 + x_2^2 + (x_3 + n)^2}} - c_n \right),
\end{equation}
where $c_n \in \R_{>0}$ are constants chosen so that the sum converges\footnote{e.g.\ one could take $c_n = \tfrac{1}{n+\tfrac{1}{2}}$}.
By Poisson resummation one can rewrite
\begin{align} \label{eq:v-sf-inst}
    V(x_1, x_2, x_3) &= \underbrace{-\frac{1}{2\pi} \log \left(\frac{|z|}{|\Lambda|} \right) \vphantom{\sum_{n \in \Z \setminus \{0\}}} }_{\ds \eqqcolon V^\sf} + \underbrace{\frac{1}{2\pi} \sum_{n \in \Z \setminus \{0\}} e^{2 \pi i n x_3} K_0 (2 \pi |nz|)}_{\ds \eqqcolon V^\inst} 
\end{align}
where $z \coloneqq z_3 = x_1 + i x_2$, the constant $|\Lambda| \in \R_{>0}$ depends on the choice of $c_n$, and $K_0$ is the zeroth modified Bessel function of the second kind.

This choice of $V$ is positive on an open set
\begin{equation*}
    U = \B \times \R \setminus \{(0,0) \} \times \Z
\end{equation*}
for some sufficiently small neighbourhood $\B \subseteq \R^2 \cong \C$ of the origin, and satisfies the required integrality condition.
The Gibbons-Hawking ansatz therefore produces a hyperkähler $U(1)$-bundle $X \to U$ as above.

Furthermore, $V$ is invariant under the $\Z$-action $x_3 \to x_3 + n$. 
This action can be lifted to $X$,\footnote{There is really a $U(1)$-worth of choice for the lift, which can be combined with the constant $|\Lambda| \in \R_{>0}$ into the complex cutoff parameter $\Lambda \in \C^*$.} and taking the quotient gives a $U(1)$-bundle $\tilde{X}$ over 
    \begin{equation*}
        \tilde{U} = \B \times S^1 \setminus \{(0,0,1) \}.
    \end{equation*}
The hyperkähler structure descends to $\tilde{X}$, and smoothly extends over $\B \times S^1$ by adding a point---this extension is called the Ooguri-Vafa space $\M^\ov(\Lambda)$.
It can be viewed as a singular torus fibration over $\B$, as shown in \cref{fig:ov-space}.
The torus fibres are parametrized by the electric angle $\theta_e = 2\pi x_3$  and a (locally defined) magnetic angle $\theta_m$ on the $U(1)$-bundle $\tilde{X}$.

\subsection{Explicit metric formulas} 
\label{sub:explicit-metric-formulas}

For reference in \cref{sec:metric-on-the-hitchin-section}, we will give some more explicit formulas for the Ooguri-Vafa metric $g^\ov$ on the Hitchin section.
As we explain in \cref{sec:duality-on-hitchin-section}, under the identification $\Xfr \cong \Mov$, the Hitchin section $\Bfr$ is given by the locus in $\Mov$ with 
\begin{itemize}
    \item $x_3 \equiv \frac{1}{2}$ (corresponding to the choice of parabolic weights $m^{(3)} = \frac{1}{2}$), and
    \item $\theta_m \equiv 0$ (corresponding to the choice of framing).
\end{itemize}
This will allow us to simplify many of the formulas from above.

The $U(1)$-connection $\Theta$ can be written \cite{Gaiotto:2010} as $\Theta = i d \theta_m + 2\pi i A$
for
\begin{equation}
    \begin{split}
    A = &\frac{i}{4\pi} \left(\log \left(\frac{z}{\Lambda} \right) - \log \left( \frac{\overline{z}}{\overline{\Lambda}}  \right)  \right) dx_3 \\
    &- \frac{1}{4\pi} \left( \sum_{n \neq 0} \text{sgn}(n) e^{2 \pi i n x_3} |z| K_1(2 \pi |nz|)  \right) \left(\frac{dz}{z} - \frac{d \overline{z}}{\overline{z}} \right),
    \end{split}
\end{equation}
where $K_1$ is the first modified Bessel function of the second kind.
Note that $A$ vanishes when $x_3 \equiv \frac{1}{2}$.

It follows that on the Hitchin section, the metric \eqref{eq:gibbons-hawking-metric} produced by the Gibbons-Hawking ansatz takes the simple form
\begin{equation}
    g^\ov|_{\Bfr} = V \|dz\|^2.
\end{equation}
Similarly, the semiflat metric is given by
\begin{equation}
    g^{\ov, \sf}|_{\Bfr} = V^\sf \|dz\|^2,
\end{equation}
where
\begin{equation*}
    V^\sf = -\frac{1}{2\pi} \log \left(\frac{|z|}{|\Lambda|} \right).
\end{equation*}

Note that with the conventions established above, the coordinate $z = x_2 + i x_3$ is holomorphic with respect to the complex structure $I^\ov = I_3$.
One can directly see from the formulas that the metric satisfies
\begin{equation}
    g^{\ov, \sf}(\cdot, \cdot) = \Re \Omega^{\ov, \sf}_{\zeta = 1}(\cdot, I^\ov \cdot)
\end{equation}
on $\Bfr$ (cf.\ \eqref{eq:metric-from-holo-form}).

Under the correspondence $\Mov(\Lambda = 4 i) \cong \Xfr$ (which associates $z \leftrightarrow -2i m$), we can regard $\Bfr$ as a polynomial Hitchin section as in \cref{sub:metric-regularization}, and write
\begin{equation} \label{eq:ov-sf-polynomial-norm}
    g^{\ov, \sf}(\dot{P}, \dot{P}) = -\frac{2}{\pi} \log \left( \frac{|m|}{2} \right) |\dot{m}|^2
\end{equation}
for a variation $\dot{P}$ of $z^2 + 2m$.

\section{Parabolic bundles and framing} 
\label{sec:parabolic-bundles-and-framing}

In this appendix we discuss the relation between compatible frames, eigenframes, and parabolic structures for bundles $(E, \theta, h, g) \in \Hfr$.

\subsection{Parabolic and filtered bundles} 
\label{sub:parabolic-and-filtered-bundles}

We start by reviewing some general definitions involving parabolic and filtered bundles (see e.g.\ \cite{Simpson:1990,Mochizuki:2011}).
We will use similar notation and conventions as in \cite{Tulli:2019}.

Fix a Riemann surface $C$ and a finite subset $D \subseteq C$.

\begin{definition}[Parabolic vector bundles]
    Let $\mathbf{c} = (c_p)_{p \in D} \in \R^D$.
    A $\mathbf{c}$-\emph{parabolic vector bundle} over $(C, D)$ consists of a holomorphic vector bundle $E \to C$, together with an increasing flag of vector spaces $E_{p,i}$ and increasing sequence of weights $\alpha_{p,i} \in (c_p -1, c_p]$ at each $p \in D$:
    \begin{equation}
    \begin{alignedat}{6}
    0 = E_{p,0} &  &&\subset E_{p,1} &&\subset E_{p,2} &&\subset \cdots &&\subset E_{p,n_p} &&= E|_p \\
    c_p - 1 & &&< \alpha_{p,1} &&< \alpha_{p,2} &&< \cdots &&< \alpha_{p,n_p} &&\leq c_p.
    \end{alignedat}
    \end{equation}
    Define the \emph{multiplicity} of the weight $\alpha_{p,i}$ to be $m_{p,i} \coloneqq \dim E_{p,i} - \dim E_{p, i-1}$, and the \emph{parabolic degree} of $E$ by
    \begin{equation}
        \pdeg E \coloneqq \deg E - \sum_{p \in D} \sum_{i=1}^{n_p}  m_{p,i} \alpha_{p,i}.
    \end{equation}
\end{definition}
Most of the constructions in \cite{Tulli:2019} involve $\frac{1}{2}$-parabolic bundles (i.e.\ restrict the weights to $(-\frac{1}{2}, \frac{1}{2}]$), but we will occasionally use the more general notion.

\begin{definition}[Parabolic frame]
    A frame $(\eta_1, \dots, \eta_r)$ for a rank $r$ parabolic vector bundle $E$ near $p$ is \emph{compatible} with the parabolic structure 
    \begin{equation*}
        0  \subset E_{p,1} \subset E_{p,2} \subset \cdots \subset E_{p,n_p} = E|_p
    \end{equation*}
    if there is a subsequence $1 \leq k_1 <  k_2 < \dots < k_{n_p} = r$ such that
    $(\eta_j: j \leq k_i )$ is a frame of $E_{p, i}$ for each $i$.
\end{definition}

The flag data of a parabolic bundle can be expressed in terms of increasing filtrations $\mathcal{P}_* (E|_p)$ indexed by $(c_p - 1, c_p]$.
This can equivalently be formulated without any restrictions on the weights using the notion of a filtered bundle.

\begin{definition}[Filtered bundles] \label{def:filtered-bundles}
    A \emph{filtered bundle} over $(C,D)$ consists of a meromorphic vector bundle $E \to C$ with poles at $D$ (i.e.\ a locally free finite-rank $\O_C(*D)$-module) together with a family $\mathcal{P}_*E = (\mathcal{P}_\alpha E: \alpha \in \R^D)$ of holomorphic subbundles of $E$ (i.e.\ locally free $\O_C$-submodules) such that: 
    \begin{enumerate}[(i)]
        \item  \label{item:filtered-def-1} $\mathcal{P}_\alpha E|_{C \setminus D} = E|_{C \setminus D}$. 

        \item For $p \in D$, the stalk $\mathcal{P}_\alpha E|_p$ depends only on the weight $\alpha_p \coloneqq \alpha(p) \in \R$. 
        We will write $\mathcal{P}_{p, \alpha_p} E  \coloneqq \mathcal{P}_\alpha E|_p$. 

        \item \label{item:filtered-def-3} $\mathcal{P}_{p, \alpha_p} E \subseteq \mathcal{P}_{p, \beta_p} E$ if $\alpha_p \leq \beta_p$, and $\mathcal{P}_{p, \alpha_p+\epsilon} E = \mathcal{P}_{p, \alpha_p} E$ for small $\epsilon > 0$. 
        
        \item \label{item:filtered-def-4} If $w$ is a local coordinate centred at $p \in D$, then $w \mathcal{P}_{p, \alpha_p} E = \mathcal{P}_{p, \alpha_p+1} E$.
    \end{enumerate}
\end{definition}

By property \ref{item:filtered-def-3}, for each $\mathbf{c} \in \R^D$ and $p \in D$ there are finitely many parabolic weights
\begin{equation}
    \{\alpha_p \in (c_p-1, c_p]: \mathcal{P}_{p, \alpha_p} E \neq \mathcal{P}_{p, \alpha_p-\epsilon} E \text{ for small $\epsilon > 0$}\}.
\end{equation}
This defines a corresponding $\mathbf{c}$-parabolic bundle $_\mathbf{c}E$, called the  $\mathbf{c}$-truncation of $\mathcal{P}_*E$. 
Conversely, any $\mathbf{c}$-parabolic bundle determines a filtered bundle by property \ref{item:filtered-def-4}.
The \emph{parabolic degree} of $\mathcal{P}_*E$ is defined to be the parabolic degree of any of its $\mathbf{c}$-truncations $_\mathbf{c}E$.\footnote{Note that replacing $c_p \to c_{p}+1$ increases the degree of $_\mathbf{c}E$ by $\rank E$ (by property \ref{item:filtered-def-4} in \cref{def:filtered-bundles})
and increases each weight $\alpha_{p, i}$ by 1, so $\pdeg \mathcal{P}_*E \coloneqq \pdeg (_\mathbf{c}E)$ is independent of the choice of $\mathbf{c}$.}

\begin{definition}[Parabolic weight of a section]
    Any section $\eta$ of the meromorphic bundle $E$ in a neighbourhood of $p$ has an associated parabolic weight 
    \begin{equation}
        \nu_p(\eta) \coloneqq \min \{\alpha_p : \eta \in \mathcal{P}_{p, \alpha_p}  \}.
    \end{equation}
\end{definition}

\begin{remark}[Growth rate filtration]
    The harmonic bundles in $\H$ naturally carry a filtered structure.
    More generally, given any wild harmonic bundle $(E, \theta, h)$ over $C \setminus D$, there is a corresponding filtered bundle $\mathcal{P}_*^h E$
    defined using the growth rates of sections near $D$ with respect to the harmonic metric $h$. 
    Each stalk $\mathcal{P}_{p, \alpha_p}^h E$ 
    consists of holomorphic sections $s$ in a punctured neighbourhood of $p$ such that
    \begin{equation} \label{eq:growth-filtration-condition}
      |s|_h = \O(|w|^{-\alpha_p - \epsilon}) \quad \text{for every } \epsilon > 0.
    \end{equation}
    We will refer to $\mathcal{P}_*^h E$ as the \emph{growth rate filtration}.
    
    The $\mathbf{c}$-truncations of $\mathcal{P}_*^h E$ are compatible with the Higgs field $\theta$, in the sense that for each $\mathbf{c} \in \R^D$ and $p \in D$ there exists a holomorphic frame of $\theta$-eigenvectors near $p$ compatible with the parabolic structure (i.e.\ with appropriate growth rates) \cite{Mochizuki:2011}.
\end{remark}

\subsection{Constructing and extending compatible frames} 
\label{sub:constructing-and-extending-compatible-frames}

Certain eigenframes can be used to produce a parabolic structure and compatible frame for a bundle in $\Hfr$.
This is explained in two constructions from \cite{Tulli:2019}, which we summarize below.

\begin{construction}[Elements of $\H$, {\cite[Lemma 3.1]{Tulli:2019}}] \label{constr:bundle-from-parameters}
    For each $m \in \C^*$ and $m^{(3)} \in (-\frac{1}{2}, \frac{1}{2}]$, the following construction yields a wild harmonic bundle $(E, \theta, h) \in \H$
    with the specified parameters:
    \begin{enumerate}
        \item Start with the trivial rank 2 bundle $E$ over $\CP^1 \setminus \{\infty\}$, equipped with its standard global frame $(e_1, e_2)$ and Higgs field
        \begin{equation*}
            \theta = \begin{pmatrix}
                0 & 1 \\ 
                z^2 + 2m & 0 
            \end{pmatrix} dz.
        \end{equation*}

        \item Choose a holomorphic $\theta$-eigenframe $(\eta_1, \eta_2)$ in a punctured neighbourhood of $\infty$ such that 
        \begin{equation} \label{eq:eigenframe-wedge}
            \eta_1 \wedge \eta_2  = 
            \begin{cases}
                 e_1 \wedge e_2 & \text{if } m^{(3)} \in (-\frac{1}{2}, \frac{1}{2}),  \\
                 z e_1 \wedge e_2  & \text{if } m^{(3)} = \frac{1}{2},
            \end{cases}
        \end{equation}
        and use it to extend $E$ over $\infty$.

        \item If $m^{(3)} \in (-\frac{1}{2}, \frac{1}{2})$, assign parabolic weights $+m^{(3)}$ to $\eta_1$ and $-m^{(3)}$ to $\eta_2$.
        Otherwise assign $m^{(3)} = \frac{1}{2}$ to both $\eta_1$ and $\eta_2$. 
        This defines a $\frac{1}{2}$-parabolic bundle with the desired parameters $m$ and $m^{(3)}$ and parabolic degree $0$, so the harmonic metric $h$ exists by the general theory of \cite{Biquard:2004}.
        Thus we have an element of $\H$.
    \end{enumerate}
\end{construction}

\begin{construction}[Compatible frame from $\theta$-eigenframe, {\cite[Proposition 3.2]{Tulli:2019}}] \label{constr:frame-from-eigenframe}
    Given a wild harmonic bundle $(E, \theta, h) \in \H$ 
    (e.g.\ produced as above), the following construction yields a compatible frame $g$:

    \begin{enumerate}
        \item Choose a holomorphic eigenframe $(\eta_1, \eta_2)$ compatible with the parabolic structure, ordered so that $\theta$ is of the framed form
        \begin{equation*}
            \theta = -H \frac{dw}{w^3} - mH \frac{dw}{w} + \text{diagonal holomorphic terms}.
        \end{equation*}
        (If $(E, \theta, h)$ was produced using \cref{constr:bundle-from-parameters}, we can use the same frame $(\eta_1, \eta_2)$ as before.)
        Note that $\eta_1$ and $\eta_2$ are asymptotically exponentially orthogonal near $w=0$ \cite{Mochizuki:2011}.

        \item Let
        \begin{equation}
        (v_1, v_2) \coloneqq \begin{cases}
            (\eta_1, \eta _2) & \text{if }  m^{(3)} \in (-\frac{1}{2}, \frac{1}{2}),  \\
            (\eta_1, w\eta_2) & \text{if } m^{(3)} = \frac{1}{2},
        \end{cases}
        \end{equation}
        so that $v_1$ and $v_2$ have respective parabolic weights $+m^{(3)}$ and $-m^{(3)}$.

        \item Let $(e_1, e_2)$ be the Gram-Schmidt orthonormalization of the frame $(v_1, v_2)$ with respect to the harmonic metric $h$, and use it construct a unitary extension of the bundle over $\infty$.
        Then $g = (e_1, e_2)|_\infty$ is a unitary frame with respect to which $\theta$ and $\overline{\partial}_E$ have the desired forms \eqref{eq:higgs-framed-form} and \eqref{eq:holo-str-framed-form}.
    \end{enumerate}
\end{construction}

Together these two constructions produce a framed bundle $(E, \theta, g)$ for any choice of parameters $m$ and $m^{(3)}$.
For each such $(E, \theta, g)$, the isomorphism classes $e^{i \vartheta} \cdot [(E, \theta, g)] = [(E, \theta, e^{i \frac{\vartheta}{2}} \cdot g)]$ exhaust $\Xfr(m, m^{(3)})$.
This gives us a fairly concrete representative for each element in $\Xfr$.

By specifying the initial choice of eigenframe $(\eta_1, \eta_2)$, we can describe the resulting frame $g$ even more explicitly.
We will do this in \cref{sub:self-dual-frames} to construct a framed version of the Hitchin section.

We can also go in the other direction, from a compatible frame to an eigenframe.

\begin{lemma}[Extension to $\theta$-eigenframe] \label{lem:extension-to-eigenframe}
    \begin{enumerate}
        \item[]
        \item Given $(E, \theta, g) \in \Hfr$, the frame $g$ admits an extension to an eigenframe $(\tilde{\eta}_1, \tilde{\eta}_2)$ for $\theta$ near $\infty$.

        \item  Furthermore, the extension can be chosen to be an $SU(2)$-frame with respect to the semiflat metric $h_\sf$ (see \cref{ssub:sf-metrics}), and so that
    \begin{align*}
        D_{h_\sf} &= d +  \frac{m^{(3)}}{2} H \left(  \frac{dw}{w} - \frac{d \overline{w}}{\overline{w}} \right)
    \end{align*}
    with respect to the frame near $w=0$.
    \end{enumerate}
\end{lemma}
\begin{proof}
    \emph{(1)} First suppose that $(E, \theta, g)$ was produced by orthonormalizing a eigenframe $(v_1, v_2)$ as in \cref{constr:frame-from-eigenframe}. 
    Then we can choose the extension $(\tilde{\eta}_1, \tilde{\eta}_2) \coloneqq (\frac{v_1}{|v_1|_h}, \frac{v_2}{|v_2|_h})$; note that it also approaches $g$ as $w \to 0$ since $v_1$ and $v_2$ are asymptotically exponentially orthogonal.
    The result also holds for $(E, \theta, e^{i \vartheta} \cdot g)$ for any $e^{i \vartheta} \in U(1)$, using the extension $e^{i \vartheta} \cdot (v_1, v_2)$.
    But then the general result follows since every isomorphism class in $\Hfr$ has a representative of this form.

    \emph{(2)} Follow a similar argument as above, but start with an initial holomorphic eigenframe $(\eta_1, \eta_2)$ such that 
    \begin{equation*}
        h_\sf = \begin{pmatrix}
            |w|^{2m^{(3)}} & 0 \\
            0 & |w|^{ -2m^{(3)}}
        \end{pmatrix}
    \end{equation*}
    near $w=0$. 
    (Such a frame is compatible with the parabolic structure---in fact, this is the model scenario, see e.g.\ \cite{Fredrickson:2022}.
    It can be obtained by pushing down the corresponding ``model frames'' for $h_\L$ over $\Sigma$.)
    Choose the extension of $g$ this time by normalizing with respect to $h_\sf$ instead of $h$, i.e.\  by taking $(\tilde{\eta}_1, \tilde{\eta}_2) \coloneqq (\frac{\eta_1}{|\eta_1|_{h_\sf}}, \frac{\eta_2}{|\eta_2|_{h_\sf}}) = (|w|^{-m^{(3)}}\eta_1, |w|^{m^{(3)}} \eta_2)$.
    Note that $(\tilde{\eta}_1, \tilde{\eta}_2)$ still approaches $g$ as $w \to 0$, since $h$ and $h_\sf$ are both compatible with the same parabolic structure.
    By a standard calculation the Chern connection $D_{h_\sf}$ is of the specified form.
\end{proof}

\section{Classical Stokes theory} 
\label{sec:classical-stokes-theory}

Stokes theory for irregular connections is typically formulated in terms of meromorphic connections, but we will prefer to work directly with the complex (but not meromorphic) framed connections $(E, \nabla_\zeta, g)$.
In this appendix we summarize some of the standard notions and results from the classical theory, and explain how they can be translated to our $C^\infty$ setting.

\subsection{Holomorphic frames} 
\label{sub:holomorphic-frames}

For each $(E, \theta, g) \in \Hfr$ and $\zeta \in \C^*$, there is a corresponding \emph{framed filtered flat bundle} $(\mathcal{P}_*^h \mathcal{E}_\zeta, \nabla_\zeta, \tau_{*,\zeta})$ as described in \cite[Section 3.4.2]{Tulli:2019}, where:

\begin{itemize}
  \item $\mathcal{E}_\zeta \coloneqq (E|_{\CP^1 \setminus \{\infty\}}, \bar{\partial}_E + \zeta \theta^\dagger)$ is a holomorphic vector bundle over $\CP^1 \setminus \{\infty\}$.

  \item $\mathcal{P}_*^h \mathcal{E}_\zeta$ is a filtered bundle over $\CP^1$, defined using the growth rate filtration induced by $h$ (see \cref{sub:parabolic-and-filtered-bundles} for these definitions).
  
  \item $\tau_{a,\zeta}$ is a holomorphic frame of $\mathcal{E}_\zeta$ in a neighbourhood of $z=\infty$, with respect to which
\begin{equation}
    \nabla_\zeta = d - (\zeta^{-1} + \overline{\zeta}) H \frac{dw}{w^3} + \Lambda(\zeta) \frac{dw}{w} + \text{holo $(1,0)$ terms},
\end{equation}
where
\begin{equation}
    \Lambda(\zeta) = \begin{pmatrix} \label{eq:exponent-formal-monodromy}
        - \zeta^{-1} m + m^{(3)} + \zeta \overline{m} + n_-(a) & 0 \\
        0 &  \zeta^{-1} m - m^{(3)} - \zeta \overline{m} + n_+(a)
    \end{pmatrix}.
\end{equation}
Here the $n_\pm(a)$ are integers chosen to ensure that the sections of the frame $\tau_{a,\zeta}$ have parabolic weights lying in $(a-1, a]$.\footnote{\cite{Tulli:2019} calls $n_- = n_1$ and $n_+ = n_2$.} 
The frame $\tau_{a,\zeta}$ thereby defines a holomorphic extension of $\mathcal{E}_\zeta$ which coincides with $\mathcal{P}_a^h \mathcal{E}_\zeta$.
\end{itemize}
More explicitly, the compatible holomorphic frame $\tau_{a,\zeta}$ for $\nabla_\zeta$ is given by 
\begin{equation} \label{eq:holo-frame}
    \tau_{a,\zeta} (w) = (e_1, e_2) \cdot g_\zeta(w) |w|^{(m^{(3)} + 2 \zeta \overline{m}) H} w^{N(a)} \exp \left(\frac{\overline{\zeta} H}{2w^2} - \frac{\zeta H}{2 \overline{w}^2}  \right),
\end{equation}
where:
\begin{itemize}
  \item $(e_1, e_2)$ is an extension of the original frame $g$ to a neighbourhood of $w=0$ (as in \cref{def:framed-harmonic}).
  \item $g_\zeta$ is a gauge transformation in a neighbourhood of $w=0$ that kills the regular $(0,1)$-part of $\bar{\partial}_E + \zeta \theta^\dagger$ and satisfies $g_\zeta(0) = \bbid$, obtained from \cite[Section 8]{Biquard:2004}.
  \item $N(a) = \diag(n_-(a), n_+(a))$, where $n_\pm(a)$ is the unique integer such that
  \begin{equation*}
      n_\pm(a) \mp (m^{(3)} + 2 \Re(\zeta \overline{m})) \in (a-1, a].
  \end{equation*}
\end{itemize}
We will write $(\mathcal{P_*E}, \nabla_\zeta, \tau_*) \coloneqq (\mathcal{P}^h_* \mathcal{E}_\zeta, \nabla_\zeta, \tau_{*,\zeta})$ to simplify the notation.

\subsection{Sectorial Stokes data} 
\label{sub:sectorial-stokes-data}

As above, fix $(E, \theta, g) \in \Hfr$ and $\zeta \in \C^*$, and consider the associated framed filtered flat bundle $(\mathcal{P_*E}, \nabla_\zeta, \tau_*)$ over $\CP^1$.
We will recall some key facts involving the classical Stokes data of the connection $\nabla_\zeta$. 
A more general summary of the theory can be found in \cite[Section 3.4.1]{Tulli:2019}, following \cite{Boalch:2001,Boalch:2002,Witten:2008}, but we will just state what is needed for our application.

Choose $a \in \R$ and consider the fixed holomorphic extension $(\mathcal{P}_a \mathcal{E}, \nabla_\zeta, \tau_a)$.
There is a unique formal gauge transformation $\hat{F}_a$ such that $\hat{F}_a(0) = \bbid$ and such that the connection has the diagonal form
\begin{equation} \label{eq:holo-diag-form}
    \nabla_\zeta = d - (\zeta^{-1} + \overline{\zeta}) H \frac{dw}{w^3} + \Lambda(\zeta) \frac{dw}{w}
\end{equation}
in the formal frame $\tau_a \cdot \hat{F}_a$ (see e.g.~\cite[Lemma~1]{Boalch:2002}).

The connection $\nabla_\zeta$ has four anti-Stokes rays $r_1, \dots, r_4$ and four Stokes rays, corresponding to directions in the $w$-plane where $-(\zeta^{-1} + \overline{\zeta}) \frac{1}{w^2}$ is real resp.\ imaginary.
(These coincide with the rays defined in \cref{ssub:stokes-conventions}.)
Recall that $\Sect_i$ denotes the sector bounded by the anti-Stokes rays $r_i$ and $r_{i+1}$, and $\eSect_i$ denotes the extended sector bounded by the adjacent Stokes rays.
The following sectorial asymptotic existence theorem (quoted from \cite[Theorem 3.6]{Tulli:2019}) is a fundamental classical result. 

\begin{prop}[Sectorial asymptotic existence, holomorphic version]
\label{prop:asymptotic-existence-holo}
  In a neighbourhood of $w=0$ in each extended sector $\eSect_i$, there is a unique invertible matrix $\Sigma_i$ of holomorphic functions such that the connection $\nabla_\zeta$ has the diagonal form \eqref{eq:holo-diag-form} in the sectorial frame $\tau_a \cdot \Sigma_i $.

  Furthermore, $\Sigma_i = \bbid + \O(|w|)$ as $w \to 0$ in $\eSect_i$.\footnote{In fact, more is true: each $\Sigma_i$ is asymptotic to the formal series $\hat{F}_a$ as $w \to 0$ in $\eSect_i$.}  
\end{prop}

Thus in each sector $\eSect_i$, the frame of flat sections $\Phi_i$ is explicitly given by
\begin{equation}
    \Phi_i = \tau_a \cdot \Sigma_i (\hat{F}_a) w^{- \Lambda(\zeta)} \exp \left(- (\zeta^{-1} + \overline{\zeta}) \frac{H}{2w^2}  \right).
\end{equation}

We will want to work directly with the original frame $g$ instead of $\tau_a$.
Recall that the Stokes matrix $S_i$ is the transition matrix from $\Phi_i$ to $\Phi_{i+1}$ on $\eSect_i \cap \eSect_{i+1}$.
The sections $\Phi_i$ are independent of the choice of $a \in \R$ \cite[Section 3.4.3]{Tulli:2019}, and hence so are the Stokes matrices $S_i$ and the formal monodromy $M_0 = e^{- 2 \pi i \Lambda(\zeta)}$.\footnote{The $a$-independence of $M_0 = e^{- 2 \pi i \Lambda(\zeta)}$ can be seen directly from \eqref{eq:exponent-formal-monodromy}.} 
Therefore we can unambiguously speak of the Stokes data of the complex connection $(E, \nabla_\zeta, g)$.

We will also need a translation of \cref{prop:asymptotic-existence-holo}

\begin{cor}[Sectorial asymptotic existence in terms of $g$]
\label{cor:appendix-asymptotic-existence-g}
    In a neighbourhood of $w=0$ in each extended sector $\eSect_i$, there is an invertible matrix $\tilde{\Sigma}_i$ of smooth functions such that the connection $\nabla_\zeta$ has the diagonal form \eqref{eq:diag-form} in the sectorial frame $g \cdot \tilde{\Sigma}_i$.

  Furthermore, each $\tilde{\Sigma}_i \to \bbid$ as $w \to 0$ in $\eSect_i$.
\end{cor}
\begin{proof}
    This follows from the holomorphic result by making an appropriate gauge transformation.

    Fix $a \in \R$ and rewrite \eqref{eq:holo-frame} as
    \begin{equation*}
        \tau_{a,\zeta}(w)  = (e_1, e_2) \cdot g_\zeta(w) M(w),
    \end{equation*}
    where
    \begin{equation*}
        M(w) \coloneqq |w|^{(m^{(3)} + 2 \zeta \overline{m})H} w^{N(a)}\exp \left(\frac{\overline{\zeta} H}{2w^2} - \frac{\zeta H}{2 \overline{w}^2}  \right).
    \end{equation*}
    It is straightforward to check that $\nabla_\zeta$ has the desired diagonal form \eqref{eq:diag-form} with respect to the frame
    \begin{align*}
        (\tau \cdot \Sigma_i) M^{-1} = (e_1, e_2) \cdot  \underbrace{g_\zeta M \Sigma_i M^{-1}}_{\ds \eqqcolon \tilde{\Sigma}_i}.
    \end{align*}
    We know that $g_\zeta \to \bbid$ as $w \to 0$, so it only remains to check that $M \Sigma_i M^{-1} \to \bbid$ as $w\to 0$ in $\eSect_i$. 
    An almost identical calculation appears in the proof of \cite[Lemma 3.5]{Tulli:2019}.

    In short, since
    \begin{equation*}
        \left|M(w) \right| = |w|^{(m^{(3)} + 2\Re( \zeta \overline{m})) H + N(a)},
    \end{equation*}
    the diagonal terms of $M \Sigma_i M^{-1}$ are the same as those of $\Sigma_i$ (which approach $1$), while the off-diagonal terms differ from those of $\Sigma_i$ by a factor of magnitude
    \begin{equation*}
        |w|^{\pm \left[n_-(a) - n_+(a) + 2(m^{(3)} + 2 \Re(\zeta \overline{m})) \right]}.
    \end{equation*}
    The exponent above lies in $(-1, 1)$ since it is the difference of two parabolic weights in $(a-1, a]$, but the off-diagonal terms of $\Sigma_i$ are $\O(|w|)$, and so the result follows.
\end{proof}

\section{Parabolic duality and related constructions} 
\label{sec:parabolic-duality-constructions}

In this appendix we describe some technical constructions involving duality for parabolic and filtered bundles.
We will apply this to construct a framed version of the Hitchin section, which will be self-dual in the appropriate sense.

\subsection{Filtered duality} 
\label{sub:filtered-duality}

There is a natural notion of duality for filtered bundles (see e.g.\ \cite{Mochizuki:0}).

\begin{definition}[Dual filtered bundle] \label{def:dual-filtered}
    If $\mathcal{P}_*E$ is a filtered bundle over $(C, D)$, then the dual meromorphic bundle $E^\vee = \shom(E, \O_C(*D))$
    has an induced filtered structure 
    \begin{equation}
        \mathcal{P}_\alpha (E^\vee) \coloneqq \{ \phi \in E^\vee: \phi(\mathcal{P}_\beta E) \subseteq \mathcal{P}_{\beta + \alpha} \O_C(*D) \quad \forall \beta \in \R^D \},
    \end{equation}
    where
\begin{equation}
    \mathcal{P}_\alpha \O_C(*D) = \O(\sum_{p \in D} \lfloor \alpha_p \rfloor p  ).
\end{equation}
\end{definition}

\begin{remark}[Dual weights] \label{rem:dual-weights}
There is a natural isomorphism of (holomorphic) bundles
\begin{equation} \label{eq:dual-filtration-isom}
    \mathcal{P}_\alpha (E^\vee) \cong (\mathcal{P}_{< - \alpha + \mathbf{1}} E)^*.
    \end{equation}
    If $\mathcal{P}_{p, *} E$ has parabolic weights $\{\alpha_p \}$ at $p \in D$ (i.e.\ indices at which the filtration $\mathcal{P}_{p, *} E$ jumps, not restricted to any subinterval), then $\mathcal{P}_{p,*} (E^\vee)$ has weights $\{-\alpha_p \}$.
\end{remark}

On the other hand, given a wild harmonic bundle $(E, \theta, h)$, we can consider the dual harmonic bundle (cf.\ \cref{sub:framed-duality})
\begin{equation*}
    (E, \theta, h)^* \coloneqq (E^*, \theta^t, h^*).
\end{equation*}
This is compatible with \cref{def:dual-filtered} in the sense that the growth rate filtration induced by $(E, \theta, h)^*$ is dual to that induced by $(E, \theta, h)$. 

Describing duality on the level of parabolic bundles is slightly more subtle.
Motivated by \cref{rem:dual-weights}, we make the following (somewhat indirect) definition:
\begin{definition}[Dual parabolic bundle]
    If $E$ is a $\mathbf{c}$-parabolic bundle with corresponding filtered bundle $\mathcal{P}_* E$, define the dual parabolic bundle $E^*$ to be the ($\mathbf{-c+1}$)-truncation of the dual filtered bundle $\mathcal{P}_* (E^\vee)$.
\end{definition}  
Note that under this definition, the dual of a $\frac{1}{2}$-parabolic bundle is again $\frac{1}{2}$-parabolic.

We would like describe the dual flag data more explicitly.
There is some asymmetry since the dual weights are truncated to half-open intervals $(-c_p, -c_p+1]$. 
If $E$ has parabolic weights $\alpha_1 < \alpha_2 < \dots < \alpha_{n_p}$ at $p \in D$, then $E^*$ must have weights
\begin{align}
    \begin{cases}
        -\alpha_{n_p} < -\alpha_{n_p -1} < \dots < -\alpha_2 < -\alpha_1 & \text{if } \alpha_{n_p} \neq c_p,  \\
         \phantom{-\alpha_{n_p} < } -\alpha_{n_p-1} < \dots < -\alpha_2 < -\alpha_1 < -c_p + 1 & \text{if } \alpha_{n_p} = c_p.
    \end{cases}
\end{align}

It is most natural to think of these weights in terms of corresponding (dual) sections.
Suppose $E|_p$ has flag data
\begin{alignat*}{6}
    0 = E_{p,0} &  &&\subset E_{p,1} &&\subset E_{p,2} &&\subset \cdots &&\subset E_{p,n_p} &&= E|_p,
\end{alignat*}
and let $(\eta_1, \dots, \eta_{\rank E})$ be a compatible frame.
Then $(\eta_{\rank E}^*, \dots, \eta_1^*)$ is a compatible frame for the dual flag of annihilators
\begin{alignat*}{6}
    0 = (E_{p,n_p})^0 &  &&\subset (E_{p,n_p-1})^0 &&\subset (E_{p,n_p - 2})^0 &&\subset \cdots &&\subset (E_{p,0})^0 &&= (E|_p)^*,
\end{alignat*}
with corresponding weights $\nu_p(\eta_i^*) = - \nu_p(\eta_i)$.
In other words, each subspace $(E_{p, i})^0$ is naturally associated with the weight $-\alpha_{p, i +1}$.
This coincides with the desired dual parabolic structure when $\alpha_{n_p} \neq c_p$, but if $\alpha_{n_p} = c_p$ we will need to modify the bundle and flag so that the dual weight $-c_p$ lies in $(-c_p, -c_p+1]$:
\begin{equation*}
    \alpha_{n_p} = c_p \xrightarrow{\text{dualize}} -c_p \xrightarrow{\text{shift weight}} -c_p + 1
\end{equation*}
Such a shift can be conveniently phrased in the language of Hecke modifications.

\subsection{Hecke modifications and parabolic duality} 
\label{sub:hecke-modifications}

We begin with a very brief overview of Hecke modifications in general; see \cite{Kapustin:2007} for more detail.
Our main focus will be on their interplay with the parabolic structures described above.

\begin{construction}[Hecke modification] \label{constr:hecke-mod}
    Let $E$ be a rank $r$ holomorphic vector bundle over $C$, and fix a point $p \in C$.
    A \emph{Hecke modification} of $E$ at $p$ of type $(n_1, \dots, n_r) \in \Z^r$ is a new bundle $\tilde{E}$, obtained by the following procedure:
    \begin{enumerate}
        \item Fix a local coordinate $w$ centred at $p$ and choose a trivialization of $E$ over a sufficiently small disc $\Delta_p$ around $p$.

        \item Let $\tilde{E}$ be the new bundle obtained by regluing $E|_{\Delta_p}$ and $E|_{C \setminus \{p\}}$ using the transition map
        \begin{equation}
            \diag(w^{-n_1}, \dots, w^{-n_r})
        \end{equation}
        over the punctured disc $\Delta_p^\times$.
    \end{enumerate}
\end{construction}

Such a modification of $E$ changes the degree of the bundle to 
\begin{equation}
    \deg \tilde{E} = \deg E + \sum_{i=1}^r n_i.
\end{equation}

Note that Hecke modifications of a given type $(n_1, \dots, n_r)$ are not unique: the above construction depends on the choice of trivialization. 
However, our situation is more rigid, since we will only be interested in modifications that are compatible with the parabolic structure.

\begin{definition}[Parabolic Hecke modification]
    If $E$ is a parabolic bundle and $p \in D$, we will say that a Hecke modification at $p$ is \emph{parabolic} if the sections used to trivialize $E$ over $\Delta_p$ are compatible with the parabolic structure at $p$.
\end{definition}

Such a modification naturally shifts the parabolic weights associated with the trivializing sections, and in this way defines a new flag at $p$. 
In particular, if $E$ is the $\mathbf{c}$-truncation of a filtered bundle $\mathcal{P}_* E$, then the bundle $\tilde{E}$ obtained by a parabolic Hecke modification of type $\pm(1, \dots, 1)$ is the $\mathbf{c}'$-truncation of $\mathcal{P}_* E$, where $c_p' = c_p \pm 1$.\footnote{In this case the Hecke modification construction is essentially a reformulation of properties \ref{item:filtered-def-1} and \ref{item:filtered-def-4} in  \cref{def:filtered-bundles}.}
More generally one can shift some subset of the parabolic weights in order to produce a new parabolic bundle, as long as all of the new weights lie in suitable half-open intervals.
This leads to a procedure for directly obtaining the dual of a parabolic bundle.

\begin{construction}[Dual parabolic bundle, redux] \label{constr:dual-parabolic-bundle}
    If $E$ is a $\mathbf{c}$-parabolic bundle with flags
    \begin{alignat*}{6}
        0 = E_{p,0} &  &&\subset E_{p,1} &&\subset E_{p,2} &&\subset \cdots &&\subset E_{p,n_p} &&= E|_p \\
        c_p - 1 & &&< \alpha_{p,1} &&< \alpha_{p,2} &&< \cdots &&< \alpha_{p,n_p} &&\leq c_p,
    \end{alignat*}
    then the dual parabolic bundle $E^*$ is obtained by taking the dual flags
        \begin{alignat*}{6}
            0 = (E_{p,n_p})^0 &  &&\subset (E_{p,n_p-1})^0 &&\subset (E_{p,n_p - 2})^0 &&\subset \cdots &&\subset (E_{p,0})^0 &&= (E|_p)^* \\
            -c_{p} & &&\leq  \hspace{8pt} -\alpha_{p,n_p} &&< -\alpha_{p,n_p - 1} &&< \cdots &&< -\alpha_{p,1} &&< -c_p + 1
        \end{alignat*}
        and performing a parabolic Hecke modification of type $\smash{(\underbrace{1, \dots, 1}_{m_{p, n_p}}, 0, \dots, 0)}$ at each point $p\in D$ such that $\alpha_{p, n+p} = c_p$.
\end{construction}

\subsection{Self-dual frames for the Hitchin section} 
\label{sub:self-dual-frames}

We can construct a frame for the Hitchin section $\Bfr \subseteq \Xfr$ (\cref{def:framed-hitchin-section}) by following the $m^{(3)} = \frac{1}{2}$ cases of \cref{constr:bundle-from-parameters,constr:frame-from-eigenframe}.\footnote{The notation in \cref{def:framed-hitchin-section} is slightly different, with $E = K_C^{-1/2} \oplus K_C^{1/2}$ and $\theta = \begin{pmatrix}
            0 & 1 \\ 
            (z^2 + 2m) dz^2 & 0
        \end{pmatrix}$.
To be consistent with the notation below, we will fix a trivialization of $E$ using the sections $dz^{\pm 1/2}$.}
This procedure involves a choice of eigenframe $(\eta_1, \eta_2)$, which we will explicitly specify below.

The eigensections of
\begin{equation*}
    \theta = \begin{pmatrix}
            0 & 1 \\ 
            z^2 + 2m & 0
            \end{pmatrix}  dz
\end{equation*}
are of the form
\begin{equation*}
    \eta_1 = \lambda_1 \begin{pmatrix}
        1 \\
        \sqrt{P}
    \end{pmatrix} \quad \text{and} \quad 
    \eta_2 = \lambda_2 \begin{pmatrix}
        -1 \\
        \sqrt{P}
    \end{pmatrix},
\end{equation*}
where 
\begin{equation*}
    P \coloneqq z^2 + 2m.
\end{equation*}
The requirement $\eta_1 \wedge \eta_2 = z e_1 \wedge e_2$ in \eqref{eq:eigenframe-wedge} is equivalent to 
\begin{equation*}
    \lambda_1 \lambda_2 = \frac{z}{2 \sqrt{P}},
\end{equation*}
which we can satisfy by taking
\begin{align*}
    \lambda_1 = \lambda_2 &= \sqrt{\frac{z}{2 \sqrt{P}}} \\
    &= \frac{1}{\sqrt{2}} + \O(w^2).
\end{align*}
Using these choices of normalization, \cref{constr:bundle-from-parameters} defines a parabolic bundle in which both eigensections $\eta_i$ have weight $\frac{1}{2}$, as well as a harmonic metric $h$. 
Then, \cref{constr:frame-from-eigenframe} produces a compatible frame $g$ by orthonormalizing $(\eta_1, w \eta_2)$ with respect to $h$.

More generally we can consider the $U(1)$-family of frames $e^{i \vartheta} \cdot g$, which equivalently could have been obtained by starting with the eigenframe $e^{i \vartheta} \cdot (\eta_1, \eta_2)$, i.e.\ choosing the normalizations
\begin{equation} \label{eq:eigenframe-coefficients}
    \lambda_1 = e^{i \vartheta} \sqrt{\frac{z}{2 \sqrt{P}}} \quad \text{and} \quad \lambda_2 = e^{-i \vartheta} \sqrt{\frac{z}{2 \sqrt{P}}}.
\end{equation}

\begin{definition}[Hitchin section frame]\label{def:hitchin-section-frame}
    We will choose the frame $g_0 \coloneqq e^{i \pi/4} \cdot g$ for our construction of the framed Hitchin section.
\end{definition}

We will prove that the magnetic angles $\theta_m$ and $\theta_m^\shift$ vanish for this choice of frame by showing that it is ``self-dual'' in the sense discussed in \cref{sub:framed-hitchin-section}.
More specifically, we will need to know how the frame $e^{i \vartheta} \cdot g$ transforms under duality and the isomorphism
\begin{equation*}
    S = \begin{pmatrix}
        0 & i \\ 
        i & 0
    \end{pmatrix}: E^* \xrightarrow{\sim} E.
\end{equation*}

\begin{lemma}[Dual frame calculation] \label{lem:dual-frame-calculation}
    The frame $e^{i \vartheta} \cdot g = (e_1, e_2)$ satisfies
    \begin{equation}
        S (e^{i \vartheta} \cdot g)^* = \left( ie^{-2i \vartheta} \frac{w}{|w|}  e_1 , -ie^{2 i \vartheta} \frac{|w|}{w} e_2 \right).
    \end{equation}
    In particular, for $\vartheta = \frac{\pi}{2}$, the Hitchin section frame $g_0$ satisfies
    \begin{equation} \label{eq:swapped-dual-frame}
        S (g_0)^* = \left(\frac{w}{|w|}  e_1 , \frac{|w|}{w} e_2 \right),
    \end{equation}
    i.e.\ it is self-dual under the isomorphism $S$ up to a unitary Hecke modification of type $1$ (see \cref{def:unitary-hecke-mod}).
\end{lemma}

\begin{proof}
    It suffices to work with the ``approximately orthogonal'' frame 
\begin{equation} \label{eq:approx-orthogonal-frame}
    \tilde{g} \coloneqq 
    \left(\frac{\eta_1}{|w|^{-1/2}}, \frac{w}{|w|} \frac{\eta_2}{|w|^{-1/2}} \right),
\end{equation}
instead of $g$, since the other Gram-Schmidt terms vanish exponentially as $w \to 0$.
This reduces the claim to a straightforward linear algebra computation.

In general, the coefficients of
\begin{equation*}
    \left( \tilde{\lambda}_1 \begin{pmatrix}
        1 \\ \sqrt{P}
    \end{pmatrix}, \ 
    \tilde{\lambda}_2 \begin{pmatrix}
        -1 \\ \sqrt{P}
    \end{pmatrix}  \right)^*
\end{equation*}
are given by the columns of the inverse transpose matrix, namely
\begin{equation*}
    \left( \frac{1}{\tilde{\lambda}_1 \cdot 2 \sqrt{P}} \begin{pmatrix}
        \sqrt{P} \\ 1
    \end{pmatrix}, \
    \frac{1}{\tilde{\lambda}_2 \cdot 2 \sqrt{P}} \begin{pmatrix}
        -\sqrt{P} \\ 1
    \end{pmatrix} \right).
\end{equation*}
It follows that
\begin{equation*}
    S \left( \tilde{\lambda}_1 \begin{pmatrix}
        1 \\ \sqrt{P}
    \end{pmatrix}, \ 
    \tilde{\lambda}_2 \begin{pmatrix}
        -1 \\ \sqrt{P}
    \end{pmatrix}  \right)^* 
    = 
    \left( \frac{i}{\tilde{\lambda}_1 \cdot 2 \sqrt{P}} \begin{pmatrix}
        1  \\ \sqrt{P}
    \end{pmatrix}, \
    \frac{-i}{\tilde{\lambda}_2 \cdot 2 \sqrt{P}} \begin{pmatrix}
        -1 \\ \sqrt{P}
    \end{pmatrix} \right).
\end{equation*}

Choosing
\begin{equation*}
    \tilde{\lambda}_1 = \frac{\lambda_1}{|w|^{-1/2}} \quad \text{and} \quad \tilde{\lambda}_2 = \frac{w \lambda_2}{|w|^{1/2}}
\end{equation*}
to match with $\tilde{g}$ in \eqref{eq:approx-orthogonal-frame} and using the normalization \eqref{eq:eigenframe-coefficients} for $\lambda_i$, we see that the original components of $\tilde{g}$ are respectively multiplied by
\begin{equation*}
     \frac{i}{\tilde{\lambda}_1^2 \cdot 2 \sqrt{P}} = i e^{-2 i \vartheta} \frac{w}{|w|}
     \quad \text{and} \quad
     \frac{-i}{\tilde{\lambda}_2^2 \cdot 2 \sqrt{P}} = -i e^{2 i \vartheta} \frac{|w|}{w},
\end{equation*}
as required.
\end{proof}


\begin{thebibliography}{MSWW19}

\bibitem[AB83]{Atiyah:1983}
M.~F. Atiyah and R.~Bott, \emph{The {Y}ang-{M}ills equations over {R}iemann
  surfaces}, Philos. Trans. Roy. Soc. A \textbf{308} (1983), no.~1505,
  523--615.

\bibitem[ANXZ24]{Alekseev:2024}
A.~Alekseev, A.~Neitzke, X.~Xu, and Y.~Zhou, \emph{W{KB} {A}symptotics of
  {S}tokes {M}atrices, {S}pectral {C}urves and {R}hombus {I}nequalities}, Comm.
  Math. Phys. \textbf{405} (2024), no.~11, Paper No. 269.

\bibitem[BB04]{Biquard:2004}
O.~Biquard and P.~Boalch, \emph{Wild non-abelian {H}odge theory on curves},
  Compos. Math. \textbf{140} (2004), no.~1, 179--204.

\bibitem[Boa01]{Boalch:2001}
P.~Boalch, \emph{Symplectic manifolds and isomonodromic deformations}, Adv.
  Math. \textbf{163} (2001), no.~2, 137--205.

\bibitem[Boa02]{Boalch:2002}
P.~P. Boalch, \emph{{$G$}-bundles, isomonodromy, and quantum {W}eyl groups},
  Int. Math. Res. Not. (2002), no.~22, 1129--1166.

\bibitem[Boa24]{Boalch:2024}
P.~Boalch, \emph{Counting the fission trees and nonabelian {Hodge} graphs},
  2024, \href{http://arxiv.org/abs/2410.23358}{{\ttfamily arXiv:2410.23358}}.

\bibitem[DN19]{Dumas:2019}
D.~Dumas and A.~Neitzke, \emph{Asymptotics of {Hitchin's} metric on the
  {Hitchin} section}, Comm. Math. Phys. \textbf{367} (2019), no.~1, 127--150.

\bibitem[DN20]{Dumas:2020}
D.~Dumas and A.~Neitzke, \emph{Opers and nonabelian {Hodge}: numerical
  studies}, 2020, \href{http://arxiv.org/abs/2007.00503}{{\ttfamily
  arXiv:2007.00503}}.

\bibitem[FG06]{Fock:2006}
V.~Fock and A.~Goncharov, \emph{Moduli spaces of local systems and higher
  {T}eichm{\"u}ller theory}, Publ. Math. Inst. Hautes \'{E}tudes Sci.
  \textbf{103} (2006), no.~1, 1--211.

\bibitem[FMSW22]{Fredrickson:2022}
L.~Fredrickson, R.~Mazzeo, J.~Swoboda, and H.~Weiss, \emph{Asymptotic geometry
  of the moduli space of parabolic {${\rm SL}(2,\mathbb C)$}-{H}iggs bundles},
  J. Lond. Math. Soc. (2) \textbf{106} (2022), no.~2, 590--661.

\bibitem[Fre99]{Freed:1999a}
D.~S. Freed, \emph{Special {K}\"{a}hler manifolds}, Comm. Math. Phys.
  \textbf{203} (1999), no.~1, 31--52.

\bibitem[Fre20]{Fredrickson:2020}
L.~Fredrickson, \emph{Exponential decay for the asymptotic geometry of the
  {H}itchin metric}, Comm. Math. Phys. \textbf{375} (2020), no.~2, 1393--1426.

\bibitem[GMN10]{Gaiotto:2010}
D.~Gaiotto, G.~W. Moore, and A.~Neitzke, \emph{Four-dimensional wall-crossing
  via three-dimensional field theory}, Comm. Math. Phys. \textbf{299} (2010),
  no.~1, 163--224.

\bibitem[GMN13a]{Gaiotto:2013}
D.~Gaiotto, G.~W. Moore, and A.~Neitzke, \emph{Spectral networks}, Ann. Henri
  Poincar\'{e} \textbf{14} (2013), no.~7, 1643--1731.

\bibitem[GMN13b]{Gaiotto:2013a}
D.~Gaiotto, G.~W. Moore, and A.~Neitzke, \emph{Wall-crossing, {Hitchin}
  systems, and the {WKB} approximation}, Adv. Math. \textbf{234} (2013),
  239--403.

\bibitem[GMN14]{Gaiotto:2014a}
D.~Gaiotto, G.~W. Moore, and A.~Neitzke, \emph{Spectral networks and snakes},
  Ann. Henri Poincar\'{e} \textbf{15} (2014), no.~1, 61--141.

\bibitem[GW00]{Gross:2000}
M.~Gross and P.~M.~H. Wilson, \emph{Large complex structure limits of {$K3$}
  surfaces}, J. Differential Geom. \textbf{55} (2000), no.~3, 475--546.

\bibitem[HHL25]{He:2025}
S.~He, J.~Horn, and N.~Li, \emph{The asymptotics of the
  $\mathrm{SL}_2(\mathbb{C})$-Hitchin metric on the singular locus:
  subintegrable systems}, 2025,
  \href{http://arxiv.org/abs/2506.04957}{{\ttfamily arXiv:2506.04957}}.

\bibitem[Hit87]{Hitchin:1987}
N.~J. Hitchin, \emph{The self-duality equations on a {R}iemann surface}, Proc.
  London Math. Soc. (3) \textbf{55} (1987), no.~1, 59--126.

\bibitem[Hit92a]{Hitchin:1992b}
N.~Hitchin, \emph{Hyperk{\"a}hler manifolds}, no. 206, 1992, S\'{e}minaire
  Bourbaki, Vol. 1991/92, pp.~Exp. No. 748, 3, 137--166.

\bibitem[Hit92b]{Hitchin:1992a}
N.~Hitchin, \emph{Lie groups and {T}eichm{\"u}ller space}, Topology \textbf{31}
  (1992), no.~3, 449 -- 473.

\bibitem[HKLR87]{Hitchin:1987b}
N.~J. Hitchin, A.~Karlhede, U.~Lindstr{\"o}m, and M.~Ro{\v{c}}ek,
  \emph{Hyperk{\"a}hler metrics and supersymmetry}, Comm. Math. Phys.
  \textbf{108} (1987), no.~4, 535--589.

\bibitem[HN16]{Hollands:2016}
L.~Hollands and A.~Neitzke, \emph{Spectral networks and {Fenchel--Nielsen}
  coordinates}, Lett. Math. Phys. \textbf{106} (2016), no.~6, 811--877.

\bibitem[HRS21]{Hollands:2021}
L.~Hollands, P.~R{\"u}ter, and R.~J. Szabo, \emph{A geometric recipe for
  twisted superpotentials}, J. High Energy Phys. \textbf{2021} (2021), no.~12.

\bibitem[Jef94]{Jeffrey:1994}
L.~C. Jeffrey, \emph{Extended moduli spaces of flat connections on {R}iemann
  surfaces}, Math. Ann. \textbf{298} (1994), no.~4, 667--692.

\bibitem[KW07]{Kapustin:2007}
A.~Kapustin and E.~Witten, \emph{Electric-magnetic duality and the geometric
  {L}anglands program}, Commun. Number Theory Phys. \textbf{1} (2007), no.~1,
  1--236.

\bibitem[Moc11]{Mochizuki:2011}
T.~Mochizuki, \emph{Wild harmonic bundles and wild pure twistor {$D$}-modules},
  Ast\'{e}risque (2011), no.~340, x+607.

\bibitem[Moc22]{Mochizuki:0}
T.~Mochizuki, \emph{Periodic monopoles and difference modules}, Lecture Notes
  in Mathematics, vol. 2300, Springer, Cham, 2022.

\bibitem[MSWW19]{Mazzeo:2019}
R.~Mazzeo, J.~Swoboda, H.~Weiss, and F.~Witt, \emph{Asymptotic geometry of the
  {Hitchin} metric}, Comm. Math. Phys. \textbf{367} (2019), no.~1, 151--191.

\bibitem[Nei]{Neitzke:swn}
A.~Neitzke, \emph{swn-plotter},
  \url{https://gauss.math.yale.edu/~an592/mathematica/swn-plotter.nb}.

\bibitem[Nei24]{Neitzke:2024}
A.~Neitzke, \emph{Integral iterations for harmonic maps}, Beijing J. of Pure
  and Appl. Math. \textbf{1} (2024), no.~1, 121--158.

\bibitem[OV96]{Ooguri:1996}
H.~Ooguri and C.~Vafa, \emph{Summing up {D}irichlet instantons}, Phys. Rev.
  Lett. \textbf{77} (1996), no.~16, 3296--3298.

\bibitem[Sim90]{Simpson:1990}
C.~T. Simpson, \emph{Harmonic bundles on noncompact curves}, J. Amer. Math.
  Soc. \textbf{3} (1990), no.~3, 713--770.

\bibitem[Sim92]{Simpson:1992}
C.~Simpson, \emph{Higgs bundles and local systems}, Inst. Hautes \'{E}tudes
  Sci. Publ. Math. \textbf{75} (1992), 5--95.

\bibitem[Tul19]{Tulli:2019}
I.~Tulli, \emph{The {Ooguri-Vafa} space as a moduli space of framed wild
  harmonic bundles}, 2019, \href{http://arxiv.org/abs/1912.00261v1}{{\ttfamily
  arXiv:1912.00261v1}}.

\bibitem[Wit08]{Witten:2008}
E.~Witten, \emph{Gauge theory and wild ramification}, Anal. Appl. (Singap.)
  \textbf{6} (2008), no.~4, 429--501.

\end{thebibliography}
\end{document}